\documentclass[reqno]{amsart}
\usepackage{enumitem}
\usepackage{amssymb, amsmath,latexsym,amsfonts,amsbsy, amsthm,mathtools,graphicx,color}
\usepackage{float}
\usepackage{tikz}
\usetikzlibrary{arrows, decorations.markings,matrix,trees}
\usepackage{hyperref}
\def\Xint#1{\mathchoice
  {\XXint\displaystyle\textstyle{#1}}%
  {\XXint\textstyle\scriptstyle{#1}}%
  {\XXint\scriptstyle\scriptscriptstyle{#1}}%
  {\XXint\scriptscriptstyle\scriptscriptstyle{#1}}%
  \!\int}
\def\XXint#1#2#3{{\setbox0=\hbox{$#1{#2#3}{\int}$}
  \vcenter{\hbox{$#2#3$}}\kern-.5\wd0}}

\def\dashint{\Xint-}
\newcommand{\al}{\alpha}       
\newcommand{\vt}{\vartheta}    
\newcommand{\lda}{\lambda}
\newcommand{\om}{\Omega}            
\newcommand{\pa}{\partial}
\newcommand{\va}{\varepsilon}       
\newcommand{\ud}{\mathrm{d}}
\newcommand{\be}{\begin{equation}} 
\newcommand{\ee}{\end{equation}}
\newcommand{\w}{\omega}      
\newcommand{\Lda}{\Lambda}

\newcommand{\bA}{\mathbb{A}}

\newcommand{\cB}{\mathcal{B}}

\newcommand{\cC}{\mathcal{C}}

\newcommand{\cD}{\mathcal{D}}

\newcommand{\cF}{\mathcal{F}}

\newcommand{\bG}{\mathbb{G}}

\newcommand{\cG}{\mathcal{G}}

\newcommand{\cL}{\mathcal{L}} 

\newcommand{\Z}{\mathbb{Z}}

\newcommand{\M}{\mathcal{M}}

\newcommand{\cN}{\mathcal{N}}

\newcommand{\R}{\mathbb{R}}   

\newcommand{\cR}{\mathcal{R}}

\newcommand{\wc}{\rightharpoonup}        

\newcommand{\HH}{\mathcal{H}}

\newcommand{\vp}{\varphi}

\newcommand{\ga}{\gamma}

\newcommand{\ift}{\infty} 
\newcommand{\wt}{\widetilde}
\newcommand{\wh}{\widehat}
\newcommand{\f}{\frac}

\newcommand{\ol}{\overline}

\newcommand{\op}{\operatorname}
\newcommand{\Sg}{\Sigma}

\newcommand{\na}{\nabla}

\DeclareMathOperator{\dist}{dist}

\DeclareMathOperator{\diam}{diam}

\DeclareMathOperator{\supp}{supp}

\DeclareMathOperator{\sing}{sing}

\DeclareMathOperator{\loc}{loc}

\DeclareMathOperator{\CN}{CN}
\DeclareMathOperator{\HSV}{HSV}

\def\<{\langle}\def\>{\rangle}
\def\({\left(}\def\){\right)}
\def\[{\left[}\def\]{\right]}
\numberwithin{equation}{section}
\theoremstyle{plain}
\newtheorem{thm}{Theorem}[section]
\newtheorem{cor}[thm]{Corollary}

\newtheorem{lem}[thm]{Lemma}
\newtheorem{prop}[thm]{Proposition}

\theoremstyle{definition}
\newtheorem{defn}[thm]{Definition}

\newtheorem{rem}[thm]{Remark}

\title[supercritical semilinear elliptic equations]{Quantitative stratification and sharp regularity estimates for supercritical semilinear elliptic equations}

\author{Haotong Fu}
\address{School of Mathematical Sciences, Peking University, Beijing 100871, China}
\email{547434974@qq.com}

\author{Wei Wang}
\address{School of Mathematical Sciences, Peking University, Beijing 100871, China}
\email{2201110024@stu.pku.edu.cn}

\author{Zhifei Zhang}
\address{School of Mathematical Sciences, Peking University, Beijing 100871, China}
\email{zfzhang@math.pku.edu.cn}

\date{\today}

\begin{document}

\begin{abstract}
In this paper, we investigate the interior regularity theory for stationary solutions of the supercritical semilinear elliptic equation
$$
-\Delta u=|u|^{p-1}u\quad\text{in }\om,\quad p>\f{n+2}{n-2},
$$
where $ \om\subset\R^n $ is a bounded domain with $ n\geq 3 $. Our primary focus is on structures of the stratification for the singular sets. We define the $ k $-th stratification $ S^k(u) $ of $ u $ based on the tangent functions and tangent measures. We show that $ S^k(u) $ is $ k $-rectifiable and establish estimates for volumes associated with points, which have lower bounds on the regular scales. These estimates enable us to derive sharp interior estimates for the solutions. Specifically, if $ \al_p=\f{2(p+1)}{p-1} $ is not an integer, then for any $ j\in\Z_{\geq 0} $, we have
$$
D^ju\in L_{\loc}^{q_j,\ift}(\om).
$$
Precisely, for any $ \om'\subset\subset\om $, 
$$
\sup\{\lda>0:\lda^{q_j}\cL^n(\{x\in\om':|D^ju(x)|>\lda\})\}<+\ift, 
$$
where $ \cL^n(\cdot) $ is the $ n $-dimensional Lebesgue measure, and
$$
q_j=\f{(p-1)(\lfloor\al_p\rfloor+1)}{2+j(p-1)},
$$
with $ \lfloor\al_p\rfloor $ being the integer part of $ \al_p $. The proofs of these results rely on Reifenberg-type theorems developed by A. Naber and D. Valtorta to study the stratification of harmonic maps. 
\end{abstract}

\maketitle

\tableofcontents

\section{Introduction}

\subsection{Problem setting and main results} In this paper, we study the semilinear elliptic equation
\be
-\Delta u=|u|^{p-1}u\quad\text{in }\om,\label{superequation}
\ee
where unless otherwise specified, $ \om\subset\R^n $ is assumed to be a bounded domain, $ n\in\Z_{\geq 3} $, and $ p>\f{n+2}{n-2} $ is a supercritical exponent.

We mainly consider the regularity theory concerning some suitable solutions to this equation. In particular, we shall focus on the stationary solution defined below.

\begin{defn}[Stationary solutions]
A measurable function $ u:\om\to\R $ is called a stationary solution of the equation \eqref{superequation} if $ u\in(H^1\cap L^{p+1})(\om) $, and exhibits the following two properties.
\begin{enumerate}[label=$(\theenumi)$]
\item $ u $ is a weak solution of \eqref{superequation}, meaning
\be
\int_{\om}(\na u\cdot\na\vp-|u|^{p-1}u\vp)=0,\label{WeakCon}
\ee
for any $ \vp\in C_0^{\ift}(\om) $.
\item $ u $ satisfies the stationary condition, namely
\be
\int_{\om}\left[\(\f{|\na u|^2}{2}-\f{|u|^{p+1}}{p+1}\)\op{div}Y-DY(\na u,\na u)\right]=0,\label{StaCon}
\ee
for any vector field $ Y\in C_0^{\ift}(\om,\R^n) $.
\end{enumerate}
\end{defn}

Here, \eqref{WeakCon} and \eqref{StaCon} arise naturally from the variational conditions. To clarify this point, we define the corresponding energy functional
\be
\cF(u,\om):=\int_{\om}\(\f{|\na u|^2}{2}-\f{|u|^{p+1}}{p+1}\).\label{FunctEll}
\ee
The function $ u $ satisfies the formula \eqref{WeakCon}, if and only if
$$
\left.\f{\ud}{\ud t}\right|_{t=0}\cF(u+t\vp,\om)=0,
$$
for any $ \vp\in C_0^{\ift}(\om) $. Additionally, the stationary condition \eqref{StaCon} can be equivalently expressed as the property that $ u $ is a critical point of $ \cF(\cdot,\om) $ under inner perturbations of $ \om $. Precisely
$$
\left.\f{\ud}{\ud t}\right|_{t=0}\cF(u(\cdot+tY(\cdot)),\om)=0,
$$
for any $ Y\in C_0^{\ift}(\om,\R^n) $. Similar stationary conditions play significant roles in analyzing diverse geometric variational problems, such as harmonic maps, Yang-Mills connections, and varifolds. Interested readers can refer to the relevant literatures \cite{All72,Bet93,Eva90,Lin99,Pri83,Sim83,Tia00,TT04}. Stationary conditions, in general, ensure the validity of Pohozaev-type identities, indicating the monotonicity of specific energy densities. These structural characteristics allow for the application of techniques from geometric measure theory.

For stationary solutions of equation \eqref{superequation}, partial regularity and convergence results were initially established in seminal works \cite{Pac93,Pac94}. Recently, in \cite{WW15}, the authors conducted a blow-up analysis using tools developed by \cite{Lin99} in the context of harmonic maps (also refer to \cite{LW99,LW02a,LW02b} for similar arguments in the study of Ginzburg-Landau models and corresponding dynamic problems). Moreover, a rich body of literature explores various problem settings related to equation \eqref{superequation}, including stable solutions. Specifically, a solution $ u $ is called stable if it satisfies \eqref{WeakCon} and the second variation of the functional \eqref{FunctEll} is nonnegative. To be precise,
$$
\left.\f{\ud^2}{\ud t^2}\right|_{t=0}\cF(u+t\vp,\om)=\int_{\om}(|\na\vp|^2-p|u|^{p-1}\vp)\geq 0,
$$
for any $ \vp\in C_0^{\ift}(\om) $. The article \cite{Far07} provided a priori estimates for stable solutions defined as above, and in the case when the solution $ u $ is positive, the subsequent study \cite{Wan12} established improved estimates for the Hausdorff dimensions of singular sets. Notably, the results of these two papers demonstrate that if $ u $ is a positive stable solution of \eqref{superequation}, then $ u\in H_{\loc}^2 $ and $ u^p\in L_{\loc}^2 $. Consequently, for any $ Y\in C_0^{\ift}(\om,\R^n) $, we can apply the condition \eqref{WeakCon} with $ \vp=Y\cdot\na u $ and obtain \eqref{StaCon} through integration by parts, indicating that $ u $ is a stationary solution. Hence, the results presented in our paper apply to positive stable solutions of \eqref{superequation}. Furthermore, the study of stable solutions for more general semilinear elliptic equations of the form $ -\Delta u=f(u) $, with some appropriate function $ f $, remains a significant and active area of research. One of the most remarkable recent breakthroughs is the work \cite{CFRS20}, which showed that if $ n\leq 9 $, $ f $ is locally Lipschitz and nonnegative, then the stable solution $ u $ belongs to the H\"{o}lder space $ C^{0,\al} $ for some dimensional constant $ \al\in(0,1) $. For a comprehensive overview of stable solutions of semilinear elliptic equations, readers are referred to the monograph \cite{Dup11} and the survey paper \cite{Cab17}.

Concerning the supercritical equation \eqref{superequation}, if $ u\in(H^1\cap L^{p+1})(\om) $ is a stationary solution, it is plausible that $ u $ may lack global regularity in $ \om $. However, for subcritical ($ 1<p<\f{n+2}{n-2} $) and critical ($ p=\f{n+2}{n-2} $) cases, weak solutions defined by \eqref{WeakCon} are known to be smooth, as evidenced by the application of Moser's iteration techniques (see, for instance, Lemma B.3 in \cite{Str00}). For the supercritical case considered in this paper, to measure the singularity of stationary solutions, we define the singular set of $ u $ as the complement of the regular set. In particular,
\begin{align*}
\sing(u):&=\{x\in\om:u\text{ is continuous in }B_r(x)\text{ for some }r>0\}^{c},\\
&=\{x\in\om:u\text{ is smooth in }B_r(x)\text{ for some }r>0\}^{c}.
\end{align*}

To study the microscopic behavior of the solution $ u $ at a given point $ x\in\om $, it is necessary to conduct a blow-up analysis. The concept of blow-ups for functions can be formally defined as follows.

\begin{defn}[Blow-ups of functions]
Let $ u:\om\to\R $ be a measurable function, $ x\in\om $, and $ 0<r<\dist(x,\pa\om) $. We define the blow-up of $ u $ (centered at $ x $, with scale $ r $) as the function $ T_{x,r}u $, given by
$$
T_{x,r}u(y):=r^{\f{2}{p-1}}u(x+ry),
$$
which makes sense in $ r^{-1}(\om-x) $.
\end{defn}

These blow-ups are connected to the elliptic equation $ -\Delta u=|u|^{p-1}u $. If $ u $ is a stationary solution of \eqref{superequation}, then $ T_{x,r}u $ is also a stationary solution. Before introducing blow-ups for Radon measures, certain settings and conventions must be established. Let $ U\subset\R^n $ be a domain, and we denote by $ \M(U) $ the set of nonnegative Radon measures on $ U $. Given a sequence $ \{\mu_i\}\subset\M(U) $ and $ \mu\in\M(U) $, we call $ \mu_i\wc^*\mu $ in $ \M(U) $ if $ \mu_i\to\mu $ in the sense of Radon measures. Precisely,
\be
\lim_{i\to+\ift}\left|\int_Uf\ud\mu_i-\int_Uf\ud\mu\right|=0,\label{Radonsense}
\ee
for any $ f\in C_0(U) $. Furthermore, for $ d\in\Z_{\geq 2} $, we use $ \M(U,\R^d) $ to denote the set of vector-valued Radon measures on $ U $. Specifically, for any $ \mu\in M(U,\R^d) $, we can express $ \mu=(\mu^j)_{j=1}^d $, where $ \mu^j $ is a Radon measure (not necessarily nonnegative), for any $ j\in\{1,2,...,d\} $. Let $ \{\mu_i\}\subset\M(U,\R^d) $ and $ \mu\in\M(U,\R^d) $, we say that $ \mu_i\wc^*\mu $, if for any $ j\in\{1,2,...,d\} $, $ \mu_i^j\to\mu^j $ in the sense of \eqref{Radonsense}.

In some situations, we will translate solutions of \eqref{superequation} into Radon measures and utilize results from geometric measure theory. The following definition serves to facilitate this process.

\begin{defn}[Induced measures]
Let $ u\in(H^1\cap L^{p+1})(\om) $. The Radon measure induced by $ u $ on $ \om $ is defined as
$$
m_u:=\(\f{p-1}{2}|\na u|^2+\f{p-1}{p+1}|u|^{p+1}\)\ud y.
$$
\end{defn}

Next, the blow-ups for Radon measures are introduced. In this context, a crucial index denoted by $ \al_p $ is defined as
$$
\al_p:=\f{2(p+1)}{p-1}\in(2,n),
$$
which is associated with the supercritical problem. 

\begin{defn}[Blow-ups of Radon measures]
Let $ \mu\in\M(\om) $, $ x\in\om $, and $ 0<r<\dist(x,\pa\om) $. We define the blow-up of $ \mu $ (centered at $ x $, with scale $ r $) as the measure $ T_{x,r}\mu\in\M(r^{-1}(\om-x)) $ such that
$$
T_{x,r}\mu(A):=r^{\al_p-n}\mu(x+rA),
$$
for any $ \mu $-measurable set $ A\subset r^{-1}(\om-x) $.
\end{defn}

Here we see that if $ x\in\om $, $ 0<r<\dist(x,\pa\om) $, and $ u\in(H^1\cap L^{p+1})(\om) $, then $ T_{x,r}m_u=m_{T_{x,r}u} $. It is natural to explore the convergence of blow-ups for stationary solutions and the associated measures they induce.

\begin{defn}[Tangent functions and tangent measures]
Let $ x\in\om $. Assume that $ u\in(H^1\cap L^{p+1})(\om) $ is a stationary solution of \eqref{superequation}. We have the following definitions.
\begin{enumerate}[label=$(\theenumi)$]
\item If there exist a sequence $ \{r_i\} $, with $ r_i\to 0^+ $, and $ v\in(H_{\loc}^1\cap L_{\loc}^{p+1})(\R^n) $ such that $ T_{x,r_i}u\wc v $ weakly in $ (H_{\loc}^1\cap L_{\loc}^{p+1})(\R^n) $, then we call $ v $ as a tangent function of $ u $ at $ x $. We define $ \cF_x(u) $ as the set of all tangent functions of $ u $ at $ x $.
\item If there exist a sequence $ \{s_i\} $, with $ s_i\to 0^+ $, and $ \mu\in\M(\R^n) $ such that $ T_{x,s_i}m_u\wc^*\mu $ in $ \M(\R^n) $, then we call $ \mu $ as a tangent measure of $ u $ at $ x $. We let $ \M_x(u) $ be the set of all tangent measures of $ u $ at $ x $. 
\end{enumerate}
\end{defn}

The subsequent proposition establishes the existence of limits for blow-ups of stationary solutions and their induced measures.

\begin{prop}\label{tangentfm}
Assume that $ u\in (H^1\cap L^{p+1})(\om) $ is a stationary solution of \eqref{superequation}. For any $ x\in\om $, $ \cF_x(u)\neq\emptyset $ and $ \M_x(u)\neq\emptyset $. If $ x\notin\sing(u) $, namely, there exists $ r>0 $ such that $ u $ is smooth in $ B_r(x) $, then $ \cF_x(u) $ only contains the zero function and $ \M_x(u) $ only contains the zero measure.
\end{prop}

To fully understand the properties of stationary solutions, we introduce the concept of tangent pairs, which allows us to effectively combine the information from both tangent functions and tangent measures at a given point $ x\in\om $.

\begin{defn}[Tangent pairs]
Assume that $ u\in (H^1\cap L^{p+1})(\om) $ is a stationary solution of \eqref{superequation}. We say that $ (v,\mu)\in \cF_x(u)\times\M_x(u) $ is a tangent pair of $ u $ at $ x $, if there exists a sequence $ \{r_i\} $, with $ r_i\to 0^+ $ such that $ T_{x,r_i}u\wc v $ weakly in $ (H_{\loc}^1\cap L_{\loc}^{p+1})(\R^n) $ and $ T_{x,r_i}m_u\wc^*\mu $ in $ \M(\R^n) $.
\end{defn}

Given that the objects and estimates we will be considering are all in the local form, we can assume that $ \om=B_{40} $ and focus on the properties of stationary solutions within $ B_{10} $. Taking subsequence if necessary, Proposition \ref{tangentfm} indicates that for any stationary solution $ u\in(H^1 \cap L^{p+1})(B_{40}) $ of \eqref{superequation} and $ x\in B_{40} $, there exists a tangent pair associated with $ u $ at that point. The existence of the tangent pair enables us to classify points within $ B_{40} $, which we refer to as the concept of stratification. Before presenting further definitions and results about this, it is necessary to give some concepts related to functions and measures that possess special symmetry properties.

\begin{defn}[$ k $-symmetric functions]
Let $ k\in\{1,2,...,n\} $, $ r>0 $, and $ x\in\R^n $. A function $ u\in L^2(B_r(x)) $ (or $ u\in L_{\loc}^2(\R^n) $) is called as $ k $-symmetric at $ x $ in $ B_r(x) $ (or $ \R^n $), with respect to a $ k $-dimensional subspace $ V\subset\R^n $, if it satisfies the following two properties.
\begin{enumerate}[label=$(\theenumi)$]
\item $ u $ is $ -\f{2}{p-1} $-homogeneous at $ x $, i.e. for any $ 0<\lda<1 $, and a.e. $ y\in B_r $ (or $ y\in\R^n $), we have $ u(x+\lda y)=\lda^{-\f{2}{p-1}}u(x+y) $.
\item $ u $ is invariant with respect to $ V $, i.e. for any $ v\in V $ and a.e. $ y\in B_r(x) $ such that $ y+v\in B_r(x) $ (or a.e. $ y\in\R^n $), there holds $ u(y+v)=u(y) $. 
\end{enumerate}
For simplicity, if $ x=0 $, we call that $ u $ is $ k $-symmetric in $ B_r $ (or $ \R^n $), and if $ u\in L_{\loc}^2(\R^n) $, we call that $ u $ is $ k $-symmetric at $ x $. 
\end{defn}

\begin{defn}[$ k $-symmetric Radon measures]\label{kmeasure}
Let $ k\in\{1,2,...,n\} $, $ r>0 $, and $ x\in\R^n $. $ \mu\in\M(B_r(x)) $ (or $ \mu\in\M(\R^n) $) is called as $ k $-symmetric at $ x $ in $ B_r(x) $ (or $ \R^n $), with respect to a $ k $-dimensional subspace $ V\subset\R^n $, if it satisfies the following two properties.
\begin{enumerate}[label=$(\theenumi)$]
\item $ \mu $ is $ (n-\al_p) $-homogeneous at $ x $, i.e. for any $ 0<\lda<1 $, and $ \mu $-measurable set $ A\subset B_r(x) $ (or $ A\subset\R^n $), we have $ \mu(x+\lda(A-x))=\lda^{n-\al_p}\mu(A) $.
\item $ \mu $ is invariant with respect to $ V $, i.e. for any $ v\in V $ and $ \mu $-measurable set $ A\subset B_r(x) $ such that $ A+v\subset B_r(x) $ (or any $ \mu $-measurable set $ A\subset\R^n $), there holds $ \mu(v+A)=\mu(A) $. 
\end{enumerate}
For simplicity, if $ x=0 $, we call that $ \mu $ is $ k $-symmetric in $ B_r $ (or $ \R^n $) and if $ \mu\in\M(\R^n) $, we call that $ \mu $ is $ k $-symmetric at $ x $. 
\end{defn}

\begin{rem}\label{extend}
It is worth mentioning that $ n $-symmetric functions and measures are equivalent to zero functions and measures, respectively.

Given the homogeneous property, we can uniquely extend $ k $-symmetric functions and measures within balls to $ \R^n $. In some cases, if there is no ambiguity, we may use these two notions interchangeably.    
\end{rem} 

We can use the definitions of $ k $-symmetric functions and measures to define $ k $-symmetric pairs.

\begin{defn}\label{kpair}
Let $ k\in\{1,2,...,n\} $, $ r>0 $, and $ x\in\R^n $. The pair $ (v,\mu)\in L^2(B_r(x))\times\M(B_r(x)) $ (or $ (v,\mu)\in L_{\loc}^2(\R^n)\times\M(\R^n) $) is called as $ k $-symmetric at $ x $ in $ B_r(x) $ (or $ \R^n $), with respect to a $ k $-dimensional subspace $ V\subset\R^n $, if $ v $ and $ \mu $ are both $ k $-symmetric at $ x $ in $ B_r(x) $ (or $ \R^n $), with respect to $ V $.
\end{defn}

\begin{rem}
The essential point in the definition of $ k $-symmetric pair is that the function and measure in this pair have to be invariant with respect to the same $ k $-dimensional subspace.
\end{rem}

Given Proposition \ref{tangentfm}, for a stationary solution $ u\in(H^1\cap L^{p+1})(B_{40}) $, we define
$$
S^0(u)\subset S^1(u)\subset S^2(u)\subset...\subset S^{n-1}(u)\subset\sing(u),
$$
where for any $ k\in\{0,1,...,n-1\} $,
$$
S^k(u):=\{x\in B_{40}:\text{no tangent pair }(v,\mu)\text{ of }u\text{ at }x\text{ is }(k+1)\text{-symmetric}\}.
$$
A preliminary result on the stratification is on the estimate of the Hausdorff dimension.

\begin{thm}\label{Snminusalp}
Assume that $ u\in(H^1\cap L^{p+1})(B_{40}) $ is a stationary solution of \eqref{superequation}. Then 
$$
S^0(u)\subset S^1(u)\subset...\subset S^{n-\lceil\al_p\rceil}(u)=\sing(u),
$$
where $ \lceil a\rceil=\inf\{b\in\Z:b\geq a\} $, for any $ a\in\R $. Moreover, 
$$
\dim_{\HH}(S^k(u))\leq k,
$$
for any $ k\in\{0,1,2,...,n-\lceil\al_p\rceil\} $, where $ \dim_{\HH} $ denotes the Hausdorff dimension.
\end{thm}

The results in this theorem are similar to those discussed in the study of harmonic maps (see Proposition 10.34 of \cite{GM05} or Corollary 1.12 of \cite{Lin99}). By employing the methods outlined in these references, an alternative form of classification for $ \sing(u) $ can be established. In \S \ref{Preliminaries}, we will demonstrate the equivalence of these two different forms. Part I of \cite{WW21} presented similar properties for the heat equation associated with the static problem \eqref{superequation}. The proof of this theorem is straightforward, utilizing basic arguments similar to those in harmonic maps. Although sharper results on the stratification will be provided later, we are proving this result independently to introduce certain fundamental concepts as preparations for the subsequent proofs of our main theorems.

In \cite{CN13b,NV17}, a series of methods was developed by the authors to improve the understanding of the stratification for harmonic maps. Specifically, they explored the estimates of the Minkowski dimension and rectifiability for the $ k $-stratification. We also aim to apply these tools, known as quantitative stratification and Reifenberg-type theorems, to investigate $ S^k(u) $ in our problem settings. To carry out this procedure, the definitions of quantitative symmetry and stratification need to be established. Prior to that, some preparatory steps are required.

We define the metric on $ \M(B_1) $ to characterize the differences of Radon measures. Since the space $ C_0(B_1) $ is separable for the norm $ \|\cdot\|_{\ift} $, we can select a countable and dense subset of it, denoted by $ \{f_i\} $ such that $ \{f_i\}\cap C_0(B_{1-j^{-1}}) $ is dense in $ C_0(B_{1-j^{-1}}) $ for any $ j\in\Z_{\geq 2} $. We can further assume that for any $ j\in\Z_{\geq 3} $, there exists $ f_j\in\{f_i\} $ such that $ f_j\equiv 1 $ in $ B_{1-2j^{-1}} $, $ 0\leq f_j\leq 1 $ in $ B_1 $, and $ \supp f_j\subset B_{1-j^{-1}} $. Through this choice of $ \{f_i\} $, we can define a metric on $ \M(B_1) $ that aligns with the weak$^*$ topology induced by the convergence of Radon measures. Drawing inspiration from the constructions in \cite{CHN15,Del08}, we set
$$
d_{0,1}(\mu,\eta):=\sum_{i=1}^{+\ift}\f{1}{2^i}\cdot\f{\left|\int_{B_1}f_i\ud\mu-\int_{B_1}f_i\ud\eta\right|}{1+\left|\int_{B_1}f_i\ud\mu-\int_{B_1}f_i\ud\eta\right|},
$$
for any $ \mu,\eta\in\M(B_1) $. Letting $ \{\mu_i\}\subset\M(B_1) $ and $ \mu\in\M(B_1) $, we have 
$$
d_{0,1}(\mu_i,\mu)\to 0\text{ if and only if }\mu_i\wc^*\mu\text{ in }\M(B_1).
$$
For any $ r>0 $ and $ x\in\R^n $, $ d_{0,1}(\cdot,\cdot) $ induces a metric on $ \M(B_r(x)) $, denoted by $ d_{x,r}(\cdot,\cdot) $. Indeed, since $ \{f_i\} $ is dense in $ C_0(B_1) $, we can deduce that $ \{f_i^{x,r}\} $ is a dense subset of $ C_0(B_r(x)) $, where $ f_i^{x,r}(\cdot)=f_i(\f{\cdot-x}{r}) $ for any $ i\in\Z_+ $. Consequently, we define a metric on $ \M(B_r(x)) $ by
$$
d_{x,r}(\mu,\eta):=\sum_{i=1}^{+\ift}\f{1}{2^i}\cdot\f{r^{\al_p-n}\left|\int_{B_r(x)}f_i^{x,r}\ud\mu-\int_{B_r(x)}f_i^{x,r}\ud\eta\right|}{1+r^{\al_p-n}\left|\int_{B_r(x)}f_i^{x,r}\ud\mu-\int_{B_r(x)}f_i^{x,r}\ud\eta\right|},
$$
for any $ \mu,\eta\in\M(B_r(x)) $. Through the change of variables for integrals to Radon measures, we have the scaling property
\be
d_{0,1}(T_{x,r}\mu,T_{x,r}\eta)=d_{x,r}(\mu,\eta),\label{Scaled01}
\ee
where $ \mu,\eta\in\M(B_r(x)) $. 

Let $ r>0 $ and $ x\in\R^n $. Based on $ d_{x,r}(\cdot,\cdot) $, we define a metric on $ L^2(B_r(x))\times\M(B_r(x)) $ by
$$
D_{x,r}((u,\mu),(v,\eta))=r^{\al_p-n-2}\int_{B_r(x)}|u-v|^2+d_{x,r}(\mu,\eta),
$$
for two pairs $ (u,\mu),(v,\eta)\in L^2(B_r(x))\times\M(B_r(x)) $. By \eqref{Scaled01}, $ D_{x,r}(\cdot,\cdot) $ satisfies
$$
D_{0,1}((T_{x,r}u,T_{x,r}\mu),(T_{x,r}v,T_{x,r}\eta))=D_{x,r}((u,\mu),(v,\eta)).
$$

Now, we can give our definitions of quantitative symmetry and stratification. 

\begin{defn}[Quantitative symmetry by pairs]\label{qunsybypair}
Assume that $ u\in(H^1\cap L^{p+1})(B_{40}) $ is a stationary solution of \eqref{superequation}. For $ \va>0 $ and $ k\in\{0,1,2,...,n\} $, we say that $ u $ is $ (k,\va) $-symmetric in $ B_r(x)\subset B_{40} $, if there exists a $ k $-symmetric pair $ (v,\mu)\in L_{\loc}^2(\R^n)\times\M(\R^n) $ (or simply a $ k $-symmetric pair $ (v,\mu)\in L^2(B_1)\times\M(B_1) $ in $ B_1 $) such that 
$$
D_{0,1}((T_{x,r}u,T_{x,r}m_u),(v,\mu))<\va.
$$
\end{defn}

Intuitively, quantitative symmetry implies that the blow-up of the pair $ (u,m_u) $ at $ x $, with scale $ r $, is in the $ \va $-neighborhood of a $ k $-symmetric pair $ (v,\mu) $.

\begin{defn}[Quantitative stratification by pairs]\label{defSkvar}
Assume that $ u\in(H^1\cap L^{p+1})(B_{40}) $ is a stationary solution of \eqref{superequation}. For any $ \va>0 $, $ k\in\{0,1,2,...,n-1\} $, and $ 0<r<1 $, the $ k $-th $ (\va,r) $-stratification of $ u $, denoted by $ S_{\va,r}^k(u) $, is given by
$$
S_{\va,r}^k(u):=\{x\in B_{10}:u\text{ is not }(k+1,\va)\text{-symmetric in }B_s(x)\text{ for any }r\leq s<1\}.
$$
We also define
$$
S_{\va}^k(u):=\bigcap_{0<r<1}S_{\va,r}^k(u).
$$
\end{defn}

From the definitions above, we see that
$$
S_{\va}^k(u)=\{x\in B_{10}:u\text{ is not }(k+1,\va)\text{ symmetric in }B_r(x)\text{ for any }0<r<1\}.
$$
Letting $ B_r(x)\subset B_{40} $, the energy density is given by
\be
\theta_r(u,x):=r^{\al_p-n}\int_{B_r(x)}\(\f{p-1}{2}|\na u|^2+\f{p-1}{p+1}|u|^{p+1}\).\label{DenDef1}
\ee
Note that $ \theta_r(u,x)=r^{\al_p-n}m_u(B_r(x)) $. We can present our main theorems on the estimates for $ r $-neighborhoods and the rectifiability of quantitative stratification.

\begin{thm}\label{volthm}
Let $ \va>0 $ and $ k\in\{0,1,2,...,n-\lceil\al_p\rceil\} $. Assume that $ u\in(H^1\cap L^{p+1})(B_{40}) $ is a stationary solution of \eqref{superequation}, satisfying $ \theta_{40}(u,0)\leq\Lda $. There exists a constant $ C>0 $, depending only on $ \va,\Lda,n $, and $ p $ such that for any $ 0<r<1 $, there holds
\be
\cL^n(B_r(S_{\va,r}^k(u)\cap B_1))\leq Cr^{n-k},\label{main1es}
\ee
where for $ A\subset\R^n $, 
$$
B_r(A):=\bigcup_{x\in A}B_r(x)=\{y\in\R^n:\dist(y,A)<r\} 
$$
is the $ r $-neighborhood of $ A $. In particular,
\be
\cL^n(B_r(S_{\va}^k(u)\cap B_1))\leq Cr^{n-k}.\label{main2es}
\ee
Moreover, $ S_{\va}^k(u) $ is upper Ahlfors $ k $-regular. 
\end{thm}

A set $ A\subset\R^n $ is said to be upper Ahlfors $ k $-regular, if there exists some constant $ C>0 $ such that 
$$
\HH^k(A\cap B_r(x))\leq Cr^k, 
$$
for any $ x\in A $ and $ 0<r<\diam(A) $.

This theorem provides the estimates for the Minkowski dimensions of $ S_{\va}^k(u) $ with $ \va>0 $. To clarify them, we will define the Minkowski content and dimension.

\begin{defn}[Minkowski content, measure, and dimension]\label{Minkdef}
Let $ k\in\{0,1,2,...,n\} $, $ r>0 $, and $ S\subset\R^n $. The $ k $-dimensional Minkowski $ r $-content of $ S $ is given by
$$
\op{Min}_r^k(S):=(2r)^{k-n}\cL^n(B_r(S)).
$$
Define the upper and lower Minkowski measures by
\begin{align*}
\ol{\op{Min}}_0^k(S)&:=\limsup_{r\to 0^+}\op{Min}_r^k(S),\\
\underline{\op{Min}}_0^k(S)&:=\liminf _{r \rightarrow 0^+}\op{Min}_r^k(S).
\end{align*}
Minkowski dimension (or box-dimension) is defined by
$$
\dim_{\op{Min}}S:=\inf\{k\geq 0:\ol{\op{Min}}_0^k(S)=0\}.
$$
\end{defn}

In view of this definition, we can directly observe that for any $ \va>0 $, if $ u\in(H^1\cap L^{p+1})(B_{40}) $ is given by Theorem \ref{volthm}, then 
$$
\dim_{\op{Min}}(S_{\va}^k(u))\leq k,
$$
for any $ k\in\{0,1,2,...,n-\lceil\al_p\rceil\} $. Next, we consider the rectifiability of $ S_{\va}^k(u) $ and $ S^k(u) $. Rectifiability is a generalized concept that characterizes sets as the graphs of Lipschitz functions. To further apply this concept, we provide some definitions as follows.

\begin{defn}[Rectifiability]
Let $ N\in\Z_+ $ and $ k\in\{1,2,...,N\} $. We call a set $ M\subset\R^N $ as countably $ n $-rectifiable (or simply rectifiable), if
$$
M\subset M_0\cup\bigcup_{i\in\Z_+} f_i(\R^k),
$$
where $ \HH^k(M_0)=0 $, and $ f_i:\R^k\to\R^N $ is a Lipschitz map for any $ i\in\Z_+ $. 
\end{defn}

By the extension theorem of Lipschitz maps (see Theorem 1.2 of Chapter 2 in \cite{Sim83} for instance), this is equivalent to
$$
M=M_0\cup\bigcup_{i\in\Z_+}f_i(A_i),
$$
where $ \HH^k(M_0)=0 $, and $ f_i:A_i\subset\R^k\to\R^N $ is a Lipschitz map for any $ i\in\Z_+ $. In the study of rectifiable sets, the most remarkable property is the $ \HH^k $-a.e. existence of approximate tangent spaces. 

\begin{defn}[Approximate tangent space]
Let $ N\in\Z_+ $ and $ k\in\{1,2,...,N\} $. Assume that $ M $ is an $ \HH^k $-measurable subset of $ \R^N $ with $ \HH^k(M\cap K)<+\ift $, for any compact subset $ K\subset\R^N $. We call that a $ k $-dimensional subspace $ V\subset\R^N $ is the approximate tangent space for $ M $ at $ x\in M $, if
$$
\lim_{\lda\to 0^+}\int_{\lda^{-1}(M-x)}\vp(y)\ud\HH^k(y)=\int_{V}\vp(y)\ud\HH^k(y),
$$
for any $ \vp\in C_0(\R^N) $. If such $ V $ exists, we denote it by $ T_xM $.
\end{defn}

The following theorem establishes the connection between rectifiability and the existence of the approximate tangent spaces.

\begin{thm}[\cite{Sim83}, Theorem 1.6 of Chapter 3]\label{Simexis}
Let $ N\in\Z_+ $ and $ k\in\{1,2,...,N\} $. Assume that $ M $ is $ \HH^k $-measurable with $ \HH^k(M\cap K)<+\ift $ for any compact set $ K\subset \R^N $. Then $ M $ is $ k $-rectifiable if and only if the approximate tangent space $ T_xM $ exists for $ \HH^k $-a.e. $ x\in M $.
\end{thm}

For more results on the rectifiability and recent developments in the study of geometric measure theory and variational problems, one can refer to the survey paper \cite{Mat23}.

\begin{thm}\label{rectthm}
Let $ \va>0 $ and $ k\in\{1,2,...,n-\lceil\al_p\rceil\} $. Assume that $ u\in(H^1\cap L^{p+1})(B_{40}) $ is a stationary solution of \eqref{superequation}, satisfying $ \theta_{40}(u,0)\leq\Lda $. Then $ S_{\va}^k(u) $ and $ S^k(u) $ are both $ k $-rectifiable. Moreover, for $ \HH^k $-a.e. $ x\in S_{\va}^k(u) $ or $ x\in S^k(u) $, there exists a $ k $-dimensional subspace $ V\subset\R^n $ such that any tangent pair of $ u $ at $ x $ is $ k $-symmetric with respect to $ V $.
\end{thm}

Owing to the analogous findings for harmonic maps presented in \cite{NV17}, the estimates \eqref{main1es} and \eqref{main2es}, together with the properties of the rectifiability of $ S^k(u) $ are all sharp. Notably, the proof of rectifiability for the top stratum $ S^{n-\rceil\al_p\rceil}(u) $ is more straightforward and has been established in \cite{WW15}, drawing on techniques from \cite{Pre87}. In a self-contained manner, \cite{Lin99} gave similar results for harmonic maps, and we believe these arguments can be adapted to demonstrate the rectifiability of $ S^{n-\lceil\al_p\rceil}(u) $. 

The central idea in the proof of Theorem \ref{rectthm} is that we can regard the set $ S_{\va}^k(u) $ as an approximation of $ S^k(u) $. Moreover, the relation is given by Lemma \ref{QuantiStrati}. Generally, $ S^k(u) $ is the union of all $ S_{\va}^k(u) $ with $ \va>0 $. Since $ S^k(u) $ is characterized by tangent pairs, here is a taking limit procedure, which makes the analysis more complicated. However, in the definition of $ S_{\va}^k(u) $, there is no such a process, and we can apply more methods here. Since we are concerned with the rectifiability, there is no loss since the union of $ S_{\va}^k(u) $ is actually countable by the inclusion property.

Numerous studies explore related themes. For approximate harmonic maps, \cite{NV18} is related; for energy-minimizing $ Q $-valued maps, one may consult \cite{DMSV18,HSV19}; and for $ p' $-harmonic maps, the findings are detailed in \cite{Ved21}. Additional relevant literature covers diverse geometric variational problems, as evidenced in works such as \cite{Alp18,Alp20,EE19,GJXZ24,Wan21}.

In contrast to the existing studies, in the problem of this paper, we perform the stratification of the stationary solution $ u $ for \eqref{superequation} based on tangent pairs, rather than only the tangent functions, applied in \cite{CN13a,CN13b}. The stratification in our paper is more subtle than those given by tangent functions. For $ u\in(H^1\cap L^{p+1})(B_{40}) $, being a stationary solution of \eqref{superequation}, we can define the stratification by tangent functions in the spirit of \cite{CN13a,CN13b} by
$$
S_{[\CN]}^k(u):=\{x\in B_{10}:\text{there is no }(k+1)\text{-symmetric tangent }\text{function of }u\text{ at }x\}.
$$
Consequently, we can also establish quantitative stratification related to $ S_{[\CN]}^k(u) $.

\begin{defn}[Quantitative stratification by tangent functions]\label{quansybyfun}
Let $ \va>0 $, $ k\in\{0,1,2,...,n\} $, $ 0<r<1 $, and $ B_r(x)\subset B_{40} $. Assume that $ u\in(H^1\cap L^{p+1})(B_{40}) $ is a stationary solution of \eqref{superequation}. We call $ u $ is [CN] $ (k,\va) $-symmetric in $ B_r(x) $ if there exists a $ k $-symmetric function $ v\in L_{\loc}^2(\R^n) $ (or simply a $ k $-symmetric function $ v\in L^2(B_1) $ in $ B_1 $) such that
$$
\int_{B_1}|T_{x,r}u-v|^2<\va.
$$
We also define $ S_{[\CN];\va,r}^k(u) $ as the collection of points $ x\in B_{10} $ such that $ u $ is not [CN] $ (k+1,\va) $-symmetric in $ B_s(x) $ for any $ r\leq s<1 $.
\end{defn}

An interesting question is the connection between $ S_{\va,r}^k(u) $ and $ S_{[\CN];\va,r}^k(u) $. In some cases, these two are equivalent, while in other cases, one is stronger. In a related study, similar results on the comparison of different types of stratification were provided in \cite{HSV19}. In this context, we present the following proposition and its corollary, which give the answer to this question.

\begin{prop}\label{proprela}
Let $ k\in\{0,1,2,...,n\} $, $ r>0 $, and $ x\in\R^n $. Assume that $ u\in(H^1\cap L^{p+1})(B_r(x)) $ is a stationary solution of \eqref{superequation}, satisfying $ \theta_r(u,x)\leq\Lda $. The following properties hold.
\begin{enumerate}[label=$(\theenumi)$]
\item For any $ \va>0 $, if $ u $ is $ (k,\va) $-symmetric in $ B_r(x) $, then $ u $ is $ [\CN] $ $ (k,\va) $-symmetric in $ B_r(x) $.
\item Assume that $ \al_p $ is not an integer. For any $ \va>0 $, there exists $ \delta>0 $, depending only on $ \va,\Lda,n $, and $ p $ such that if $ u $ is $ [\CN] $ $ (k,\delta) $-symmetric in $ B_r(x) $, then $ u $ is $ (k,\va) $-symmetric in $ B_r(x) $.
\end{enumerate}
\end{prop}

\begin{cor}\label{correla}
Let $ k\in\{0,1,2,...,n\} $, and $ 0<r<1 $. Assume that $ u\in(H^1\cap L^{p+1})(B_{40}) $ is a stationary solution of \eqref{superequation}, satisfying $ \theta_{40}(u,x)\leq\Lda $. The following properties hold.
\begin{enumerate}[label=$(\theenumi)$]
\item For $ \va>0 $, $ S_{[\CN];\va,r}^k(u)\subset S_{\va,r}^k(u) $.
\item If $ \al_p $ is not an integer, then for any $ \va>0 $, there exists $ \delta>0 $, depending only on $ \va,\Lda,n $, and $ p $ such that $ S_{\va,r}^k(u)\subset S_{[\CN];\delta,r}^k(u) $. 
\item If $ \al_p $ is not an integer, then $ S_{[\CN]}^k(u)=S^k(u) $.
\end{enumerate}
\end{cor}

The stratification based on symmetric properties of tangent pairs offers a significant advantage. In cases when $ \alpha_p $ is an integer, it is possible that 
$$
\sing(u)\backslash S_{[\CN]}^{n-1}(u)\neq\emptyset.
$$
As a result, a point could belong to the singular set even if the tangent function vanishes. Similar problems arise in the study of other scenarios such as harmonic maps, as demonstrated in \cite{HLP92,Poo91}. Bubbling conditions on the target manifold, as discussed in \cite{Lin99,SU82}, can help resolve such issues in the context of harmonic maps. However, for the current model, imposing such conditions would be unnatural, given that the bubble represents a non-trivial smooth solution to the lower-dimensional Yamabe problem (refer to Theorem 1.2 in \cite{WW15}). Consequently, the stratification based on tangent pairs avoids this phenomenon and can lead to more robust estimates.

The outcomes derived in our model exhibit similarities to those of $ p' $-harmonic maps as presented in \cite{Ved21}. This connection can be intuitively understood from the energy density in \eqref{DenDef1}. More precisely, for any $ x\in\sing(u)\backslash S_{[\CN]}^{n-1}(u) $, up to a subsequence $ r_i\to 0^+ $, there exist measures $ \nu_1,\nu_2\in\M(\R^n) $, and $ (\nu_1^{\beta\ga})_{\beta,\ga=1}^n\in\M(\R^n,\R^{n\times n}) $ such that
\be
\begin{gathered}
\f{|\na u_i|^2}{2}\ud y\wc^*\nu_1,\,\,\f{|u_i|^{p+1}}{p+1}\ud y\wc^*\nu_2\text{ in }\M(\R^n),\\
\(\f{\pa_{\beta}u_i\pa_{\ga}u_i}{2}\ud y\)\wc^*(\nu_1^{\beta\ga})\text{ in }\M(\R^n,\R^{n\times n}),
\end{gathered}\label{limitupau}
\ee
where $ u_i:=T_{x,r_i}u $. From Lemma 2.6 of \cite{WW15} or the second property of Proposition \ref{DefMea}, we deduce that $ 2\nu_1=(p+1)\nu_2 $. Consequently, a tangent measure at $ x $ is given by
$$
\mu=(p-1)(\nu_1+\nu_2)=\f{(p-1)(p+3)}{p+1}\nu_1.
$$
This, together with \eqref{StaCon} and \eqref{limitupau}, implies that
\be
\int_{\R^n}(\op{div}Y\ud\mu-\al_pD_{\beta}Y_{\ga}\ud\mu^{\beta\ga})=0,\label{statimeas}
\ee
for any $ Y\in C_0^{\ift}(\R^n,\R^n) $, where 
$$
(\mu^{\beta\ga}):=\(\f{(p-1)(p+3)}{p+1}\nu_1^{\beta\ga}\)\in\M(\R^n,\R^{n\times n}).
$$
By the terminology in \cite{Mos03}, we see that the matrix-valued Radon measure $ (\mu^{\beta\ga})_{\beta,\ga=1}^n $ is identified as a stationary measure with index $ \al_p $. For $ p' $-harmonic map $ f\in W^{1,p'}(B_{40},\cN) $, where $ \cN\subset\R^d $ is a smooth and compact manifold, we see that $
\{|\na f|^{p'-2}\pa_{\beta}f\cdot\pa_{\ga}f\}_{\beta,\ga=1}^n $ is a stationary measure with index $ p' $. The analysis above reveals the relations between our model and $ p' $-harmonic maps.\smallskip

Theorem \ref{volthm} holds importance in enhancing the regularity of stationary solutions of \eqref{superequation} when $ \al_p $ is not an integer. Initially, we introduce the concept of the regularity scale.

\begin{defn}
Let $ j\in\Z_{\geq 0} $. Assume that $ u\in(H^1\cap L^{p+1})(B_{40}) $ is a stationary solution of \eqref{superequation}. For $ x\in B_1 $, we define the regularity scale $ r_u^{j}(x) $ of $ u $ at $ x $ with order $ j $ by
$$
r_u^{j}(x):=\sup\left\{0\leq r\leq 1:\sup_{B_r(x)}\(\sum_{i=0}^{j}r^i|D^iu|\)\leq r^{-\f{2}{p-1}}\right\}.
$$
\end{defn}

It is worth noting that by standard regularity theory of elliptic equations, if $ r_u^j(x)=r>0 $ for some $ j\in\Z_+ $, then $ u\in C^{\ift}(B_{\f{r}{2}}(x)) $.

\begin{thm}\label{enhance}
Let $ j\in\Z_{\geq 0} $. Assume that $ u\in(H^1\cap L^{p+1})(B_{40}) $ is a stationary solution of \eqref{superequation}, satisfying $ \theta_{40}(u,0)\leq\Lda $. The following properties hold.
\begin{align*}
&\cL^n(B_r(\sing(u)\cap B_1))+\cL^n\(\left\{x\in B_1:\sum_{i=0}^{j}r^i|D^iu(x)|>r^{-\f{2}{p-1}}\right\}\)\\
&\quad\quad\leq\left\{\begin{aligned}
&Cr^{\lfloor\al_p\rfloor+1},&\text{ if }\al_p\notin\Z_+,\\
&Cr^{\al_p},&\text{ if }\al_p\in\Z_+,
\end{aligned}\right.
\end{align*}
for any $ r\in(0,1] $, where $ \lfloor a\rfloor=\sup\{b\in\Z:b\leq a\} $ for $ a\in\R $, and  $ C>0 $ depends only on $ j,\Lda,n $, and $ p $. 
\end{thm}

A direct consequence is that if $ \al_p $ is not an integer, we can obtain improved interior regularity estimates for \eqref{superequation}.

\begin{cor}\label{regq}
Let $ j\in\Z_{\geq 0} $. Assume that $ u\in(H^1\cap L^{p+1})(B_{40}) $ is a stationary solution of \eqref{superequation}, satisfying $ \theta_{40}(u,0)\leq\Lda $. If $ \al_p $ is not an integer, then $ D^{j}u\in L^{q_j,\ift}(B_1) $, where
\be
q_j=\f{(p-1)(\lfloor\al_p\rfloor+1)}{2+j(p-1)}.\label{qj}
\ee
In particular, we have the estimate
$$
\|D^ju\|_{L^{q_j,\ift}(B_1)}\leq C,
$$
where $ C>0 $ depends only on $ j,\Lda,n $, and $ p $.
\end{cor}

Given some interpolation methods in harmonic analysis (see Chapter 1 of \cite{Gra08}), for $ u $ given above, we observe that
$$
\|D^ju\|_{L^q(B_1)}\leq C(j,\Lda,n,p),\quad 0<q<q_j.
$$
Selecting $ p>\f{n+2}{n-2} $ such that $ n=\lfloor\al_p\rfloor+1 $ and $ \al_p $ is not an integer, we identify the function $ v_0(x)=c_0|x|^{-\f{2}{p-1}} $ as a stationary solution of \eqref{superequation}, where $ c_0\neq 0 $ is a constant dependent solely on $ n $ and $ p $. For any $ j\in\Z_+ $, it follows that
$$
\cL^n(\{x\in B_1:|D^jv_0(x)|>\lda\})\sim\lda^{-q_j},
$$
for any $ \lda>0 $, and $ q_j $ is given by \eqref{qj}. Consequently, the estimates provided in Corollary \ref{regq} are proven to be optimal. In cases when $ \al_p $ is an integer, the analysis of blow-ups may involve defect measures, preventing the enhancement of regularity as in instances where $ \al_p\not\in\Z_+ $ using the same method.

\subsection{Difficulties and approaches}

In the proof of main theorems in our model, we will adopt the Reifenberg-type theorems given in \cite{NV17}. However, arguments applied to the study of harmonic maps in \cite{NV17} cannot be directly used here due to several inherent difficulties in this problem.

In the study of harmonic maps, a commonly used energy density is defined as
$$
\theta_r^{\op{har}}(f,x):=r^{2-n}\int_{B_r(x)}|\na f|^2,
$$
where $ f\in H^1(B_{40},\cN) $ is a stationary harmonic map with some compact smooth manifold $ \cN\subset\R^d $ and $ B_r(x)\subset B_{40} $. An essential property of this energy density is that if $ B_r(x)\subset B_s(y)\subset B_{40} $, then
\be
\theta_r^{\op{har}}(f,x)\leq\(\f{r}{s}\)^{n-2}\theta_s^{\op{har}}(f,y).\label{simesti}
\ee
In \cite{NV17}, the authors applied this property to obtain a useful lemma (Lemma 8.2), which concerns the energy difference between two points with closed distance. In our problem, the widely used density given in \cite{Pac93} is
\be
\theta_r^{\op{sem}}(u,x):=r^{\al_p-n}\int_{B_r(x)}\(\f{|\na u|^2}{2}-\f{|u|^{p+1}}{p+1}\)+\f{r^{\al_p-n-1}}{p-1}\int_{\pa B_r(x)}|u|^2,\label{semidensi}
\ee
where $ u\in(H^1\cap L^{p+1})(B_{40}) $ denotes a stationary solution of \eqref{superequation} with $ B_r(x)\subset B_{40} $. Unfortunately, \eqref{simesti} does not hold for this density.

After \cite{NV17}, the same authors simplified the arguments in \cite{NV18} for approximate harmonic maps. However, they still used similar estimates in the proof (see formula (4.32) of that paper). Therefore, to apply the theories in \cite{NV17} to our problem, we need to give some non-trivial adjustments to the original energy density. Inspired by works of \cite{HSV19,Ved21,NV24,Sin18}, we will apply some modifications to the energy density above by utilizing cutoff functions (see \eqref{DenDef2}). These modifications allow us to employ the arguments in \cite{NV18} without relying on similar estimates like \eqref{simesti}. Intuitively, the idea of using cutoff functions to modify the energy density comes from the observation that even though a sequence of Radon measures $ \{\mu_i\}\subset\M(\R^n) $ satisfies $ \mu_i\wc^*\mu\in\M(\R^n) $, it is not necessarily true
$$
\lim_{i\to+\ift}\mu_i(B_r(x))=\mu(B_r(x)),
$$
for any $ B_r(x)\subset\R^n $, but
$$
\lim_{i\to+\ift}\int_{B_r(x)}\vp\(\f{|y-x|}{r}\)\ud\mu_i(y)=\int_{B_r(x)}\vp\(\f{|y-x|}{r}\)\ud\mu(y),
$$
where $ \vp\in C_0(B_1) $.

In the proof of \cite{NV18}, the authors also utilized the unique continuation property of harmonic maps to simplify the arguments from \cite{NV17}. This property asserts that if two weakly harmonic maps are smooth in a set with sufficiently large measure and identical in an open set, they are equivalent across the entire domain. In our problem settings, we can still show the unique continuation for weak solutions of \eqref{superequation} (see Proposition \ref{Ucp}). However, for measures, there is no such result. In this paper, most results are on tangent pairs, so we have to deal with the properties of measures. Consequently, we must find an alternative approach to replace the usage of these unique continuation properties. Fortunately, our modifications of the energy density also allow us to achieve this goal.

\subsection{Further discussions, applications, and remarks} The heat equation associated with \eqref{superequation} is given by
\be
\pa_tu-\Delta u =|u|^{p-1}u,\text{ in } B_1\times[0,T].\label{heatequ}
\ee
The blow-up analysis of the equation \eqref{heatequ}, based on the series methods developed by Lin and Wang in \cite{LW99,LW02a,LW02b}, was presented in the recent paper \cite{WW21}. Note that we can also investigate the enhancement of regularity through quantitative stratification. For mean curvature flow and harmonic heat flow, one can refer to \cite{CHN13} and \cite{CHN15}, respectively. In particular, the methods employed in \cite{CHN13,CHN15} can be applied to the study of the equation \eqref{heatequ}, yielding new estimates.

In this paper, we study stratification based on tangent pairs. Consequently, we can apply these arguments to consider the stratification of stationary measures $ (\mu^{\beta\ga})\in\M(\R^n,\R^{n\times n}) $ satisfying \eqref{statimeas}. We believe that the $ k $-stratum is $ k $-rectifiable, and similar estimates to those outlined in Theorem \ref{volthm} can likewise be derived.

\subsection{Organization of the paper} 

In this manuscript, we will follow the frameworks of \cite{NV18}. In \S \ref{Preliminaries}, we will present some primary ingredients in the proof, such as the modified energy density, the corresponding monotonicity formula, and properties of tangent functions and tangent measures. In \S \ref{QuantiStrati}, we apply the arguments in \cite{CN13b,NV17,NV18} to derive some essential properties in quantitative stratification. Moving on to \S \ref{RandL2}, we will introduce the Reifenberg-type theorems and present the $ L^2 $-best approximation results, which establish the relationships between Reifenberg-type theorems and the monotonicity formula. Subsequently, in \S \ref{coveringlemma}, we will employ inductive reasoning and theorems from preceding sections to establish valuable covering lemmas. Finally, we will use these lemmas to complete proofs of the main theorems in \S \ref{Mainproof}.

\subsection{Notations and conventions}
\begin{itemize}
\item Throughout this paper, we will use $ C $ to denote positive constants. Sometimes to emphasize that $ C $ depends on parameters $ a,b,... $, we use the notation $ C(a,b,...) $, which may change from line to line.
\item We will use the Einstein summation convention throughout this paper, summing the repeated index without the sum symbol.
\item For $ k\in\{1,2,...,n\} $, the Grassmannian $ \bG(n,k) $ is the set of all $ k $-dimensional subspaces of $ \R^n $ and $ \bA(n,k) $ is the collections of all $ k $-dimensional affine subspaces of $ \R^n $. For $ \{V\}\cup\{V_i\}\subset\bG(n,k) $, we call $ V_i\to V $ as $ i\to+\ift $ in the sense that $ \lim_{i\to+\ift}d_{\bG(n,k)}(V_i,V)=0 $, where $ d_{\bG(n,k)} $ is the Grassmannian metric. For $ \{L\}\cup\{L_i\}\subset\bA(n,k) $, we say $ L_i\to L $ if $ L_i=x_i+V_i $, $ L=x+V $, where $ \{V\}\cup\{V_i\}\subset\bG(n,k) $, with $ V_i\to V $ and $ x_i\to x $. 
\item Assume that $ U\subset\R^n $ is a domain. We call $ \mu\in\M(U) $ as a probability measure if $ \mu(U)=1 $. 
\item For $ r>0 $, $ x\in\R^n $, and $ k\in\{1,2,...,n\} $, we let 
$$
B_r^k(x):=\{x\in\R^k:|y-x|<r\}.
$$
If $ k=n $, we drop the superscription. If $ x=0 $, we denote it by $ B_r^k $. $ \HH^k $ is $ k $-dimensional Hausdorff measure on $ \R^n $. When $ k=n $, we denote $ \cL^n=\HH^n $ as the Lebesgue measure and $ \ud\cL^n(x)=\ud x $. If ambiguity occurs, we will drop $ \ud x $ in integrals. 
\item For a $ k $-dimensional subspace $ V=\op{span}\{v_i\}_{i=1}^k $, where $ \{v_i\}_{i=1}^k $ is an orthonormal basis and $ u\in H^1(\R^n) $, we set 
$$
|V\cdot\na u|^2=\sum_{i=1}^k|v_i\cdot\na u|^2.
$$
\item We have a convention that $ 0 $-dimensional affine subspaces refer to single points.
\item For a $ \cL^n $-measurable set $ A\subset\R^n $ with $ \cL^n(A)<+\ift $, and $ u\in L^1(A) $, we denote the average of integral of $ u $ on $ A $ by $ \dashint_Au:=\f{1}{\cL^n(A)}\int_{A}u $.
\item For $ k\in\Z_+ $, we denote $ \w_k:=\HH^k(B_1^k) $.
\end{itemize}

\section{Preliminaries}\label{Preliminaries}

\subsection{Monotonicity formula and partial regularity} In this subsection, we will establish the monotonicity formula for stationary solutions of \eqref{superequation}. Firstly, we need to modify the classical energy density \eqref{semidensi} by using some cutoff functions. 

\begin{defn}\label{defnofphi}
Let $ \phi:[0,+\ift)\to[0,+\ift) $ be a smooth function that satisfies the following properties.
\begin{enumerate}[label=$(\theenumi)$]
\item $ \supp\phi\subset[0,10) $.
\item $ \phi(t)\geq 0 $, and $ |\phi''(t)|+|\phi'(t)|\leq 100 $, for any $ t\in[0,10) $.
\item $ -2\leq\phi'(t)\leq -1 $, for any $ t\in[0,8] $.
\item $ \phi'(t)<0 $ for any $ t\in[0,9] $.
\end{enumerate}
For $ x\in\R^n $, we define $ \phi_{x,r},\dot{\phi}_{x,r}:\R^n\to[0,+\ift) $ as
$$
\phi_{x,r}(y):=\phi\(\f{|y-x|^2}{r^2}\),\quad\dot{\phi}_{x,r}(y):=\phi'\(\f{|y-x|^2}{r^2}\).
$$
\end{defn}
If $ u\in (H^1\cap L^{p+1})(B_{40}) $, and $ B_{10r}(x)\subset B_{40} $, we set
\be
\vt_r(u,x):=r^{\al_p-n}\int_{\R^n}\(\f{|\na u|^2}{2}-\f{|u|^{p+1}}{p+1}\)\phi_{x,r}-\f{2r^{\al_p-n-2}}{p-1}\int_{\R^n}|u|^2\dot{\phi}_{x,r}.\label{DenDef2}
\ee
By Definition \ref{defnofphi}, the integrals in the formula \eqref{DenDef2} are actually on $ B_{10r}(x) $, so they are all well-defined. Applying H\"{o}lder's inequality, we have
\be
\vt_r(u,x)\leq C\max\{\theta_{10r}(u,x),r^2(\theta_{10r}(u,x))^{\f{2}{p+1}}\},\label{thevt}
\ee
where $ C>0 $ depends only on $n$ and $p$. Additionally, through simple calculations, $ \theta_{(\cdot)}(\cdot,\cdot) $ (defined in \eqref{DenDef1}) and $ \vt_{(\cdot)}(\cdot,\cdot) $ are scaling invariant, namely, if $ u\in(H_{\loc}^1\cap L_{\loc}^{p+1})(\R^n) $, $ x,y\in\R^n $, and $ r,R>0 $, then
\be
\begin{aligned}
\theta_R(T_{x,r}u,y)&=\theta_{Rr}(u,x+ry),\\
\vt_R(T_{x,r}u,y)&=\vt_{Rr}(u,x+ry).
\end{aligned}\label{Scaleinv}
\ee

\begin{prop}[Monotonicity formula]\label{MonFor}
Let $ r>0 $ and $ x\in\R^n $. Assume that $ u\in (H^1\cap L^{p+1})(B_{10r}(x)) $ is a stationary solution of \eqref{superequation}. Then
\begin{align*}
&\vt_{r_2}(u,x)-\vt_{r_1}(u,x)\\
&\quad\quad=-2\int_{r_1}^{r_2}\(\rho^{\al_p-n-1}\int_{\R^n}\left|\f{(y-x)\cdot\na u}{\rho}+\f{2u}{(p-1)\rho}\right|^2\dot{\phi}_{x,\rho}\ud y\)\ud\rho\geq 0,
\end{align*}
for any $ 0<r_1<r_2<r $. In particular, $ \vt_{\rho}(u,x) $ is a nondecreasing function with respect to $ \rho\in(0,r) $.
\end{prop}

Before proving this proposition, we will present some relevant consequences.

\begin{cor}\label{HomProCor}
Let $ r>0 $, and $ x\in\R^n $. Assume that $ u\in (H^1\cap L^{p+1})(B_{10r}(x)) $ is a stationary solution of \eqref{superequation}. If $ \vt_{r_1}(u,x)=\vt_{r_2}(u,x) $, for some $ 0<r_1<r_2<r $, then $ u $ is $ 0 $-symmetric at $ x $ in $ B_{9r_2}(x) $, namely
\be
u(x+\lda y)=\lda^{-\f{2}{p-1}}u(x+y)\label{HomPro}
\ee
for any $ \lda\in(0,1) $ a.e. $ y\in B_{9r_2} $.
\end{cor}

\begin{rem}\label{RemHom}
If $ u\in H_{\loc}^1(\R^n) $, taking derivatives for both sides of \eqref{HomPro} with respect to $ \lda $, we can obtain that $ u $ is $ 0 $-symmetric at $ x $, if and only if 
\be
(y-x)\cdot\na u+\f{2u}{p-1}=0,\text{ for a.e. }y\in\R^n.\label{yxup1}
\ee 
\end{rem}

\begin{proof}[Proof of Corollary \ref{HomProCor}]
Since $ \vt_{r_1}(u,x)=\vt_{r_2}(u,x) $, we have, for any $ \rho\in(r_1,r_2) $,
$$
\int_{\R^n}\left|\f{(y-x)\cdot\na u}{\rho}+\f{2u}{(p-1)\rho}\right|^2\dot{\phi}_{x,\rho}\ud y=0.
$$
Given Definition \ref{defnofphi}, this implies that \eqref{yxup1} holds for a.e. $ y \in B_{9r_2}(x) $, and \eqref{HomPro} then follows directly from Remark \ref{RemHom}.
\end{proof}

\begin{cor}\label{coruse}
Let $ r>0 $ and $ x\in\R^n $. Assume that $ u\in(H^1\cap L^{p+1})(B_{20r}(x)) $ is a stationary solution of \eqref{superequation}. Then
$$
\vt_{2s}(u,x)-\vt_s(u,x)\geq Cs^{\al_p-n-2}\int_{B_{8s}(x)}\left|(y-x)\cdot\na u+\f{2u}{p-1}\right|^2\ud y
$$
for any $ 0<s<r $, where $ C>0 $ depends only on $ n $ and $ p $.
\end{cor}
\begin{proof}
Using Definition \ref{defnofphi}, and Proposition \ref{MonFor}, we have
\begin{align*}
\vt_{2s}(u,x)-\vt_s(u,x)&\geq Cs^{\al_p-n-1}\int_s^{2s}\(\int_{B_{8s}(x)}\left|\f{(y-x)\cdot\na u}{s}+\f{2u}{(p-1)s}\right|^2\ud y\)\ud\rho\\
&\geq Cs^{\al_p-n-2}\int_{B_{8s}(x)}\left|(y-x)\cdot\na u+\f{2u}{p-1}\right|^2\ud y,
\end{align*}
for any $ 0<s<r $, which implies the result.
\end{proof}

We also want to provide some comparisons between the modified energy density $ \vt_{(\cdot)}(\cdot,\cdot) $ and the ordinary one, defined by \eqref{semidensi}. If $ u\in (H^1\cap L^{p+1})(B_r(x)) $ is a stationary solution of \eqref{superequation}, then we have the monotonicity formula
\be
\f{\ud}{\ud r}\theta_r^{\op{sem}}(u,x)=r^{\al_p-n}\int_{\pa B_r(x)}\left|\f{(y-x)\cdot\na u}{r}+\f{2u}{(p-1)r}\right|^2\ud\HH^{n-1}(y)\geq 0.\label{stanmo}
\ee
Considering this form, we see that $ \theta_{r_1}^{\op{sem}}(u,x)=\theta_{r_2}^{\op{sem}}(u,x) $ implies the $ -\f{2}{p-1} $-homogeneity in $ B_{r_2}(x)\backslash B_{r_1}(x) $. For the information of $ u $ in $ B_{r_1}(x) $, we have to apply the unique continuation Proposition \ref{Ucp}. Moreover, the estimate provided in Corollary \ref{coruse} may not be true for this density. When we consider the convergence of $ u $ in the sense of Radon measures, the cutoff function will allow us to take the limit directly.

\begin{proof}[Proof of Proposition \ref{MonFor}]
The proof is similar to that of the standard monotonicity formula \eqref{stanmo} given in \S 2 of \cite{Pac93}, and we will make some modifications here. Up to a translation, we can assume that $ x=0 $. For simplicity, we denote $ n_{\rho}(y):=\f{y}{\rho} $ with $ n_{\rho}^i(y)=\f{y_i}{\rho} $ for $ i\in\{1,2,...,n\} $, $ \phi_{\rho}=\phi_{0,\rho} $, and $ \pa_{\rho}u=n_{\rho}\cdot\na u $. Choosing the vector field $ Y(y)=\phi_{\rho}(y)y $, we can obtain 
\begin{align*}
\op{div}Y&=2|n_{\rho}|^2\dot{\phi}_{\rho}+n\phi_{\rho},\\
DY&=(\pa_iY_j)_{i,j=1}^n=(2n_{\rho}^in_{\rho}^j\dot{\phi}_{\rho}+\delta_{ij}\phi_{\rho})_{i,j=1}^n.    
\end{align*}
Using the stationary condition \eqref{StaCon}, we deduce that
\be
\begin{aligned}
2\int |\pa_{\rho}u|^2\dot{\phi}_{\rho}&=\f{n-2}{2}\int|\na u|^2\phi_{\rho}-\f{n}{p+1}\int|u|^{p+1}\phi_{\rho}\\
&\quad\quad+\int|\na u|^2|n_{\rho}|^2\dot{\phi}_{\rho}-\f{2}{p+1}\int |u|^{p+1}|n_{\rho}|^2\dot{\phi}_{\rho}.
\end{aligned}\label{Sta1}
\ee
Testing \eqref{WeakCon} by $ u\phi_{\rho} $, there holds
\be
\int|u|^{p+1}\phi_{\rho}-\int|\na u|^2\phi_{\rho}=2\rho^{-1}\int u\pa_{\rho}u\dot{\phi}_{\rho}.\label{Weak1}
\ee
Taking derivatives on both sides of \eqref{Weak1} with respect to $ \rho $, we have
\be
\rho^{-1}\int|\na u|^2|n_{\rho}|^2\dot{\phi}_{\rho}-\rho^{-1}\int|u|^{p+1}|n_{\rho}|^2\dot{\phi}_{\rho}=\f{\ud}{\ud\rho}\(\rho^{-1}\int u\pa_{\rho}u\dot{\phi}_{\rho}\).\label{Weak2}
\ee
Combining \eqref{Sta1}, \eqref{Weak1}, and \eqref{Weak2}, we obtain 
\begin{align*}
&\rho^{-1}\(\f{n-2}{2}-\f{n}{p+1}\)\int|u|^{p+1}\phi_{\rho}+\rho^{-1}\(\f{p-1}{p+1}\)\int|u|^{p+1}|n_{\rho}|^2\dot{\phi}_{\rho}\\
&\quad\quad=\rho^{-1}\int\(2|\pa_{\rho}u|^2+\f{(n-2)u\pa_{\rho}u}{\rho}\)\dot{\phi}_{\rho}-\f{\ud}{\ud\rho}\(\rho^{-1}\int u\pa_{\rho}u\dot{\phi}_{\rho}\).
\end{align*}
Multiplying both sides of the formula above by $ \rho^{\al_p-n} $, this yields that
\be
\begin{aligned}
&\rho^{\al_p-n-1}\int\(2|\pa_{\rho}u|^2+\f{(n-2)u\pa_{\rho}u}{\rho}\)\dot{\phi}_{\rho}\\
&\quad\quad=\rho^{\al_p-n}\f{\ud}{\ud\rho}\(\rho^{-1}\int u\pa_{\rho}u\dot{\phi}_{\rho}\)-\f{p-1}{2(p+1)}\f{\ud}{\ud\rho}\(\rho^{\al_p-n}\int|u|^{p+1}\phi_{\rho}\).
\end{aligned}\label{Cal1}
\ee
To proceed, since $ \al_p=\f{2(p+1)}{p-1} $, we see that
\be
\begin{aligned}
\rho^{\al_p-n}\f{\ud}{\ud\rho}\(\rho^{-1}\int u\pa_{\rho}u\dot{\phi}_{\rho}\)&=\f{\ud}{\ud\rho}\(\rho^{\al_p-n-1}\int u\pa_{\rho}u\dot{\phi}_{\rho}\)\\
&\quad\quad+\(n-\f{2(p+1)}{p-1}\)\(\rho^{\al_p-n-2}\int u\pa_{\rho}u\dot{\phi}_{\rho}\).
\end{aligned}\label{Cal2}
\ee
On the other hand, it follows from direct calculations that
\begin{align*}
&\f{(p+3)\rho^{\al_p-n-1}}{p-1}\int\(\f{2|u|^2}{(p-1)\rho^2}+\f{u\pa_{\rho}u}{\rho}\)\dot{\phi}_{\rho}\\
&\quad\quad=\f{1}{2}\f{\ud^2}{\ud\rho^2}\(\rho^{\al_p-n-1}\int |u|^2\dot{\phi}_{\rho}\)-\f{\ud}{\ud \rho}\(\rho^{\al_p-n-1}\int u\pa_{\rho}u\dot{\phi}_{\rho}\).
\end{align*}
This, together with \eqref{Cal1} and \eqref{Cal2}, implies that
\begin{align*}
&\f{p-1}{2(p+1)}\f{\ud}{\ud\rho}\(\rho^{\al_p-n}\int|u|^{p+1}\phi_{\rho}\)-\f{1}{2}\f{\ud^2}{\ud\rho^2}\(\rho^{\al_p-n-1}\int |u|^2\dot{\phi}_{\rho}\)\\
&\quad\quad=-\rho^{\al_p-n-1}\int\(2|\pa_{\rho}u|^2+\f{(p+7)u\pa_{\rho}u}{(p-1)\rho}+\f{2(p+3)|u|^2}{(p-1)^2\rho^2}\)\dot{\phi}_{\rho},
\end{align*}
and then
\begin{align*}
&-2\rho^{\al_p-n-1}\int_{\R^n}\left|\pa_{\rho}u+\f{2u}{(p-1)\rho}\right|^2\dot{\phi}_{\rho}\\
&\quad\quad=\f{\ud}{\ud\rho}\[\f{(p-1)\rho^{\al_p-n}}{2(p+1)}\int|u|^{p+1}\phi_{\rho}-\f{\ud}{\ud\rho}\(\f{\rho^{\al_p-n-1}}{2}\int |u|^2\dot{\phi}_{\rho}\)+\f{\rho^{\al_p-n-2}}{2}\int |u|^2\dot{\phi}_{\rho}\].
\end{align*}
Applying \eqref{Weak2} again, we have
$$
\f{\ud}{\ud\rho}\vt_{\rho}(u,0)=-2\rho^{\al_p-n-1}\int_{\R^n}\left|\pa_{\rho}u+\f{2u}{(p-1)\rho}\right|^2\dot{\phi}_{\rho},
$$
and the result follows.
\end{proof}

\begin{rem}
Let $ r>0 $ and $ x\in\R^n $. Assume that $ u\in(H^1\cap L^{p+1})(B_{10r}(x)) $ is a stationary solution of \eqref{superequation}. We have another equivalent form of $ \vt_{\rho}(u,x) $ as follows.
\be
\begin{aligned}
\vt_{\rho}(u,x)&=\f{\rho^{\al_p-n}}{p+3}\int_{\R^n}\(\f{p-1}{2}|\na u|^2+\f{p-1}{p+1}|u|^{p+1}\)\phi_{x,\rho}\\
&\quad\quad-\f{2}{p+3}\f{\ud}{\ud\rho}\(\rho^{\al_p-n-1}\int_{\R^n}|u|^2\dot{\phi}_{x,\rho}\),\quad\rho\in(0,r).
\end{aligned}\label{Den2}
\ee
\end{rem}

For the proceeding proposition, we consider the domination of $ \theta_{\rho}(u,x) $ by $ \vt_{\rho}(u,x) $. In \cite{Pac93}, an analogous result was proved for the density \eqref{semidensi}. Regarding our new modified density, we need to adjust a little.

\begin{prop}\label{NonDegThePro}
Let $ r>0 $ and $ x\in\R^n $. Assume that $ u\in (H^1\cap L^{p+1})(B_{10r}(x)) $ is a stationary solution of \eqref{superequation}. Then
\be
\theta_{4\rho}(u,x)\leq C\vt_{\rho}(u,x),\label{NonDegThe}
\ee
for any $ 0<\rho<r $, where $ C>0 $ depends only on $ n $ and $ p $.
\end{prop}

Before the proof of this proposition, we need a useful lemma.

\begin{lem}[\cite{Eva18}, Theorem 2.10]\label{HauMeaIntBal}
Let $ f\in L_{\loc}^1(\R^n) $. If $ 0\leq s<n $, then
$$
\HH^s\(\left\{x\in\R^n:\limsup_{\rho\to 0^+}\(\rho^{-s}\int_{B_{\rho}(x)}|f|\)>0\right\}\)=0.
$$
\end{lem}

\begin{proof}[Proof of Proposition \ref{NonDegThePro}]
For simplicity, we define 
$$
g(r)=\rho^{\al_p-n}\int|u|^{p+1}\phi_{x,\rho}. 
$$
Since $ \al_p>2 $, we can apply Lemma \ref{HauMeaIntBal} to obtain that
$$
\HH^{n+2-\al_p}\(\left\{y\in B_{10r}(x):\limsup_{\rho\to 0^+}\(\rho^{\al_p-n-2}\int_{B_{\rho}(y)}|u|^2\)>0\right\}\)=0,
$$
which implies that for a.e. $ y\in B_{10\rho}(x) $, there holds
\be
\limsup_{\rho\to 0^+}\(\rho^{\al_p-n-2}\int_{B_{\rho}(y)}|u|^2\)=0.\label{0Con}
\ee
Since both sides of \eqref{NonDegThe} are continuous with respect to variables $ \rho $ and $ x $, we can assume that the point $ x $ satisfies the condition \eqref{0Con}. If not, we can apply the result of Lemma \ref{NonDegThePro} near $ x $ and use the continuity to conduct approximating arguments. Now, integrating both sides of the formula \eqref{Den2} from $ 0 $ to $ \rho $, we have
$$
(p-1)\int_0^{\rho}g(s)\ud s-2\rho^{\al_p-n-1}\int |u|^2\dot{\phi}_{x,\rho}\leq(p+3)\int_0^{\rho}\vt_s(u,x)\ud s.
$$
Using Definition \ref{defnofphi} and Proposition \ref{MonFor}, we obtain 
$$
(p-1)\int_0^{\rho}g(s)\ud s\leq(p+3)\int_0^{\rho}\vt_s(u,x)\ud s\leq (p+3)\rho\vt_{\rho}(u,x).
$$
By average arguments, there must be some $ \rho_0\in[\f{\rho}{2},\rho] $ such that
$$
g(\rho_0)\leq\f{8}{\rho}\int_0^{\rho}g(s)\ud s\leq C\vt_{\rho}(u,x),
$$
and then 
$$
g\(\f{\rho}{2}\)\leq Cg(\rho_0)\leq C(n,p)\vt_{\rho}(u,x).
$$
We conclude the result by the defition of $ \theta_{\rho}(u,x) $ and $ \vt_{\rho}(u,x) $.
\end{proof}

The following result is the well-known partial regularity estimate for stationary solutions of \eqref{superequation}. It was shown for positive solutions in Theorem 2 of \cite{Pac93}, and the general case can be dealt with by minor modifications.

\begin{prop}\label{smareg}
Let $ j\in\Z_{\geq 0} $, $ r>0 $, and $ x\in\R^n $. Assume that $ u\in(H^1\cap L^{p+1})(B_{2r}(x)) $ is a stationary solution of \eqref{superequation}. There exists $ \va_0>0 $, depending only on $ n $ and $ p $ such that if $ \theta_{2r}(u,x)\leq\va_0 $, then
$$
\sup_{B_r(x)}\(\sum_{i=0}^jr^i|D^iu|\)\leq C(\va_0)r^{-\f{2}{p-1}},
$$
where $ C(\va_0)>0 $ depends only on $ j,\Lda,n,p $ and $ C(\va_0)\to 0^+ $ as $ \va_0\to 0^+ $.
\end{prop}

Here we remark on the property that $ C(\va_0)\to 0^+ $ as $ \va_0\to 0^+ $. It is a direct consequence of the compactness arguments and properties in Proposition \ref{DefMea}.

Now, we introduce the unique continuation property for the weak solution of \eqref{superequation}, and will not use this in our proofs.

\begin{prop}[Unique continuation property]\label{Ucp}
Assume that $ u,v\in(H^1\cap L^{p+1})(B_1) $ are two weak solutions of \eqref{superequation}, which are smooth away from relatively closed sets $ S(u),S(v)\subset B_1 $ with $ \HH^{n-\al_p}(S(u)\cup S(v))<+\ift $. If $ u=v $ a.e. in an open subset of $ B_1 $, then $ u=v $ a.e. in $ B_1 $.
\end{prop}

To begin with, let us recall the Aronszajn generalization of Carleman's unique continuation theorem.

\begin{lem}[\cite{Aro57}, p.248]\label{Aron}
Assume that $ u\in C^{\ift}(B_1) $ satisfies $
|\Delta u|\leq C(|\na u|+|u|) $
for some constant $ C>0 $. If $ u=0 $ in an open subset of $ B_1 $, then $ u=0 $ in $ B_1 $.
\end{lem}

\begin{proof}[Proof of Proposition \ref{Ucp}]
We will follow the proof of similar continuation property of weakly harmonic maps in Proposition 16 of \cite{NV18}. We first assume that $ u $ and $ v $ are smooth. Suppose that the result is true for such a case. Since $ \HH^{n-\al_p}(S_0)<+\ift $, with $ S_0=S(u)\cup S(v) $, we have that $ B_1\backslash S_0 $ is connected, and can easily obtain that $ u=v $ in $ B_1\backslash S_0 $, which directly completes the proof. If $ u $ and $ v $ are smooth in $ B_1 $, we define $ w=u-v $ and deduce that
$$
|\Delta w|=|(|u|^{p-1}u-|v|^{p-1}v)|\leq C|u-v|=C(w,r)|w|,
$$
in $ B_r $ with $ 0<r<1 $, where we have used the mean value theorem and the fact that the function $ x\mapsto|x|^{p-1}x $ is $ C^1 $ if $ p>1 $. Applying Lemma \ref{Aron}, $ w=0 $ in $ B_1 $, and then $ u=v $ in $ B_1 $.
\end{proof}

\subsection{Cone splitting lemmas}

We will present several cone splitting lemmas and their corollaries for functions and measures. Intuitively speaking, these results show that the homogeneity for two different points implies the invariance with respect to the affine subspace generated by the two points. The cone splitting phenomenon is the essential cause of the stratification of singularities for various variational problems. Readers can refer to \S 4 of \cite{CN13b} for similar results. 

\begin{lem}[Cone splitting of functions]\label{ConSpl}
If $ u\in L_{\loc}^2(\R^n) $ is $ 0 $-symmetric at $ x_1,x_2\in\R^n $ with $ x_1\neq x_2 $, then $ u $ is $ 1 $-symmetric at $ x_1 $ with respect to $ \op{span}\{x_1-x_2\} $. 
\end{lem}
\begin{proof}
Up to a translation, we can assume that $ x_2=0 $. By using the $ 0 $-symmetry property at $ x_1 $ and $ 0 $, for any $ \lda>0 $ and a.e. $ y\in\R^n $, we have
\begin{align*}
u(y)&=\lda^{\f{2}{p-1}}u(\lda y)=\lda^{\f{2}{p-1}}u(\lda y-x_1+x_1)\\
&=\lda^{-\f{2}{p-1}}u(x_1+\lda^{-2}(\lda y-x_1))=u(y+(\lda-\lda^{-1})x_1).
\end{align*}
We can choose $ t=\lda-\lda^{-1} $ as any real number, so the result follows.
\end{proof}

\begin{cor}\label{CorSpl}
Let $ k\in\{0,1,2,...,n-1\} $. If $ u\in L_{\loc}^2(\R^n) $ is $ k $-symmetric with respect to $ V\in\bG(n,k) $, and is $ 0 $-symmetric at $ x\notin V $, then $ u $ is $ (k+1) $-symmetric with respect to $ \op{span}\{x,V\} $.
\end{cor}
\begin{proof}
Applying Lemma \ref{ConSpl}, we conclude that $ u $ is invariant with respect to $ \op{span}\{x\} $. Combined with the fact that $ u $ is also invariant with respect to $ V $, we now complete the proof.
\end{proof}

\begin{lem}[Cone splitting of Radon measures]\label{ConSplm}
If $ \mu\in\M(\R^n) $ is $ 0 $-symmetric at $ x_1,x_2\in\R^n $ with $ x_1\neq x_2 $, then $ \mu $ is $ 1 $-symmetric at $ x_1 $ with respect to $ \op{span}\{x_1-x_2\} $. 
\end{lem}
\begin{proof}
Without loss of generality, we assume that $ x_2=0 $. For any $ \lda>0 $ and $ \mu $-measurable set $ A\subset\R^n $, we have
\begin{align*}
\mu(A)&=\lda^{\al_p-n}\mu(\lda A)=\lda^{\al_p-n}\mu(\lda A-x_1+x_1)\\
&=\lda^{-\al_p+n}\mu(x_1+\lda^{-2}(\lda A-x_1))=\mu(A+(\lda-\lda^{-1})x_1).
\end{align*}
The result follows by choosing $ t=\lda-\lda^{-1} $ as any real number.
\end{proof}

Analogous to Corollary \ref{CorSpl}, we apply Lemma \ref{ConSplm} to obtain the following result. 

\begin{cor}\label{CorSplm}
Let $ k\in\{0,1,2,...,n-1\} $. If $ \mu\in\M(\R^n) $ is $ k $-symmetric with respect to $ V\in\bG(n,k) $, and is $ 0 $-symmetric at $ x\notin V $, then $ \mu $ is $ (k+1) $-symmetric with respect to $ \op{span}\{x,V\} $.
\end{cor}

\subsection{Defect measures, tangent functions, and tangent measures} In this section, we will analyze the defect measures that arise when a sequence of stationary solutions to \eqref{superequation} converges, and we will use them to examine tangent functions and tangent measures. Furthermore, we will establish Proposition \ref{tangentfm}, Theorem \ref{Snminusalp}, Proposition \ref{proprela}, and Corollary \ref{correla}.

\begin{prop}\label{DefMea}
Assume that $ \{u_i\}\subset(H^1\cap L^{p+1})(B_1) $ is a sequence of stationary solutions of \eqref{superequation} with $ \theta_1(u,0)\leq\Lda $, satisfying
\begin{align*}
u_i&\wc u\text{ weakly in }(H^1\cap L^{p+1})(B_1),\\
\mu_{1,i}:=\f{|\na u_i|^2}{2}\ud y&\wc^*\mu_1:=\f{|\na u|^2}{2}\ud y+\nu_1\text{ in }\M(B_1),\\
\mu_{2,i}:=\f{|u_i|^{p+1}}{p+1}\ud y&\wc^*\mu_2:=\f{|u|^{p+1}}{p+1}\ud y+\nu_2\text{ in }\M(B_1),
\end{align*}
where $ \nu_1,\nu_2\in\M(B_1) $ are called defect measures. Define
\begin{align*}
\mu&:=(p-1)(\mu_1+\mu_2),\\
\nu&:=(p-1)(\nu_1+\nu_2).
\end{align*}
The following properties hold.
\begin{enumerate}[label=$(\theenumi)$]
\item There exists a relatively closed set $ \Sg\subset B_1 $ with 
$$
\HH^{n-\al_p}(\Sg\cap K)<+\ift 
$$
for any compact set 
$ K\subset B_1 $ such that $ u\in C_{\loc}^{\ift}(B_1\backslash\Sg) $
and $ u_i\to u $ in $ C_{\loc}^{\ift}(B_1\backslash\Sg) $. More precisely,
$$
\Sg:=\bigcap_{0<r<1}\left\{x\in B_1:\liminf_{i\to+\ift}r^{\al_p-n}m_{u_i}(B_r(x))\geq\f{\va_0}{2}\right\},
$$
where $ \va_0>0 $ is the constant given in \eqref{smareg}.
\item $ 2\nu_1=(p+1)\nu_2 $ and $ \mu\llcorner\Sg=\nu $. 
\item $ u $ is a weak solution of \eqref{superequation} in $ B_1 $, and 
$$
\Sg=\supp(\nu)\cup\sing(u).
$$
\item For $ \HH^{n-\al_p} $-a.e. $ x\in\Sg $, the limit 
$$
\Theta(\nu,x):=\lim_{r\to 0^+}r^{\al_p-n}\nu(B_r(x))
$$
exists and $ \Theta(\nu,x)\in(\f{\va_0}{2},C\Lda) $, where $ C>0 $ depends only on $ n $ and $ p $.
\item If $ \al_p $ is not an integer, then $ \nu=0 $, and $ u_i\to u $ strongly in $ (H^1\cap L^{p+1})(B_1) $. Moreover, $ u $ is a stationary solution of \eqref{superequation}. If $ \al_p $ is an integer, then 
$$
\nu=\Theta(\nu,x)\HH^{n-\al_p}\llcorner\Sg, $$
and $ \Sg $ is an $ (n-\al_p) $-rectifiable set.
\end{enumerate}
\end{prop}
\begin{proof}
The proof follows from results in \S 3-4 of \cite{WW15}.
\end{proof}

Indeed, analogous to those given for stationary harmonic maps as in \cite{Lin99,LW08}, there are some additional results on the convergence of singular sets of $ u_i $ in the sense of Hausdorff distance. However, such results are not used in this paper directly or to clarify some insights of the proofs, and we do not present them here. Let us show Proposition \ref{tangentfm} by using the above convergence result of stationary solutions of \eqref{superequation}.

\begin{proof}[Proof of Proposition \ref{tangentfm}]
By \eqref{Scaleinv}, we have $ \vt_R(T_{x,r}u,0)=\vt_{Rr}(u,x) $, for any $ R>0 $, as $ r>0 $ is sufficiently small. Given Proposition \ref{MonFor} and \ref{NonDegThePro}, $ T_{x,r}u $ is uniformly bounded in $ H_{\loc}^1\cap L_{\loc}^{p+1} $, and consequently, up to a subsequence, we can assume that $
T_{x,r_i}u\wc v $ weakly in $ (H_{\loc}^1\cap L_{\loc}^{p+1})(\R^n) $ and $ T_{x,r_i}m_u\wc^*\mu $ in $ \M(\R^n) $. This convergence implies that $ v\in\cF_x(u) $ and $ \mu\in\M_x(u) $. Indeed, $ (v,\mu) $ is a tangent pair of $ u $ at $ x $. If $ x $ is a regular point of $ u $, we can choose $ C>0 $ and $ r>0 $ such that $ \sup_{B_r(x)}(|\na u|+|u|)\leq C $. Consequently, the tangent function and the tangent measure are both zero.
\end{proof}

To further study the properties of tangent functions and tangent measures, we assume that $ u\in(H^1\cap L^{p+1})(B_{40}) $ is a stationary solution of \eqref{superequation}. Let $ x\in B_{40} $ and $ r_i\to 0^+ $ such that 
\be
\begin{aligned}
u_i:=T_{x,r_i}u&\wc v\text{ weakly in }(H_{\loc}^1\cap L_{\loc}^{p+1})(\R^n),\\
m_{u_i}=T_{x,r_i}m_u&\wc^*\mu\text{ in }\M(\R^n).
\end{aligned}\label{uimunuidef}
\ee
As of Proposition \ref{DefMea}, up to a subsequence, we can also assume that there are two defect measures $ \nu_1,\nu_2\in\M(\R^n) $ such that
\be
\begin{aligned}
\f{|\na u_i|^2}{2}\ud y&\wc^*\mu_1=\f{|\na v|^2}{2}\ud y+\nu_1,\\
\f{|u_i|^{p+1}}{p+1}\ud y&\wc^*\mu_2=\f{|v|^{p+1}}{p+1}\ud y+\nu_2,
\end{aligned}\label{mu1jmu2j}
\ee
in $ \M(\R^n) $. We have $ \mu=(p-1)(\mu_1+\mu_2) $.
For convenience, we also define
\be
\nu=(p-1)(\nu_1+\nu_2).\label{nu1nu2nu}
\ee

\begin{lem}\label{tanfunmea}
If $ u,\{u_i\},v,\mu $, and $ \nu $ are given in \eqref{uimunuidef} and \eqref{nu1nu2nu}, then $ v,\mu,\nu $ are all $ 0 $-symmetric.
\end{lem}

\begin{proof}
We will employ the arguments presented in \cite{Mos03}, and refer to readers the proofs in \S 4 of \cite{LW08} for additional materials. For any $ i\in\Z_+ $, $ u_i $ remains a stationary solution to \eqref{superequation}. Define $ \R^{n\times n} $-valued Radon measures $ (\mu_{1,i}^{\beta\ga})_{\beta,\ga=1}^n $ as 
$$
(\mu_{1,i}^{\beta\ga}):=\(\f{\pa_{\beta}u_i\pa_{\ga}u_i}{2}\ud y\)\in\M(r_i^{-1}(B_{40}-x),\R^{n\times n}).
$$
Note that $ \mu_{1,i}=\sum_{\beta=1}^n\mu_{1,i}^{\beta\beta} $. By Cauchy's inequality, 
$$ \sup_{i\in\Z_+}\(\sum_{\beta,\ga=1}^n|\mu_{1,i}^{\beta\ga}(B_R)|\)\leq C(\Lda,n,p,R),
$$
for any $ R>0 $. Up to a subsequence, we assume that 
\be 
(\mu_{1,i}^{\beta\ga})\wc^*(\mu_1^{\beta\ga}):=\(\f{\pa_{\beta}v\pa_{\ga}v}{2}\ud y\)+(\nu_1^{\beta\ga})\text{ in }\M(\R^n,\R^{n\times n}).\label{mu1j}
\ee
Extracting a subsequence if necessary, it follows from \eqref{mu1j} that 
\be
\f{1}{2}|y\cdot\na u|^2\ud y\wc^*\f{1}{2}|y\cdot\na v|^2\ud y+y_{\beta}y_{\ga}\nu_1^{\beta\ga},\label{ybetayga2}
\ee
where $ y_{\beta}y_{\ga}\nu_1^{\beta\ga}\in\M(\R^n) $. In view of Proposition \ref{MonFor}, we have, for any $ 0<s<r $, 
\be
\begin{aligned}
&\vt_{r_ir}(u,x)-\vt_{r_is}(u,x)=\vt_r(u_i,0)-\vt_s(u_i,0)\\
&\quad\quad=-2\int_{s}^{r}\(\rho^{\al_p-n-1}\int_{\R^n}\left|\f{y\cdot\na u_i}{\rho}+\f{2u_i}{(p-1)\rho}\right|^2\dot{\phi}_{0,\rho}\)\ud\rho.
\end{aligned}\label{Therj}
\ee
In particular, limits $ \lim_{i\to+\ift}\vt_{r_ir}(u,x),\vt_{r_is}(u,x) $ exist and are the same. Combined with \eqref{ybetayga2}, we take $ i\to+\ift $ for both sides of \eqref{Therj} and deduce that
\begin{align*}
0&\leq-2\int_{s}^{r}\(\rho^{\al_p-n-1}\int_{\R^n}\left|\f{y\cdot\na v}{\rho}+\f{2v}{(p-1)\rho}\right|^2\dot{\phi}_{0,\rho}\)\ud\rho\\
&\quad-2\int_s^r\(\rho^{\al_p-n-3}\int_{\R^n}y_{\beta}y_{\ga}\dot{\phi}_{0,\rho}\ud\nu_1^{\beta\ga}\)\ud\rho\\
&=\lim_{i\to+\ift}\[-2\int_{s}^{r}\(\rho^{\al_p-n-1}\int_{\R^n}\left|\f{y\cdot\na u_i}{\rho}+\f{2u_i}{(p-1)\rho}\right|^2\dot{\phi}_{0,\rho}\)\ud\rho\]\\
&=\lim_{i\to+\ift}(\vt_{r_ir}(u,x)-\vt_{r_is}(u,x))=0,
\end{align*}
for any $ 0<s<r $. Consequently, for a.e. $ y\in\R^n $,
\be
y\cdot\na v+\f{2v}{p-1}=0,\label{homequ}
\ee
and
\be
y_{\beta}y_{\ga}\nu_1^{\beta\ga}=0\text{ in }\M(\R^n).\label{ybetayga}
\ee
By Remark \ref{RemHom}, $ v $ is $ 0 $-symmetric. To complete the proof, we must prove that $ \mu $ and $ \nu $ are all $ 0 $-symmetric. Due to \eqref{uimunuidef}, \eqref{mu1jmu2j}, and \eqref{nu1nu2nu}, $
\mu=m_v+\nu $. Consequently, we only need to show that $ \nu $ is $ 0 $-symmetric. We claim that for any $ \vp\in C_0^{\ift}(\R^n) $, 
\be
\int_{\R^n}y_{\ga}\pa_{\beta}\vp\ud\nu_1^{\beta\ga}=0.\label{phidnu}
\ee
For any $ i\in\Z_+ $, define a sequence of $ \R^{n\times n} $-valued Radon measures $ (\nu_{1,i}^{\beta\ga})_{\beta,\ga=1}^n $ with 
$$
(\nu_{1,i}^{\beta\ga}):=\(\f{1}{2}\pa_{\beta}(u_i-v)\pa_{\ga}(u_i-v)\ud y\)\in\M(r_i^{-1}(B_{40}-x),\R^{n\times n}).
$$
In view of previous formula \eqref{mu1j}, we have $
(\nu_{1,i}^{\beta\ga})\wc^*(\nu_1^{\beta\ga}) $ in $ \M(\R^n,\R^{n\times n}) $. By Cauchy's inequality, we deduce that
\begin{align*}
\(\int y_{\ga}\pa_{\beta}\vp\ud\nu_{1,i}^{\beta\ga}\)^2\leq\( \int\pa_{\beta}\vp\pa_{\ga}\vp\ud\nu_{1,i}^{\beta\ga}\)\(\int_{\supp\vp} y_{\beta}y_{\ga}\ud\nu_{1,i}^{\beta\ga}\),
\end{align*}
which directly implies \eqref{phidnu} after letting $ i\to+\ift $ and taking \eqref{ybetayga} into account. Given $ \psi\in C_0^{\ift}(\R^n) $, we define $ \psi_{\lda}(y)=\psi(\f{y}{\lda}) $. Applying the stationary condition \eqref{StaCon} to the function $ u_i $ with the test vector field $ Y(y)=\psi_{\lda}(y)y $, we have
\begin{align*}
&\f{n-2}{2}\int|\na u_i|^2\psi_{\lda}+\f{1}{2}\int|\na u_i|^2(y\cdot\na\psi_{\lda})-\f{1}{p+1}\int|u_i|^{p+1}(y\cdot\na\psi_{\lda})\\
&\quad\quad-\f{n}{p+1}\int|u_i|^{p+1}\psi_{\lda}-\int(\na u_i\cdot\na\psi_{\lda})(y\cdot\na u_i)=0.
\end{align*}
In view of \eqref{mu1jmu2j} and \eqref{mu1j}, we can take $ i\to+\ift $ and obtain 
\be
\begin{aligned}
&\f{n-2}{2}\int|\na v|^2\psi_{\lda}+\f{1}{2}\int|\na v|^2(y\cdot\na\psi_{\lda})-\f{1}{p+1}\int|v|^{p+1}(y\cdot\na\psi_{\lda})\\
&\quad\quad-\f{n}{p+1}\int|v|^{p+1}\psi_{\lda}-\int(\na v\cdot\na\psi_{\lda})(y\cdot\na v)+(n-2)\int\psi_{\lda}\ud\nu_1\\
&\quad\quad+\int(y\cdot\na\psi_{\lda})(\ud\nu_1-d\nu_2)-n\int\psi_{\lda}\ud\nu_2-\int y_{\ga}\pa_{\beta}\psi_{\lda}\ud\nu_1^{\beta\ga}=0.
\end{aligned}\label{n2long}
\ee
Since $ v $ is $ 0 $-symmetric, there hold
\begin{align*}
(\al_p-n)\int|\na v|^2\psi_{\lda}&=\int|\na v|^2(y\cdot\na\psi_{\lda}),\\
(\al_p-n)\int|v|^{p+1}\psi_{\lda}&=\int|v|^{p+1}(y\cdot\na\psi_{\lda}),
\end{align*}
for any $ \lda>0 $. It follows from \eqref{n2long} that
\be
\begin{aligned}
&\f{2}{p-1}\(\int|\na v|^2\psi_{\lda}-\int|v|^{p+1}\psi_{\lda}\)+\(\f{(p+1)(n-2)}{2}-n\)\int\psi_{\lda}\ud\nu_2\\
&\quad\quad+\int(\na v\cdot\na\psi_{\lda})(y\cdot\na v)+\f{p-1}{2}\int(y\cdot\na\psi_{\lda})\ud\nu_2=0,
\end{aligned}\label{2p1}
\ee
where we have used the second property of Proposition \ref{DefMea} and the claim \eqref{phidnu}. Employing the third property of the same proposition, we can test \eqref{WeakCon} by $ v\psi_{\lda} $ and obtain that
$$
\int|\na v|^2\psi_{\lda}+\int v(\na v\cdot\na\psi_{\lda})=\int|v|^{p+1}\psi_{\lda}.
$$
This, together with \eqref{homequ} and \eqref{2p1}, implies that
$$
(\al_p-n)\int\psi_{\lda}\ud\nu_2-\int(y\cdot\na\psi_{\lda})\ud\nu_2=0.
$$
With this property, we have
$$
\f{\ud}{\ud\lda}\(\int\psi \ud T_{0,\lda}\nu_2\)=\f{\ud}{\ud\lda}\(\lda^{\al_p-n}\int\psi_{\lda}\ud\nu_2\)=0,
$$
which completes the proof by \eqref{nu1nu2nu} and the arbitrariness of $ \psi $.
\end{proof}

Assume that $ u,\{u_i\},v,\mu $, and $ \nu $ are the same as those given in Lemma \ref{tanfunmea}. We define
\be
\vt_{r}(\mu,y):=\vt_r(v,y)+\f{1}{p+3}\int\phi_{y,r}\ud\nu,\label{mudesi1}
\ee
for $ r>0 $ and $ y\in\R^n $. Here $ \vt_r(\mu,y) $ is the density obtained by taking limit from \eqref{DenDef2}. Intuitively, it can be regarded as the density for the pair $ (v,\mu) $. Indeed, in view of \eqref{uimunuidef}, \eqref{nu1nu2nu}, and the second property of Proposition \ref{DefMea}, we see that
\be
\vt_{r}(\mu,y)=\lim_{i\to+\ift}\vt_r(u_i,y).\label{vtlim}
\ee
This, together with \eqref{Den2}, implies that $ \vt_r(\mu,y) $ has the following equivalent form
\be
\vt_{r}(\mu,y)=\f{r^{\al_p-n}}{(p+3)}\int_{\R^n}\phi_{y,r}\ud\mu-\f{2}{(p+3)}\f{\ud}{\ud r}\(r^{\al_p-n-2}\int_{\R^n}|v|^2\dot{\phi}_{y,r}\).\label{thetamuya}
\ee
For the notions above, we have some basic observations.

\begin{lem}\label{5prop}
The following properties hold.
\begin{enumerate}[label=$(\theenumi)$]
\item $ \vt_r(\mu,y) $ is nondecreasing with respect to $ r>0 $. In particular, the limit
$$
\vt(\mu,y):=\lim_{r\to 0^+}\vt_r(\mu,y)
$$
exists for any $ y\in\R^n $.
\item $ x\in\sing(u) $ if and only if $ \vt(\mu,0)>0 $.
\item For any $ r>0 $, $ \vt_r(\mu,0)=\vt(u,x) $, where
\be
\vt(u,x):=\lim_{r\to 0^+}\vt_r(u,x).\label{uxthetalim}
\ee
\item If $ y\in\R^n $, then $ \vt(\mu,y)\leq\vt(\mu,0) $.
\item The set 
$$
\Sg(\mu):=\{y\in\R^n:\vt(\mu,y)=\vt(\mu,0)\} 
$$
is a subspace of $ \R^n $, and $ v,\mu,\nu $ are invariant with respect to $ \Sg(\mu) $.
\item $ \vt(\mu,y) $ is upper semicontinuous with respect to the variable $ y $, namely if $ y_i\to y_0 $, then 
$$
\limsup_{i\to+\ift}\vt(\mu,y_i)\leq\vt(\mu,y_0).
$$
\end{enumerate}
\end{lem}
\begin{proof}
We note that the first property follows directly from \eqref{mudesi1}, \eqref{vtlim}, and Proposition \ref{MonFor}. More precisely, we have the monotonicity formula
\be
\begin{aligned}
&\vt_r(\mu,y)-\vt_s(\mu,y)\\
&\quad\quad=-2\int_s^r\(\rho^{\al_p-n-1}\int_{\R^n}\left|\f{(z-y)\cdot\na v}{\rho}+\f{2v}{(p-1)\rho}\right|^2\dot{\phi}_{y,\rho}\ud z\)\ud\rho\\
&\quad\quad\quad-2\int_s^r\(\rho^{\al_p-n-3}\int_{\R^n}(z-y)_{\beta}(z-y)_{\ga}\dot{\phi}_{y,\rho}\ud\nu_1^{\beta\ga}(z)\)\ud\rho,
\end{aligned}\label{Therjmu}
\ee
for any $ 0<s<r $, where $ (\nu_1^{\beta\ga})_{\beta,\ga=1}^n\in\M(\R^n,\R^{n\times n}) $ is given by \eqref{mu1j}. To prove the second property, we first fix $ x\notin\sing(u) $. Given \eqref{thevt}, we have
$$
\vt_r(u_i,0)\leq C(n,p)\max\{\theta_{10r}(u_i,0),r^2(\theta_{10r}(u_i,0))^{\f{2}{p+1}}\},
$$
for any $ r>0 $. Taking $ i\to+\ift $, it follows that
\be
\vt_r(\mu,0)\leq C(n,p)\max\{r^{\al_p-n}\mu(\ol{B}_{10r}),r^{\al_p+2-n}(\mu(\ol{B}_{10r}))^{\f{2}{p+1}}\}.\label{muovt}
\ee
Since $ x $ is a regular point of $ u $, $ |\na u(y)|^2+|u(y)|^{p+1}\leq C $, for any $ y\in B_R(x) $ with some $ R>0 $ and a constant $ C>0 $. This yields that
\begin{align*}
0&\leq\mu(B_1)\leq\lim_{i\to+\ift}\[r_i^{\al_p-n}\int_{B_{r_i}(x)}\(\f{p-1}{2}|\na u|^2+\f{p-1}{p+1}|u|^{p+1}\)\]=0.
\end{align*}
Combining with \eqref{muovt}, we see that $ \vt_r(\mu,0)=0 $ for $ 0<r<\f{1}{20} $. By the first property, this implies that $ \vt(\mu,0)=0 $. On the other hand, we assume that $ x\in\sing(u) $. Applying \eqref{NonDegThe} and Proposition \ref{smareg}, we have
$$
\vt_1(u_i,0)=\vt_{r_i}(u,x)\geq C(n,p)\theta_{4r_i}(u,x)>\f{C(n,p)\va_0}{2}.
$$
Due to Lemma \ref{tanfunmea}, the second property follows by letting $ i\to+\ift $. By the convergence of $ \{u_i\} $, for any $ r>0 $,
$$
\vt_r(\mu,0)=\lim_{i\to+\ift}\vt_r(u_i,0)=\lim_{i\to+\ift}\vt_{r_ir}(u,x)=\vt(u,x).
$$
Now, we conclude the third property. As to the fourth one, we adopt similar arguments in the proof of Lemma 3.2 of \cite{Wan12}. For fixed $ y\in\R^n $, it follows from \eqref{thetamuya} and \eqref{Therjmu} that
\be
\begin{aligned}
\vt_s(\mu,y)&\leq\lim_{r\to+\ift}\vt_r(\mu,y)\\
&\leq\lim_{r\to+\ift}\f{1}{r}\int_r^{2r}\vt_{\rho}(\mu,y)\ud\rho\\
&=\lim_{r\to+\ift}\[\f{(p-1)}{(p+3)r}\int_{r}^{2r}\(\rho^{\al_p-n}\int_{\R^n}\phi_{y,\rho}\ud\mu\)\ud\rho\right.\\
&\quad\left.-\f{2}{(p+3)r}\((2r)^{\al_p-n-1}\int_{\R^n}|v|^2\dot{\phi}_{y,2r}-r^{\al_p-n-1}\int_{\R^n}|v|^2\dot{\phi}_{y,r}\)\],
\end{aligned}\label{vtmuy0}
\ee
for any $ s>0 $. By Definition \ref{defnofphi} and simple calculations,
$$
|\phi_{y,\rho}(z)-\phi_{0,\rho}(z)|+|\dot{\phi}_{y,\rho}(z)-\dot{\phi}_{0,\rho}(z)|\leq\f{C(n)|y|}{\rho}\chi_{B_{10(\rho+|y|)}}(z),
$$
for any $ \rho>0 $ and a.e. $ z\in\R^n $. This, together with Lemma \ref{tanfunmea}, implies that
\begin{align*}
&\f{1}{r}\int_{r}^{2r}\(\rho^{\al_p-n}\int_{\R^n}\phi_{y,\rho}\ud\mu\)\ud\rho\\
&\quad\leq\f{1}{r}\int_{r}^{2r}\(\rho^{\al_p-n}\int_{\R^n}\phi_{0,\rho}\ud\mu\)\ud\rho+\f{C|y|}{r^2}\int_r^{2r}\(\f{\rho}{\rho+|y|}\)^{\al_p-n}\f{\mu(B_{10(\rho+|y|)})}{(10(\rho+|y|))^{n-\al_p}}\ud\rho\\
&\quad\leq\f{1}{r}\int_{r}^{2r}\(\rho^{\al_p-n}\int_{\R^n}\phi_{0,\rho}\ud\mu\)\ud\rho+\f{C|y|}{r^2}\int_r^{2r}\(\f{r}{r+|y|}\)^{\al_p-n}\mu(B_1)\ud\rho,
\end{align*}
and
\begin{align*}
&\left|s^{\al_p-n-1}\(\int_{\R^n}|v|^2\dot{\phi}_{y,s}-\int_{\R^n}|v|^2\dot{\phi}_{0,s}\)\right|\leq C|y|s^{\al_p-n-2}\int_{B_{6(s+|y|)}}|v|^2\\
&\quad\quad=C|y|\(\f{s}{s+|y|}\)^{\al_p-n-2}\int_{B_1}|v|^2,\quad\text{where }s=r,2r.
\end{align*}
Combining \eqref{vtmuy0}, we obtain 
\be
\vt(\mu,y)\leq\vt_s(\mu,y)\leq\vt(\mu,0),\quad\text{for any }s>0.\label{muys}
\ee
Due to \eqref{Therjmu} and \eqref{muys}, we can take $ s\to 0^+ $ to deduce that
\begin{align*}
\vt(\mu,0)-\vt(\mu,y)&\geq-2\int_0^{r}\(\rho^{\al_p-n-1}\int_{\R^n}\left|\f{(z-y)\cdot\na v}{\rho}+\f{2v}{(p-1)\rho}\right|^2\dot{\phi}_{y,\rho}\ud z\)\ud\rho\\
&\quad-2\int_0^r\(\rho^{\al_p-n-3}\int_{\R^n}(z-y)_{\beta}(z-y)_{\ga}\dot{\phi}_{y,\rho}\ud\nu_1^{\beta\ga}(z)\)\ud\rho,
\end{align*}
for any $ r>0 $. If $ \vt(\mu,0)=\vt(\mu,y) $ with $ y\neq 0 $, then the right hand side of the above is $ 0 $. Applying analogous arguments to those in the proof of Lemma \ref{tanfunmea}, we have that $ v,\mu,\nu $ are all $ 0 $-symmetric at $ y $. In view of Lemmas \ref{ConSpl} and \ref{ConSplm}, $ v,\mu,\nu $ are invariant with respect to $ \op{span}\{y\} $. Combined with results in Corollaries \ref{CorSpl} and \ref{CorSplm}, this implies that $ \Sg(\mu) $ is a subspace of $ \R^n $ and $ v,\mu,\nu $ are also invariant with respect to $ \Sg(\mu) $. Finally, the upper semicontinuity of $ \vt(\mu,\cdot) $ follows from the continuity of $ \vt_r(\mu,\cdot) $ for any $ r>0 $ by taking $ r\to 0^+ $. 
\end{proof}

By the previous analysis, analogous to \cite{Lin99}, for $ u\in(H^1 \cap L^{p+1})(B_{40}) $ being a stationary solution of \eqref{superequation}, we can define  
$$
\Sg^k(u):=\{x\in\sing(u):\dim(\Sg(\mu))\leq k\text{ for any tangent measure }\mu\text{ of }u\text{ at }x\},
$$
where $ k\in\{0,1,2,...,n\} $. Furthermore, we have the following result.

\begin{lem}\label{lemcut}
If $ u\in (H^1\cap L^{p+1})(B_{40}) $ is a stationary solution of \eqref{superequation}, then 
$$
\Sg^0(u)\subset\Sg^1(u)\subset...\subset\Sg^{n-\lceil\al_p\rceil}(u)=\sing(u).
$$
\end{lem}

Before we show this lemma, for the sake of simplicity, in the remainder of this paper, we define
\be
k_{n,p}:=\left\{\begin{aligned}
&n-\lfloor\al_p\rfloor&\text{ if }\al_p\notin\Z_+,\\ 
&n-\al_p+1&\text{ if }\al_p\in\Z_+.
\end{aligned}\right.\label{kp}
\ee

\begin{lem}\label{symproen}
If $ \mu\in\M(\R^n) $ is $ k_{n,p} $-symmetric, then $ \mu=0 $ in $ \R^n $.
\end{lem}

\begin{proof}
Assume that $ \mu $ is $ k $-symmetric with respect to $ V\in\bG(n,k) $, which is not a zero measure. Up to a rotation, we can assume that $ V=\R^k\times\{0\}^{n-k} $. Simple geometric observation implies that $ B_R^k\times B_R^{n-k} $ contains a finite number of balls $ \{B_1(y_i)\}_{i=1}^N $, with $ N\sim R^k $ such that $ y_i\in\R^k\times\{0\}^{n-k} $ and $ B_1(y_i)\cap B_1(y_j)=\emptyset $ for any $ i,j\in\{1,2,...,N\}$ with $ i\neq j $. Owing to the $ k $-symmetry of $ \mu $ with respect to $ V $, we have
$$
\mu(B_1)=R^{\al_p-n}\mu(B_R)\geq R^{\al_p-n}\sum_{i=1}^N\mu(B_1(y_i))\sim R^{k+\al_p-n}\mu(B_1).
$$
Letting $ R\to+\ift $, we see that $ k\leq n-\al_p $, so if $ k>n-\al_p $, then $ \mu=0 $, which implies the result of the lemma.
\end{proof}

\begin{proof}[Proof of Lemma \ref{lemcut}]
Let $ x\in B_{40} $. Assume that $ \mu $ is a tangent measure of $ u $ at $ x $. By Lemma \ref{tanfunmea} and the second and fifth properties of Lemma \ref{5prop}, $ \mu $ is a nonzero measure and is $ k $-symmetric with respect to $ \Sg(\mu) $, where $ k=\dim(\Sg(\mu)) $. Applying Lemma \ref{symproen}, we obtain that $ k\leq n-\lceil\al_p\rceil $.
\end{proof}

Next, we consider the estimates of Hausdorff dimensions for $ \{\Sg^k(u)\} $.

\begin{prop}\label{hausdim}
If $ u\in (H^1\cap L^{p+1})(B_{40}) $ is a stationary solution of \eqref{superequation}, then 
$$
\dim_{\HH}(\Sg^k(u))\leq k 
$$
for any $ k\in\{0,1,2,...,n-\lceil\al_p\rceil\} $.
\end{prop}

We will prove this proposition by applying the standard arguments used for proofs of similar results concerning harmonic maps. To begin with, we need some preparations.

\begin{lem}\label{deltaaplem}
Let $ k\in\{0,1,2,...,n-\lceil\al_p\rceil\} $. Assume that $ u\in (H^1\cap L^{p+1})(B_{40}) $ is a stationary solution of \eqref{superequation}. For any $ y\in \Sg^k(u) $ and $ \delta>0 $, there exists $ \va>0 $, depending only on $ \delta,u $, and $ y $ such that for any $ \rho\in(0,\va] $, there holds 
$$
\rho^{-1}(\{x\in B_\rho(y):\vt(u,x)\geq\vt(u,y)-\va\}-y)\subset\{x\in\R^n:\dist(x,V)<\delta\}, 
$$
for some $ V\in\bG(n,k) $.
\end{lem}
\begin{proof}
Assume that the result is not true. There exist $ \delta_0>0 $, $ x_0\in\Sg^k(u) $, and sequences $ \{\rho_i\},\{\va_i\}\subset(0,1) $ with $ \rho_i\to 0^+ $ and $ \va_i\to 0^+ $ such that
\be
\begin{aligned}
&\rho_i^{-1}(\{y\in B_{\rho_i}(x_0):\vt(u,y)\geq\vt(u,x_0)-\va_i\}-x_0)\\
&\quad\quad=\{y\in B_1:\vt(T_{x_0,\rho_i}u,y)\geq\vt(u,x_0)-\va_i\}\\
&\quad\quad\not\subset\{y\in\R^n:\dist(y,V)<\delta\},
\end{aligned}\label{subcon}
\ee
for any $ V\in\bG(n,k) $, where we have used \eqref{Scaleinv} for the equality. Up to a subsequence, we can assume that $ u_i:=T_{x_0,\rho_i}u\wc v $ weakly in $ (H_{\loc}^1\cap L_{\loc}^{p+1})(\R^n) $, and $ m_{u_i}\wc^*\mu\text{ in }\M(\R^n) $. In particular, $ (v,\mu) $ is a tangent pair of $ u $ at $ x_0 $. By the third property of Lemma \ref{5prop}, we have $ \vt(\mu,0)=\vt(u,x_0) $. Since $ x_0\in \Sg^k(u) $, we obtain $ \dim(\Sg(\mu))\leq k $, and can choose $ V_0\in\bG(n,k) $ containing $ \Sg(\mu) $. By the fifth property of Lemma \ref{5prop}, $ \vt(\mu,\cdot) $ is upper semicontinuous, and then for any $ y\in B_1\backslash\Sg(\mu) $, we have $ \vt(\mu,y)<\vt(\mu,0) $. This, together with the sixth property of Lemma \ref{5prop}, implies that there exists $ \xi_0>0 $ such that 
\be
\vt(\mu,y)<\vt(\mu,0)-\xi_0,\quad\text{for any }y\in\ol{B}_1\text{ with }\dist(y,V_0)\geq\delta.\label{vtmuxvt}
\ee
We claim, for $ i\in\Z_+ $ sufficiently large,
\be
\{y\in B_1:\vt(u_i,y)\geq\vt(\mu,0)-\xi_0\}\subset\{y\in\R^n:\dist(y,V_0)<\delta\}.\label{claimvt}
\ee
If the claim is not true, then there exists $ y_i\to y_0 $ in $ \ol{B}_1 $ such that $ \vt(u_i,y_i)\geq\vt(\mu,0)-\xi_0 $ with $ \dist(y_i,V_0)\geq\delta $. Applying Proposition \ref{MonFor}, there holds $ \vt_r(u_i,y_i)\geq\vt(\mu,0)-\xi_0 $, for any $ r>0 $. It follows by letting $ i\to+\ift $ that $ \vt(\mu,0)-\xi_0\leq\vt_r(\mu,y_0) $, for any $ r>0 $. Taking $ r\to 0^+ $, this implies that $
\vt(\mu,0)-\xi_0\leq\vt(\mu,y_0) $, which contradicts to \eqref{vtmuxvt}. Therefore, we have proved the claim \eqref{claimvt}. Consequently, \eqref{subcon} cannot be true for any $ V\in\bG(n,k) $.
\end{proof}

\begin{defn}\label{deltaj}
Let $ \delta>0 $ and $ j\in\{0,1,2,...,n\} $. A set $ A\subset\R^n $ is said to satisfy the $ (\delta,j) $-approximation property if there is $ \rho_0>0 $ such that for any $ y\in A $ and $ \rho\in(0,\rho_0] $, there holds 
$$
(\rho^{-1}(A-y))\cap B_1\subset\{x\in\R^n:\dist(x,V)<\delta\}, 
$$
for some $ V\in\bG(n,j) $.
\end{defn}

\begin{lem}[\cite{GM05}, Lemma 10.38]\label{lemhau}
Let $ j\in\{0,1,2,...,n\} $. There is a function $ \beta:(0,+\ift) \to(0,+\ift)$ with $ \lim_{\delta\to 0^{+}}\beta(\delta)=0 $ such that if $ \delta>0 $ and $ A\subset\R^n $ satisfies the $ \delta $-approximation property above, then $ \HH^{j+\beta(\delta)}(A)=0 $.
\end{lem}

\begin{proof}[Proof of Proposition \ref{hausdim}] 
Let $ \delta>0 $, $ i\in\Z_+ $, and $ k\in\{0,1,2,...,n-\lceil\al_p\rceil\} $. Assume that $ \Sg^{k,i}(u) $ is the set of points in $ \Sg^k(u) $ such that for any $ \rho\in(0,i^{-1}] $, there holds 
$$
\rho^{-1}(\{x\in B_\rho(y):\vt(u,x)\geq\vt(u,y)-i^{-1}\}-y)\subset\{x\in\R^n:\dist(x,V)<\delta\}, 
$$
for some $ V\in\bG(n,k) $. By Lemma \ref{deltaaplem}, $$
\Sg^k(u)=\bigcup_{i\in\Z_+}\Sg^{k,i}(u).
$$
For any $ \ell\in\Z_+ $, define
\be
\Sg^{k,i,\ell}(u):=\{x\in\Sg^{k,i}(u):\vt(u,x)\in((\ell-1)i^{-1},\ell i^{-1}]\}.\label{defSgkiq}
\ee
As a result, we have 
$$
\Sg^k(u)=\bigcup_{i,\ell\in\Z_+}\Sg^{k,i,\ell}(u).
$$
For any $ x\in\Sg^{k,i,\ell}(u) $, by \eqref{defSgkiq}, we see that
$$
\Sg^{k,i,\ell}(u)\subset\{y\in\Sg^{k,i}(u):\vt(u,y)>\vt(u,x)-i^{-1}\}.
$$
Using Lemma \ref{deltaaplem} with $ \va=i^{-1} $, for any $ 0<\rho\leq i^{-1} $, we obtain
$$
(\rho^{-1}(\Sg^{k,i,\ell}(u)-y))\cap B_1\subset\{y\in\R^n:\dist(y,V)<\delta\},
$$
for some$ V\in\bG(n,k) $. As a consequence, the set $ \Sg^{k,i,\ell}(u) $ has the $ (\delta,k) $-approximation property for any $ \delta>0 $ with $ \rho_0=i^{-1} $. Then it follows from Lemma \ref{lemhau} that $ \dim_{\HH}(\Sg^{k,i,\ell}(u))\leq k $.
Since the Hausdorff dimension does not increase under countable union, we deduce that $ \dim_{\HH}(\Sg^k(u))\leq k $.
\end{proof}

\begin{lem}\label{sgkSkrela}
Let $ k\in\{0,1,2,...,n-1\} $. Assume that $ u\in(H^1\cap L^{p+1})(B_{40}) $ is a stationary solution of \eqref{superequation}. Then $
S^k(u)=\Sg^k(u) $.
\end{lem}
\begin{proof}
If $ x\notin\Sg^k(u) $, then there exists a tangent measure $ \mu\in \M(\R^n) $ at $ x $ such that $ \dim(\Sg(\mu))\geq k+1 $. For this $ \mu $, we can choose $ v \in(H_{\loc}^1\cap L_{\loc}^{p+1})(\R^n) $ such that $ (v,\mu) $ is a tangent pair of $ u $ at $ x $. Using the fifth property of Lemma \ref{5prop}, we obtain that $ v $ and $ \mu $ are invariant with respect to $ \Sg(\mu) $. By Lemma \ref{tanfunmea}, $ v $ and $ \mu $ are both $ 0 $-symmetric. Then $ (v,\mu) $ is a $ (k+1) $-symmetric pair, which implies that $ x\notin S^k(u) $. On the other hand, if $ x\notin S^k(u) $, then we can choose $ (v,\mu) $ as a $ (k+1) $-symmetric tangent pair of $ u $ at $ x $ with respect $ V\in\bG(n,k+1) $. In view of the form of $ \vt_r(\mu,y) $ given by \eqref{thetamuya} and the $ (k+1) $-symmetry of $ (v,\mu) $, we see that for any $ r>0 $ and $ y\in V $, there holds 
$$
\vt_r(\mu,y)=\vt_r(\mu,0)=\vt(\mu,0)=\vt(u,x),
$$
which implies that $ V\subset\Sg(\mu) $, and then $ \dim(\Sg(\mu))\geq k+1 $. As a result, $ x\notin\Sg^k(u) $, and we can complete the proof.
\end{proof}

\begin{proof}[Proof of Theorem \ref{Snminusalp}]
The proof follows directly from Lemma \ref{lemcut}, Proposition \ref{hausdim}, and Lemma \ref{sgkSkrela}.
\end{proof}

Next, we present the proofs of Proposition \ref{proprela} and its corollaries. In subsequent proofs, we will consider the limits of $ k $-symmetric functions and measures, necessitating the following results. Intuitively, the $ k $-symmetry properties of these functions and measures are maintained under the weak convergence in $ L_{\loc}^2 $ and convergence in the sense of Radon measures. We can conclude such results in the following two lemmas, and we will omit proofs of them for simplicity.

\begin{lem}\label{ConSym}
Let $ k\in\{0,1,2,...,n\} $. Assume that $ \{u_i\}\subset L_{\loc}^2(\R^n) $ is a sequence of $ k $-symmetric functions at $ \{x_i\}\subset\R^n $, with respect to $ \{V_i\}\subset\bG(n,k) $. If $ u_i\wc u $ weakly in $ L_{\loc}^2(\R^n) $, $ V_i\to V\in\bG(n,k) $ and $ x_i\to x $, then $ u $ is $ k $-symmetric at $ x $ with respect to $ V $.
\end{lem}

\begin{lem}\label{convergesymmetry}
Assume that $ \{\eta_i\}\subset\M(\R^n) $ is a sequence of $ k $-symmetric Radon measures at $ \{x_i\}\subset\R^n $, with respect to $ \{V_i\}\subset\bG(n,k) $. If $ \eta_i\wc^*\eta $ in $ \M(\R^n) $, $ V_i\to V\in\bG(n,k) $ and $ x_i\to x $, then $ \eta $ is $ k $-symmetric at $ x $, with respect to $ V $.
\end{lem}

The lemma below is a consequence of the $ 0 $-symmetry properties for functions and measures. We will apply it in compactness arguments.

\begin{lem}\label{comsympro}
Let $ r>0 $. Assume that $ \{u_i\}\subset L_{\loc}^2(\R^n) $ is a sequence of $ 0 $-symmetric functions and $ \{\mu_i\} $ is a sequence of $ 0 $-symmetric measures in $ \M(\R^n) $. The following properties hold.
\begin{enumerate}[label=$(\theenumi)$]
\item If $ \sup_{i\in\Z_+}\|u_i\|_{L^2(B_r)}<+\ift $, then there exists $ u\in L_{\loc}^2(\R^n) $ such that up to a subsequence, $ u_i\to u $ in $ L_{\loc}^2(\R^n) $.
\item If $ \sup_{i\in\Z_+}\mu(B_r)<+\ift $, then there exists $ \mu\in\M(\R^n) $ such that up to a subsequence, $ \mu_i\wc^*\mu $ in $ \M(\R^n) $. In particular, if $ \sup_{i\in\Z_+}d_{0,r}(\mu_i,0)<+\ift $, such subsequence also exists.
\end{enumerate}
\end{lem}
\begin{proof}
The result follows directly from the $ 0 $-symmetry of functions and measures, together with weak and weak${}^*$ compactness of $ L_{\loc}^2(\R^n) $ and $ \M(\R^n) $.
\end{proof}

\begin{proof}[Proof of Proposition \ref{proprela}]
Up to a scaling and a translation, we assume that $ r=1 $ and  $ x=0 $. The first property of this proposition follows directly from Definition \ref{qunsybypair} and \ref{quansybyfun}. For the second one, assume that the result is not true. There exist $ \va_0>0 $, and a sequence of stationary solutions of \eqref{superequation}, denoted by $ \{u_i\}\subset(H^1\cap L^{p+1})(B_1) $ such that for any $ i\in\Z_+ $, the following properties hold.
\begin{itemize}
\item $ \theta_1(u_i,0)\leq\Lda $, $ u_i $ is [CN] $ (k,i^{-1}) $-symmetric in $ B_1 $. In particular, there exists a $ k $-symmetric function $ v_i:\R^n\to\R $, satisfying
\be
\int_{B_1}|u_i-v_i|^2<i^{-1}.\label{uiviiki}
\ee
\item $ u_i $ is not $ (k,\va_0) $-symmetric in $ B_1 $.
\end{itemize}
Since $ \al_p $ is not an integer, we can use \eqref{uiviiki}, the fifth property of Proposition \ref{DefMea}, and Lemma \ref{comsympro} to obtain that up to a subsequence,
\be
\begin{aligned}
u_i&\to w\text{ strongly in }(H^1\cap L^{p+1})(B_1),\\
m_{u_i}&\wc^*m_w\text{ in }\M(B_1),\\
v_i&\wc v\text{ weakly in }L_{\loc}^2(\R^n).
\end{aligned}\label{uiwconve}
\ee
Taking $ i\to+\ift $ in \eqref{uiviiki}, we have $ w=v $ a.e. in $ B_1 $ and then $ m_w=m_v $ in $ \M(B_1) $. Using Lemma \ref{ConSym}, we deduce that $ v $ is also $ k $-symmetric. Given \eqref{uiwconve}, for $ i\in\Z_+ $ sufficiently large, $ D_{0,1}((u_i,m_{u_i}),(v,m_v))<\va_0 $ and $ (v,m_v) $ is a $ k $-symmetric pair, a contradiction.
\end{proof}

\begin{proof}[Proof of Corollary \ref{correla}]
Since the first and the second properties follow directly from Proposition \ref{proprela}, we only need to show the third one. Obviously, by the definition of $ S^k(u) $ and $ S_{[\CN]}^k(u) $, we have $ S_{[\CN]}^k(u)\subset S^k(u) $. On the other hand, if $ x\notin S_{[\CN]}^k(u) $, there exists a $ (k+1) $-symmetric tangent function $ v\in(H_{\loc}^1\cap L_{\loc}^{p+1})(\R^n) $ of $ u $ at $ x $. Applying almost the same arguments in the proof of Proposition \ref{proprela}, since $ \al_p $ is not an integer, we obtain that $ m_v\in\M(\R^n) $ is a $ (k+1) $-symmetric tangent measure of $ u $ at $ x $. Noting that $ v $ and $ m_v $ are both invariant with respect to $ V\in\bG(n,k+1) $, we have that $ (v,m_v) $ is a $ (k+1) $-symmetric tangent pair of $ u $ at $ x $, and then $ x\notin S^k(u) $. 
\end{proof}

\subsection{Comparison with the stratification used in \texorpdfstring{\cite{HSV19}}{}}

In the paper \cite{HSV19}, the authors established a different definition of quantitative stratification. Analogous to this, we can give the corresponding one in our model. We first define [HSV] $ (k,\va) $-symmetry for stationary solutions.

\begin{defn}\label{defnHSV}
Let $ k\in\{0,1,2,...,n\} $, $ r>0 $, and $ x\in\R^n $. Assume that $ u\in(H^1\cap L^{p+1})(B_{20r}(x)) $ is a stationary solution of \eqref{superequation}. We say that $ u $ is [HSV] $ (k,\va) $-symmetric in $ B_r(x) $ if it satisfies the following two properties.
\begin{enumerate}[label=$(\theenumi)$]
\item $ \vt_{2r}(u,x)-\vt_{r}(u,x)<\va $.
\item There exists $ V\in\bG(n,k) $ such that
$$
r^{\al_p-n}\int_{B_r(x)}|V\cdot\na u|^2<\va.
$$
\end{enumerate}
\end{defn}

We can use this notion to define quantitative stratification.

\begin{defn}
Let $ \va>0 $, $ k\in\{0,1,2,...,n-1\} $, and $ 0<r<1 $. Assume that $ u\in(H^1\cap L^{p+1})(B_{40}) $ is a stationary solution of \eqref{superequation}. We define $ S_{[\HSV];\va,r}^k(u) $ by 
\begin{align*}
S_{[\HSV];\va,r}^k(u)&:=\{x\in B_{10}:u\text{ is not [HSV] }\\
&(k+1,\va)\text{-symmetric in }B_s(x)\text{ for any }r\leq s<1\}.
\end{align*}
\end{defn}

The proposition below presents the connections between the [HSV] quantitative stratification above and that in Definition \ref{defSkvar}.

\begin{prop}\label{proprela2}
Let $ k\in\{0,1,2,...,n\} $, $ r>0 $, and $ x\in\R^n $. Assume that $ u\in(H^1\cap L^{p+1})(B_{20r}(x)) $ is a stationary solution of \eqref{superequation}, satisfying $ \theta_{20r}(u,x)\leq\Lda $. The following properties hold.
\begin{enumerate}[label=$(\theenumi)$]
\item For any $ \va>0 $, there exists $ \delta>0 $, depending only on $ \va,\Lda,n $, and $ p $ such that if $ u $ is $ (k,\delta) $-symmetric in $ B_{20r}(x) $, then $ u $ is $ [\HSV] $ $ (k,\va) $-symmetric in $ B_r(x) $.
\item  For any $ \va>0 $, there exists $ \delta>0 $, depending only on $ \va,\Lda,n $, and $ p $ such that if $ u $ is $ [\HSV] $ $ (k,\delta) $-symmetric in $ B_r(x) $, then $ u $ is $ (k,\va) $-symmetric in $ B_r(x) $.
\end{enumerate}
\end{prop}

\begin{proof}
Up to a scaling and a translation, we assume that $ r=1 $ and  $ x=0 $. For the first property, assume that the result is not true. There exist $ \va_0>0 $ and a sequence of stationary solutions of \eqref{superequation}, denoted by $ \{u_i\}\subset (H^1\cap L^{p+1})(B_{20}) $ such that for any $ i\in\Z_+ $, the following properties hold.
\begin{itemize}
\item $ \theta_{20}(u_i,0)\leq\Lda $, and $ u_i $ is $ (k,i^{-1}) $-symmetric in $ B_1 $. As a consequence, there exist $ V_i\in\bG(n,k) $, $ v_i\in L_{\loc}^2(\R^n) $, and $ \eta_i\in\M(\R^n) $ such that $ (v_i,\eta_i) $ is a $ k $-symmetric pair, with respect to $ V_i $, and
\be
D_{0,20}((u_i,m_{u_i}),(v_i,\eta_i))<i^{-1}.\label{uiviiki2}
\ee
\item $ u_i $ is not [HSV] $ (k,\va_0) $-symmetric in $ B_1 $.
\end{itemize} 
In view of assumptions on $ u_i $, up to a subsequence, we can assume that $ V_i\to V\in\bG(n,k) $, and
\be
\begin{aligned}
u_i&\wc w\text{ weakly in }(H^1\cap L^{p+1})(B_{20}),\\
\(\f{\pa_{\beta}u_i\pa_{\ga}u_i}{2}\ud y\)&\wc^*\(\f{\pa_{\beta}w\pa_{\ga}w}{2}\ud y\)+(\nu_1^{\beta\ga})\text{ in }\M(B_{20},\R^{n\times n}),\\
\f{|\na u_i|^2}{2}\ud y&\wc^*\f{|\na w|^2}{2}\ud y+\nu_1\text{ in }\M(B_{20}),\\
\f{|u_i|^{p+1}}{p+1}\ud y&\wc^*\f{|w|^{p+1}}{p+1}\ud y+\nu_2\text{ in }\M(B_{20}),
\end{aligned}\label{matrcon}
\ee
where $ \nu_1,\nu_2\in\M(B_{20}) $ are defect measures. We also denote $ \mu=m_w+(p-1)(\nu_1+\nu_2) $. It follows from \eqref{matrcon} that $ m_{u_i}\wc^*\mu $ in $ \M(B_{20}) $. By \eqref{uiviiki2} and Lemma \ref{comsympro}, we also assume that
\be
\begin{aligned}
v_i&\wc v\text{ weakly in }L_{\loc}^2(\R^n),\\
\eta_i&\wc^*\eta\text{ in }\M(\R^n).
\end{aligned}\label{vivmuimu}
\ee
Applying Lemma \ref{ConSym} and \ref{convergesymmetry}, we see that $ (v,\eta) $ is a $ k $-symmetric pair with respect to $ V $. In view of \eqref{matrcon} and \eqref{vivmuimu}, we can take $ i\to+\ift $ in \eqref{uiviiki2} to obtain that $ w=v $ a.e. in $ B_{20} $ and $ \mu=\eta $ in $ \M(B_{20}) $. Consequently, $ (v,\mu) $ is a $ k $-symmetric pair in $ B_{20} $, with respect to $ V $. Now, we claim that in $ \M(B_2) $, 
\be
\tau_{\beta}\tau_{\ga}\nu_1^{\beta\ga}=0,\label{claim2use}
\ee
for any $ \tau=\tau^{(j)}\in\{1,2,...,k\} $, where 
\be
\{\tau^{(j)}\}_{j=1}^k=\{(\tau_{\beta}^{(j)})_{\beta=1}^n\}_{j=1}^k\label{orthorbase} 
\ee
is an orthonormal basis of $ V $. We first assume that the claim \eqref{claim2use} is true. Using \eqref{vtlim}, Proposition \ref{MonFor} and the fact that $ (w,\mu) $ is $ k $-symmetric in $ B_{20} $, we have
\be
\begin{aligned}
&\lim_{i\to+\ift}(\vt_2(u_i,0)-\vt_1(u_i,0))=\vt_2(\mu,0)-\vt_1(\mu,0)=0.
\end{aligned}\label{thet2thet1}
\ee
By \eqref{Therjmu}, this also implies that
\be
y_{\beta}y_{\ga}\nu_1^{\beta\ga}=0\text{ in }\M(B_{10}).\label{homnu1beta}
\ee
On the other hand, for any $ j\in\{1,2,...,k\} $, it follows from \eqref{matrcon} that
$$
\lim_{i\to+\ift}\(\sum_{j=1}^k\int_{B_1}|\tau^{(j)}\cdot\na u_i|^2\)\leq\sum_{j=1}^k\int_{B_2}|\tau^{(j)}\cdot\na w|^2+\sum_{j=1}^k\int_{B_2}\tau_{\beta}^{(j)}\tau_{\ga}^{(j)}\ud\nu_1^{\beta\ga}=0,
$$
where we have used the fact that $ w $ is invariant with respect to $ V $. This, together with \eqref{thet2thet1}, implies that if $ i\in\Z_+ $ is sufficiently large, then
$$
\vt_2(u_i,0)-\vt_1(u_i,0)<\va_0,\quad\int_{B_1}|V\cdot\na u_i|^2<\va_0,
$$
a contradiction. Let us turn to the proof of \eqref{claim2use}. Fix $ \tau=\tau^{(j)} $ for some $ j\in\{1,2,...,k\} $. Analogous to the proof of Lemma \ref{5prop}, since $ (w,\mu) $ is a $ k $-symmmetric pair with respect to $ V $, we see that $ (w,\mu) $ is $ 0 $-symmetric at the point $ \tau\in V $, and then $ \vt_r(\mu,\tau)=\vt(\mu,\tau) $ for any $ 0<r<4 $. Combining \eqref{Therjmu}, we have 
\be
(y-\tau)_{\beta}(y-\tau)_{\ga}\nu_1^{\beta\ga}=0\text{ in }\M(B_8(\tau)).\label{Btaueight}
\ee
By \eqref{homnu1beta}, we have \eqref{claim2use}. Indeed, we define $ (\nu_{1,i}^{\beta,\ga})_{\beta,\ga=1}^n\in\M(B_{20},\R^{n\times n}) $ as $ \nu_{1,i}^{\beta,\ga}:=\f{1}{2}\pa_{\beta}(u_i-w)\pa_{\ga}(u_i-w)\ud y $, and deduce from \eqref{matrcon} that
\be
(\nu_{1,i}^{\beta\ga})\wc^*(\nu_1^{\beta\ga})\text{ in }\M(B_{20},\R^{n\times n}).\label{nu1betaga2}
\ee
As a result, it follows from \eqref{homnu1beta} and \eqref{Btaueight} that
\begin{align*}
0&\leq\int_{B_2}\tau_{\beta}\tau_{\ga}\ud\nu_{1,i}^{\beta\ga}\\
&\leq 2\(\int_{B_4(\tau)}|(y-\tau)\cdot\na(u_i-w)|^2\ud y+\int_{B_4}|y\cdot\na(u_i-w)|^2\ud y\)\\
&=2\(\int_{B_4(\tau)}(y-\tau)_{\beta}(y-\tau)_{\ga}\ud\nu_{1,i}^{\beta\ga}+\int_{B_4} y_{\beta}y_{\ga}\ud\nu_{1,i}^{\beta\ga}\)\to 0,
\end{align*}
as $ i\to+\ift $, which implies \eqref{claim2use}. 

For the second property, we still prove it by contradiction. Assume that the result is not true. There exist $ \va_0>0 $ and a sequence of stationary solutions of \eqref{superequation}, denoted by $ \{u_i\}\subset (H^1\cap L^{p+1})(B_{20}) $ such that for any $ i\in\Z_+ $, the following properties hold.
\begin{itemize}
\item $ \theta_{20}(u_i,0)\leq\Lda $ and $ u_i $ is [HSV] $ (k,i^{-1}) $-symmetric in $ B_1 $. Precisely,
\be
\vt_2(u_i,0)-\vt_1(u_i,0)<i^{-1},\label{vtui2}
\ee
and there exists $ V_i\in\bG(n,k) $ such that
\be
\int_{B_1}|V_i\cdot u_i|^2<i^{-1}.\label{uiviiki3}
\ee
\item $ u_i $ is not $ (k,\va_0) $-symmetric in $ B_1 $.
\end{itemize}
Up to a subsequence, we can still assume that \eqref{matrcon} holds and $ V_i\to V\in\bG(n,k) $. \eqref{orthorbase} is also an orthonormal basis of $ V $. In view of \eqref{vtui2}, we obtain 
\begin{align*}
0&\leq-2\int_1^2\(\rho^{\al_p-n-1}\int_{\R^n}\left|\f{y\cdot\na w}{\rho}+\f{2w}{(p-1)\rho}\right|^2\dot{\phi}_{0,\rho}\ud y\)\ud\rho\\ 
&\quad-2\int_1^2\(\rho^{\al_p-n-3}\int_{\R^n}y_{\beta}y_{\ga}\dot{\phi}_{0,\rho}\ud\nu_1^{\beta\ga}(y)\)\ud\rho\\
&\leq\lim_{i\to+\ift}\[-2\int_1^2\(\rho^{\al_p-n-1}\int_{\R^n}\left|\f{y\cdot\na u_i}{\rho}+\f{2u_i}{(p-1)\rho}\right|^2\dot{\phi}_{0,\rho}\ud y\)\ud\rho\]\\
&\leq\lim_{i\to+\ift}(\vt_2(u_i,0)-\vt_1(u_i,0))=0,
\end{align*}
which implies that $ w $ is $ 0 $-symmetric in $ B_{10} $ and $ y_{\beta}y_{\ga}\nu_1^{\beta\ga}=0 $ in $ \M(B_{10}) $. Applying almost the same arguments in the proof of Lemma \ref{tanfunmea}, we deduce that $
\nu=(p-1)(\nu_1+\nu_2) $ is also $ 0 $-symmetric in $ B_{10} $. By the convergence in \eqref{matrcon},
\be
\mu=m_u+\nu,\label{etamunu}
\ee
and then $ (w,\mu) $ is a $ 0 $-symmetric pair in $ B_{10} $. Taking $ i\to+\ift $ in \eqref{uiviiki3},
\be
\begin{aligned}
\tau\cdot\na w&=0\text{ a.e. in }B_1,\\
\tau_{\beta}\tau_{\ga}\nu_1^{\beta\ga}&=0\text{ in }\M(B_1),
\end{aligned}\label{wnavwnu}
\ee
for any $ \tau=\tau^{(j)} $ with $ j\in\{1,2,...,k\} $. As a result, $ w $ is invariant with respect to $ V $ in $ B_1 $, so it remains to show that $ \mu $ is also invariant with respect to $ V $ in $ B_1 $. Fix $ \tau=\tau^{(j)} $ for some $ j\in\{1,2,...,k\} $. We claim that for any $ \vp\in C_0^{\ift}(B_1) $, 
\be
\int_{B_1}\tau_{\ga}\pa_{\beta}\vp\ud\nu_1^{\beta\ga}=0.\label{phidnu2}
\ee
Analogously, we deduce from \eqref{matrcon} that \eqref{nu1betaga2} is still true. By Cauchy's inequality, we have
\begin{align*}
\(\int_{B_1}\tau_{\ga}\pa_{\beta}\vp\ud\nu_{1,i}^{\beta\ga}\)^2\leq\( \int_{B_1}\pa_{\beta}\vp\pa_{\ga}\vp\ud\nu_{1,i}^{\beta\ga}\)\(\int_{B_1} \tau_{\beta}\tau_{\ga}\ud\nu_{1,i}^{\beta\ga}\).
\end{align*}
Letting $ i\to+\ift $ and taking \eqref{wnavwnu} into account, \eqref{phidnu2} follows directly. Fix $ \psi\in C_0^{\ift}(B_1) $ and choose $ \lda_0>0 $ such that $ \supp\psi_{\lda}\subset B_1 $, for any $ \lda\in(0,\lda_0) $, where $ \psi_{\lda}(y)=\psi(y+\lda\tau) $. Testing \eqref{StaCon} with $ Y(y)=\psi_{\lda}(y)\tau $, there holds
\be
\int_{B_1}\[\(\f{|\na u_i|^2}{2}-\f{|u_i|^{p+1}}{p+1}\)(\tau\cdot\na\psi_{\lda})+(\na u_i\cdot\na\psi_{\lda})(\tau\cdot\na u_i)\]=0.\label{takevi}
\ee
We can take $ i\to+\ift $ for both sides of the above and use \eqref{matrcon}, \eqref{wnavwnu}, \eqref{phidnu2}, and the second property of Proposition \ref{DefMea} to obtain that
$$
\int_{B_1}\(\f{|\na w|^2}{2}-\f{|w|^{p+1}}{p+1}\)(\tau\cdot\na\psi_{\lda})+\f{1}{p+3}\int_{B_1}(\tau\cdot\na\psi_{\lda})\ud\nu=0.
$$
Again, by the invariance of $ w $ with respect to $ V $ in $ B_1 $, we have
\begin{align*}
\f{\ud}{\ud\lda}\(\int_{B_1}\psi_{\lda}|\na w|^2\)&=\int_{B_1}(\tau\cdot\na\psi_{\lda})|\na w|^2=0,\\
\f{\ud}{\ud\lda}\(\int_{B_1}\psi_{\lda}|w|^{p+1}\)&=\int_{B_1}(\tau\cdot\na\psi_{\lda})|w|^{p+1}=0,
\end{align*}
for any $ \lda\in(0,\lda_0) $. Combining \eqref{takevi}, there holds
$$
\f{\ud}{\ud\lda}\(\int_{B_1}\psi_{\lda}\ud\nu\)=\int_{B_1}(\tau\cdot\na\psi_{\lda})\ud\nu=0.
$$
By the arbitrariness of $ \psi $ and $ \tau=\tau^{(j)}\in V $, $ \nu $ is invariant with respect to $ V $. It follows from \eqref{etamunu} that $ \mu $ is also invariant with respect to $ V $ in $ B_1 $. Now we have that $ (w,\mu) $ is a $ k $-symmetric pair with respect to $ V $ in $ B_1 $. As a result, \eqref{matrcon} implies that for $ i\in\Z_+ $ sufficiently large, $ D_{0,1}((u_i,m_{u_i}),(w,\mu))<\va_0 $, which is a contradiction.
\end{proof}

We can apply almost the same arguments in the above proposition to obtain the lemma below, which will be used in the later proof. 

\begin{lem}\label{lemusefre}
Let $ r>0 $ and $ x\in\R^n $. Assume that $ \{u_i\}\subset(H^1\cap L^{p+1})(B_{10r}(x)) $ is a sequence of stationary solution of \eqref{superequation} such that $
u_i\wc v $ weakly in $ (H^1\cap L^{p+1})(B_{10r}(x)) $ and $ m_{u_i}\wc^*\mu $ in $ \M(B_{10r}(x)) $. The following properties hold.
\begin{enumerate}[label=$(\theenumi)$]
\item If we have a limit
$$
\lim_{i\to+\ift}(\vt_r(u_i,x)-\vt_{\f{r}{2}}(u_i,x))=0, 
$$
then $ (v,\mu) $ is a $ 0 $-symmetric pair at $ x $ in $ B_{9r}(x) $.
\item If there exists a sequence $ \{V_i\}\subset\bG(n,k) $ with $ k\in\{1,2,...,n-1\} $ such that
$$
\lim_{i\to+\ift}\(\int_{B_r(x)}|V_i\cdot\na u_i|^2\)=0,
$$
then up to a subsequence, $ V_i\to V\in\bG(n,k) $, and $ (v,\mu) $ is in variant with respect to $ V $ in $ B_r(x) $.
\end{enumerate}
\end{lem}

The following corollary is a direct consequence of Proposition \ref{proprela2}.

\begin{cor}\label{HSVStraCor}
Let $ \va>0 $, $ k\in\{0,1,2,...,n-1\} $, and $ 0<r<1 $. Assume that $ u\in(H^1\cap L^{p+1})(B_{40}) $ is a stationary solution of \eqref{superequation}, satisfying $ \theta_{40}(u,0)\leq\Lda $. The following properties hold.
\begin{enumerate}[label=$(\theenumi)$]
\item For any $ \va>0 $, there exists $ \delta>0 $, depending only on $ \va,\Lda,n $, and $ p $ such that $
S_{[\HSV];\va,\f{r}{20}}^k(u)\subset S_{\delta,r}^k(u) $.
\item For any $ \va>0 $, there exists $ \delta>0 $, depending only on $ \va,\Lda,n $, and $ p $ such that $
S_{\va,r}^k(u)\subset S_{[\HSV];\delta,r}^k(u) $.
\end{enumerate}
\end{cor}
\begin{proof}
For the first property, if $ \va>0 $, we choose $ \delta=\delta(\va,\Lda,n,p)>0 $ such that the first property of Proposition \ref{proprela2} holds. If $ x\notin S_{\delta,r}^k(u) $, then $ u $ is $ (k+1,\delta) $-symmetric in $ B_s(x) $ for some $ r\leq s<1 $. By the choice of $ \delta $, we can obtain that $ u $ is $ (k+1,\delta) $-symmetric in $ B_{\f{s}{20}}(x) $, which implies that $ x\notin S_{[\HSV];\va,\f{r}{20}}^k(u) $ and the result follows directly. The second property of this corollary can be proved by almost the same methods as the first one using the second property of Proposition \ref{proprela2}. We omit it for the sake of simplicity.
\end{proof}

This corollary implies that the stratification in our paper is equivalent to the one used by \cite{HSV19} in some context. 

\subsection{Effectively spanned affine subspaces}

In this subsection, we will give the definition and properties of the effectively spanned affine subspaces. 

\begin{defn}
In $ \R^n $, for $ k\in\{1,2,...,n\} $, let $ \{x_i\}_{i=0}^k\subset\R^n $ and $ \rho>0 $. We call that these points $ \rho $-effectively span $ L=x_0+\op{span}\{x_i-x_0\}_{i=1}^k\in\bA(n,k) $ if for all $ i\in\{2,3,...,k\} $,
$$
\dist(x_i,x_0+\op{span}\{x_1-x_0,...,x_{i-1}-x_0\})\geq 2\rho.
$$
We also call such points are $ \rho $-independent. For a set $ F\subset\R^n $, we say that it $ \rho $-effectively spans a $ k $-dimensional affine subspace if there exist $ \{x_i\}_{i=0}^k\subset F $, which are $ \rho $-independent.
\end{defn}

Regarding $ \rho $-independent points, we have some basic observations.

\begin{lem}\label{rhoind}
Let $ k\in\{1,2,...,n\} $. We have the following properties.
\begin{enumerate}[label=$(\theenumi)$]
\item If $ \{x_i\}_{i=0}^k $ $ \rho $-effectively span $ L\in\bA(n,k) $, then for any $ x\in L $, there exists a unique set of numbers $ \{\al_i\}_{i=1}^k $ such that
$$
x=x_0+\sum_{i=1}^k\al_i(x_i-x_0),\quad|\al_i|\leq \f{C|x-x_0|}{\rho},
$$
where $ C>0 $ depends only on $ n $.
\item If $ \{x_{i,j}\}_{i=1}^k $ are $ \rho $-independent for any $ j\in\Z_+ $, and $ x_{i,j}\to x_i $, for any $ i\in\{1,2,...,k\} $, then $ \{x_i\}_{i=1}^k $ are also $ \rho $-independent.
\end{enumerate}
\end{lem}
\begin{proof}
Note that $ \{x_i\}_{i=0}^k $ are $ \rho $-independent, if and only if $ \{\f{x_i}{\rho}\}_{i=0}^k $ are $ 1 $-independent. Applying Lemma 4.6 and its proof of \cite{NV17}, the first property follows directly. The second property is the direct consequence of the definition of $ \rho $-independence.
\end{proof}

\section{Quantitative stratification}\label{QuantiStrati}

\subsection{Some basic results on quantitative stratification} Now, we will examine a set of fundamental characteristics of quantitative stratification. We start this subsection with the subsequent lemma, which is on the characterization of $ S^k(u) $ through $ S_{\va,r}^k(u) $ for a stationary solution of \eqref{superequation}.

\begin{lem}\label{SkSkva}
Let $ k\in\{0,1,2,...,n-1\} $. Assume that $ u\in(H^1\cap L^{p+1})(B_{40}) $ is a stationary solution of \eqref{superequation}. Then
\be
S^k(u)\cap B_{10}=\bigcup_{\va>0}\bigcap_{0<r<1}S_{\va,r}^k(u).\label{Sktworep}
\ee
\end{lem}

\begin{proof}
For simplicity, we denote the right hand side of \eqref{Sktworep} as $ A^k(u) $. If $ x\notin S^k(u)\cap B_{10} $, then there exist a $ (k+1) $-symmetric tangent pair $ (v,\mu)\in\cF_x(u)\times\M_x(u) $. As a result, we have a sequence $ \{r_i\} $, with $ r_i\to 0^+ $ such that $
T_{x,r_i}u\wc v $ weakly in $ (H_{\loc}^1\cap L_{\loc}^{p+1})(\R^n) $, and $ T_{x,r_i}m_u\wc^*\mu\text{ in }\M(\R^n) $. For any $ \va>0 $, if $ i\in\Z_+ $ is sufficiently large, then $ D_{0,1}((T_{x,r_i}u,T_{x,r_i}m_u),(v,\mu))<\va $. Consequently, $ 
x\notin A^k(u) $ and hence $ A^k(u)\subset S^k(u)\cap B_{10} $. 

On the other hand, we assume that $ x\notin A^k(u) $. There exist $ (k+1) $-symmetric pairs $ \{(v_i,\eta_i)\}\subset L_{\loc}^2(\R^n)\times\M(\R^n) $, with respect to $ \{V_i\}\subset\bG(n,k+1) $ such that 
\be
D_{0,1}((T_{x,r_i}u,T_{x,r_i}m_u),(v_i,\eta_i))<i^{-1},\label{Txrihi}
\ee
for any $ i\in\Z_+ $. By \eqref{Txrihi} and Lemma \ref{comsympro}, up to a subsequence, we can assume that $ V_i\to V\in\bG(n,k+1) $, and
\be
\begin{aligned}
v_i&\wc v\text{ weakly in }L_{\loc}^2(\R^n),\\
\eta_i&\wc^*\eta\text{ in }\M(\R^n).
\end{aligned}\label{vetacon}
\ee
It follows from Lemma \ref{ConSym} and \ref{convergesymmetry} that $ (v,\eta) $ is a $ (k+1) $-symmetric pair with respect to $ V $. If $ \inf r_i=r>0 $, up to a subsequence, we can assume that $ r_i\to r $, $
T_{x,r_i}u\to T_{x,r}u $ weakly in $ (H^1\cap L^{p+1})(B_1) $, and $ T_{x,r_i}m_u\wc^*T_{x,r}m_u $ in $ \M(B_1) $. By \eqref{Txrihi} and the $ (k+1) $-symmetry of $ (v,\eta) $, $ (T_{x,r}u,T_{x,r}m_u) $ is a $ (k+1) $-symmetric pair with respect to $ V $ in $ B_1 $. Given Remark \ref{extend}, we have that $ (T_{x,r}u,T_{x,r}m_u) $ can be trivially extended to be a $ (k+1) $-symmetric tangent pair of $ u $ at $ x $, and then $ x\notin S^k(u)\cap B_{10} $. If $ \inf r_i=0 $, using Proposition \ref{tangentfm}, up to a subsequence, we can assume that as $ r_i\to 0^+ $, $
T_{x,r_i}u\wc w $ weakly in $ (H_{\loc}^1\cap L_{\loc}^{p+1})(\R^n) $, and $ T_{x,r_i}m_u\wc^*\mu $ in $ \M(\R^n) $. Letting $ i\to+\ift $ in \eqref{Txrihi}, since $ (v,\eta) $ is $ (k+1) $-symmetric, we can deduce that $ (w,\mu) $ is $ (k+1) $-symmetric in $ B_1 $ with respect to $ V $. Again, by Remark \ref{extend}, $ (w,\mu) $ is a $ (k+1) $-symmetric tangent pair of $ u $ at $ x $, with respect to $ V $. This implies that $ x\notin S^k(u)\cap B_{10} $, and then $ S^k(u)\cap B_{10}\subset A^k(u) $.
\end{proof}

The following lemma is a consequence of Lemma \ref{symproen} and standard compactness arguments. It is a useful observation in the proof of Theorem \ref{enhance}.

\begin{lem}\label{integraln}
Let $ \va>0 $, $ r>0 $, and $ x\in\R^n $. Assume that $ u\in(H^1\cap L^{p+1})(B_r(x)) $ is a stationary solution of \eqref{superequation}, satisfying $ \theta_{r}(u,x)\leq\Lda $. There exists $ \delta>0 $, depending only on $ \va,\Lda,n $, and $ p $ such that if $ u $ is $ (k_{n,p},\delta) $-symmetric in $ B_r(x) $, where $ k_{n,p} $ is given by \eqref{kp}, then $ u $ is $ (n,\va) $-symmetric in $ B_r(x) $.
\end{lem}

\begin{proof}
Up to a scaling and a translation, we assume that $ r=1 $ and  $ x=0 $. Assume that the result is not true. There exist $ \va_0>0 $, and a sequence of stationary solutions of \eqref{superequation}, denoted by $ \{u_i\}\subset(H^1\cap L^{p+1})(B_1) $ such that for any $ i\in\Z_+ $, the following properties hold.
\begin{itemize}
\item $ \theta_1(u_i,0)\leq\Lda $, and $ u_i $ is $ (k_{n,p},i^{-1}) $-symmetric in $ B_1 $.
\item $ u_i $ is not $ (n,\va_0) $-symmetric in $ B_1 $.
\end{itemize}
By the second property, for any $ i\in\Z_+ $, there is a $ k_{n,p} $-symmetric pair $ (v_i,\eta_i)\in L_{\loc}^2(\R^n)\times\M(\R^n) $ with respect to $ V_i\in\bG(n,k_{n,p}) $ such that 
\be
D_{0,1}((u_i,m_{u_i}),(v_i,\eta_i))<i^{-1}.\label{D01uimuivi}
\ee
Using Lemma \ref{comsympro}, up to a subsequence, we assume that $ V_i\to V\in\bG(n,k_{n,p}) $, and
\be
\begin{aligned}
u_i&\wc w\text{ weakly in }(H^1\cap L^{p+1})(B_1),\\
m_{u_i}&\wc^*\mu\text{ in }\M(B_1),\\
v_i&\wc v\text{ weakly in }L_{\loc}^2(\R^n),\\
\eta_i&\wc^*\eta\text{ in }\M(\R^n).
\end{aligned}\label{uimuivi}
\ee
In view Lemma \ref{ConSym} and \ref{convergesymmetry}, $ (v,\eta) $ is a $ k_{n,p} $-symmetric pair with respect to $ V $. By \eqref{D01uimuivi}, we see that $ w=v $ a.e. in $ B_1 $ and $ \mu=\eta $ in $ \M(B_1) $, which implies that $ m_v,\eta\in
M(\R^n) $ are $ k_{n,p} $-symmetric with respect to $ V $. By the choice of $ k_{n,p} $, we can apply Lemma \ref{symproen} to obtain that $ v=0 $ a.e. in $ \R^n $ and $ \eta=0 $ in $ \M(\R^n) $. As a result, $ (w,\mu)=(0,0) $. Combining \eqref{D01uimuivi} and \eqref{uimuivi}, for sufficiently large $ i\in\Z_+ $, we have $ D_{0,1}((u_i,m_{u_i}),(0,0))<\va_0 $, which is a contradiction.
\end{proof}

Using the terminology of quantitative stratification, we can rewrite Proposition \ref{smareg} as the following form.

\begin{lem}\label{nvasym}
Let $ j\in\Z_{\geq 0} $, $ r>0 $, and $ x\in\R^n $. Assume that $ u\in(H^1\cap L^{p+1})(B_r(x)) $ is a stationary solution of \eqref{superequation}, satisfying $ \theta_r(u,x)\leq\Lda $. There exists $ \va>0 $, depending only on $ j,\Lda,n $, and $ p $ such that if $ u $ is $ (n,\va) $-symmetric in $ B_r(x) $, then $ u $ is smooth in $ B_{\f{r}{2}}(x) $ and satisfies 
$$
\sup_{B_{\f{r}{2}}(x)}\(\sum_{i=0}^{j}\(\f{r}{2}\)^i|D^iu|\)\leq\(\f{r}{2}\)^{-\f{2}{p-1}}.
$$
In particular, $ r_u^{j}(x)\geq \f{r}{2} $.
\end{lem}
\begin{proof}
Up to a scaling and a translation, we assume that $ r=1 $ and  $ x=0 $. If $ u $ is $ (n,
\va) $-symmetric in $ B_r(x) $, then we can obtain that $ d_{0,1}(m_u,0)<\va $. Choosing $ \va>0 $ sufficiently small, and applying basic arguments of approximation, we have $ m_u(B_{\f{3r}{4}}(x))<C(n,p)\va $ and hence the result follows from Proposition \ref{smareg}.
\end{proof}

Combining Lemma \ref{integraln} and \ref{nvasym}, we establish the lemma as follows.

\begin{lem}\label{estimaterux}
Let $ j\in\Z_{\geq 0} $, $ r>0 $, and $ x\in\R^n $. Assume that $ u\in(H^1\cap L^{p+1})(B_r(x)) $, satisfying $ \theta_r(u,x)\leq\Lda $. There exists $ \delta>0 $ depending only on $ j,\Lda,n $, and $ p $ such that if $ u $ is $ (k_{n,p},\delta) $-symmetric in $ B_r(x) $, then $
r_u^{j}(x)\geq \f{r}{2} $
\end{lem}
\begin{proof}
In view of Lemma \ref{nvasym}, we can choose $ \va=\va(j,\Lda,n,p) $ such that if $ u $ is $ (n,\va) $-symmetric in $ B_r(x) $, then $ u $ is smooth in $ B_{\f{r}{2}}(x) $ and satisfies $ r_u^{j}(x)\geq \f{r}{2} $. Applying Lemma \ref{integraln}, there exists $ \delta=\delta(\va,\Lda,n,p)>0 $ such that if $ u $ is $ (k_{n,p},\delta) $-symmetric in $ B_r(x) $, then $ u $ is $ (n,\va) $-symmetric in $ B_r(x) $. Now, it can be obtained that if $ u $ is $ (k_{n,p},\delta) $-symmetric in $ B_r(x) $, then $ r_u^{j}(x)\geq \f{r}{2} $.
\end{proof}

\begin{cor}\label{rucor}
Let $ j\in\Z_{\geq 0} $, $ 0<r\leq\f{1}{2} $, and $ x\in B_2 $. Assume that $ u\in(H^1\cap L^{p+1})(B_{40}) $, satisfying $ \theta_{40}(u,0)\leq\Lda $. Then
$$
\left\{x\in B_1:r_u^{j}(x)<r\right\}
\subset S_{\delta,2r}^{k_{n,p}-1}(u)\cap B_1,
$$
where $ \delta>0 $ is given by Lemma \ref{estimaterux}.
\end{cor}
\begin{proof}
If $ x\notin S_{\delta,2r}^{k_{n,p}-1}(u)\cap B_1 $, then there exists some $ 2r\leq s<1 $ such that $ u $ is $ (k_{n,p},\delta) $-symmetric in $ B_s(x) $. Using Lemma \ref{estimaterux}, we have $ r_u^{j}(x)\geq\f{s}{2}\geq r $, which is a contradiction.
\end{proof}

\subsection{Quantitative cone splitting results} In this subsection, we will establish quantitative cone splitting results, which are generalizations of the classical ones Lemma \ref{ConSpl} and \ref{ConSplm}.

\begin{lem}\label{SmaHom}
Let $ \va,r>0 $, and $ x\in\R^n $. Assume that $ u\in(H^1\cap L^{p+1})(B_{10r}(x)) $ is a stationary solution of \eqref{superequation}, satisfying $ \theta_{10r}(u,x)\leq\Lda $. There exists $ \delta>0 $ depending only on $ \va,\Lda,n $, and $ p $ such that if
$$
\vt_r(u,x)-\vt_{\f{r}{2}}(u,x)<\delta,
$$
then $ u $ is $ (0,\va) $-symmetric in $ B_r(x) $.
\end{lem}
\begin{proof}
Up to a scaling and a translation, we assume that $ r=1 $ and $ x=0 $. Assume that the result is not true. There exist $ \va_0>0 $, and a sequence of stationary solutions of \eqref{superequation}, denoted by $ \{u_i\}\subset(H^1\cap L^{p+1})(B_{10}) $ such that for any $ i\in\Z_+ $, the following properties hold.
\begin{itemize}
\item $ \theta_{10}(u_i,0)\leq\Lda $ and $
\vt_1(u,0)-\vt_{\f{1}{2}}(u,0)<i^{-1} $,
\item $ u_i $ is not $ (0,\va_0) $-symmetric in $ B_1 $.
\end{itemize}
Up to a subsequence, we can assume that 
\be
\begin{aligned}
u_i&\wc v\text{ weakly in }(H^1\cap L^{p+1})(B_{10}),\\
m_{u_i}&\wc^*\mu\text{ in }\M(B_{10}).
\end{aligned}\label{uieweakly}
\ee
Applying Lemma \ref{lemusefre}, we deduce that $ (v,\mu) $ is a $ 0 $-symmetric pair in $ B_9 $. If $ i\in\Z_+ $ is sufficiently large, then \eqref{uieweakly} implies that $ D_{0,1}((u_i,m_{u_i}),(v,\mu))<\va_0 $, which is a contradiction.
\end{proof}

Next, we will give the following quantitative cone splitting result for stationary solutions of \eqref{superequation}. This proposition is a generalization of Lemma \ref{SmaHom}.

\begin{prop}[Quantitative cone splitting]\label{QuaConSpl}
Let $ \va,r>0 $, $ 0<\beta<\f{1}{2} $, and $ k\in\{0,1,2,...,n\} $. Assume that $ u\in(H^1\cap L^{p+1})(B_{10r}) $ is a stationary solution of \eqref{superequation}, satisfying $ \theta_{40r}(u,0)\leq\Lda $ and $ \{x_i\}_{i=0}^k\subset B_r $ with $ x_0=0 $. There exists $ \delta>0 $, depending only on $ \beta,\va,\Lda,n $, and $ p $ such that if 
\begin{enumerate}[label=$(\theenumi)$]
\item $ \vt_r(u,x_i)-\vt_{\f{r}{2}}(u,x_i)<\delta $ for any $ i=0,1,2,...,k $,
\item $ \{x_i\}_{i=0}^k $ are $ \beta r $-independent,
\end{enumerate}
then $ u $ is $ (k,\va) $-symmetric in $ B_r $.
\end{prop}

\begin{proof}
Up to a scaling, we can assume that $ r=1 $. Assume that the result is not true. There exist $ \va_0>0 $, $ 0<\beta_0<\f{1}{2} $, a sequence of stationary solutions of \eqref{superequation}, denoted by $ \{u_j\}\subset(H^1\cap L^{p+1})(B_{40}) $, and $ \{\{x_{i,j}\}_{i=0}^k\}\subset B_1 $ with $ x_{0,j}=0 $ such that for any $ j\in\Z_+ $ the following properties hold.
\begin{itemize}
\item $ \theta_{40}(u_j,0)\leq\Lda $, and $
\sup_{0\leq i\leq k}(\vt_1(u_j,x_{i,j})-\vt_{\f{1}{2}}(u_j,x_{i,j}))<j^{-1} $.
\item $ \{x_{i,j}\}_{i=0}^k $ are $ \beta_0 $-independent points.
\item $ u_j $ is not $ (k,\va_0) $-symmetric in $ B_1 $.
\end{itemize}
Up to a subsequence, we can assume that
\be
\begin{aligned}
u_j&\wc v\text{ weakly in }(H^1\cap L^{p+1})(B_{40}),\\
m_{u_j}&\wc^*\mu\text{ in }\M(B_{40}),\\
x_{i,j}&\to x_i\in\ol{B}_1,\text{ for any }i\in\{0,1,2,...,k\}.
\end{aligned}\label{Conujvmu}
\ee
Using Lemma \ref{lemusefre}, we obtain that $ (v,\mu) $ is $ 0 $-symmetric in $ B_9(x_i) $ for any $ i\in\{0,1,2,...,k\} $. In view of Lemma \ref{ConSpl} and Corollary \ref{CorSpl}, $ (v,\mu) $ is $ k $-symmetric pair in $ B_2 $, with respect to $ V=\op{span}\{x_i\}_{i=1}^k $ (note that here $ x_0=0 $). For $ j\in\Z_+ $ sufficiently large, we deduce from \eqref{Conujvmu} that $ D_{0,1}((u_j,m_{u_j}),(v,\mu))<\va_0 $, which is a contradiction.
\end{proof}

\subsection{Further results on quantitative stratification} In this subsection, we will present some essential results that form the basis for the proofs of main theorems.

\begin{prop}\label{Fprop}
Let $ 0<\beta<\f{1}{2} $, $ \va>0 $, $ k\in\{0,1,2,...,n-1\} $, $ 0<s\leq 1 $, and $ x\in B_2 $. Assume that $ u\in(H^1\cap L^{p+1})(B_{40}) $ is a stationary solution of \eqref{superequation}, satisfying $ \theta_{40}(u,0)\leq\Lda $. There exist $ \delta,\delta'>0 $, depending only on $ \beta,\va,\Lda,n $, and $ p $ such that if the set
$$
F=\{y\in B_{2s}(x):\vt_s(u,y)-\vt_{\beta s}(u,y)<\delta\}
$$
$ \beta s $-effectively spans $ L\in\bA(n,k) $, then
$$
S_{\va,\delta's}^k(u)\cap B_s(x)\subset B_{2\beta s}(L).
$$
\end{prop}

To prove this proposition, we have to use the following lemma. 

\begin{lem}\label{kplus1}
Let $ \va>0 $, $ k\in\{0,1,2,...,n-1\} $, $ 0<s\leq 1 $, and $ x\in B_2 $. Assume that $ u\in(H^1\cap L^{p+1})(B_{40}) $ is a stationary solution of \eqref{superequation}, satisfying $ \theta_{40}(u,0)\leq\Lda $. There exists $ \delta>0 $, depending only on $ \va,\Lda,n $, and $ p $ such that if there exists $ V\in\bG(n,k+1) $ such that
\be
s^{\al_p-n}\int_{B_s(x)}|V\cdot\na u|^2<\delta,\label{detaVsmall}
\ee
then $ S_{\va,r(\delta)s}^k(u)\cap B_{\f{s}{2}}(x)=\emptyset $, where $ r(\delta)=\delta^{\f{1}{2(n-\al_p)}} $.
\end{lem}

Before the proof of this lemma, we first give a useful observation.

\begin{lem}\label{rxlogr}
Let $ 0<r<\f{1}{100} $, $ 0<s\leq 1 $, and $ x\in\R^n $. Assume that $ u\in(H^1\cap L^{p+1})(B_{10s}(x)) $ is a stationary solution of \eqref{superequation}, satisfying $ \theta_{10s}(u,x)\leq\Lda $. For any $ y\in B_{\f{s}{2}}(x) $, there exists $ r_y\in[r,\f{1}{2}] $ such that
\be
\vt_{r_ys}(u,y)-\vt_{\f{r_ys}{2}}(u,y)\leq\f{C}{\log\f{1}{r}},\label{vtrys}
\ee
where $ C>0 $ depends only on $ \Lda,n $, and $ p $. 
\end{lem}
\begin{proof}
Up to a translation, we can assume that $ s=1 $ and $ x=0 $. Fix $ y\in B_{\f{1}{2}} $. It follows from \eqref{thevt} and $ \theta_{10}(u,0)\leq\Lda $ that $ \vt_{\f{1}{2}}(u,y)\leq C(\Lda,n,p) $. Choosing $ \ell\in\Z_+ $ such that $ \f{1}{2^{\ell+1}}\leq r<\f{1}{2^{\ell}} $ with $ \ell\sim\log\f{1}{r} $, we can apply Proposition \ref{MonFor} to deduce that
$$
0\leq\sum_{i=1}^{\ell}\(\vt_{\f{1}{2^i}}(u,y)-\vt_{\f{1}{2^{i+1}}}(u,y)\)\leq\vt_{\f{1}{2}}(u,y)\leq C(\Lda,n,p),
$$
which directly implies the existence of $ r_y\in[r,\f{1}{2}] $ satisfying \eqref{vtrys}.
\end{proof}

\begin{proof}[Proof of Lemma \ref{kplus1}]
Fix $ y\in S_{\va,r(\delta)s}^k(u)\cap B_{\f{s}{2}}(x) $. Using \eqref{thevt}, \eqref{NonDegThe}, and Proposition \ref{MonFor}, we have $ \theta_{10s}(u,y)\leq C(\Lda,n,p) $. As a result, it follows from Lemma \ref{rxlogr} that
\be
\vt_{2r_ys}(u,y)-\vt_{r_ys}(u,y)\leq\f{C(\Lda,n,p)}{\log\delta^{-1}},\label{dyiasmall}
\ee
where $ 2r_y\in[2r(\delta),\f{1}{2}] $. Combining \eqref{detaVsmall}, it follows that
$$
(r_ys)^{\al_p-n}\int_{B_{r_ys}(y)}|V\cdot\na u|^2\leq(r_ys)^{\al_p-n}\int_{B_s(x)}|V\cdot\na u|^2\leq C\delta^{\f{1}{2}}.
$$
This, together with \eqref{dyiasmall}, implies that given $ \va'>0 $, we can choose sufficiently small $ \delta=\delta(\va',\Lda,n,p)>0 $ such that $ u $ is $ (k+1,\va') $ [HSV] symmetric in $ B_{r_ys}(y) $. By Proposition \ref{proprela2}, we can choose $ \va'=\va'(\va,\Lda,n,p)>0 $ such that $ u $ is $ (k+1,\va) $-symmetric in $ B_{r_ys}(y) $, which implies that $ y\notin S_{\va,r(\delta)s}(u) $.
\end{proof}

\begin{proof}[Proof of Proposition \ref{Fprop}]
Up to a translation, we can assume that $ x=0 $. By the assumption, there are points $ \{x_i\}_{i=0}^k\subset F $ being $ \beta s $-independent and spanning $ L=x_0+\op{span}\{x_i-x_0\}_{i=1}^k $ such that
\be
\vt_s(u,x_i)-\vt_{\beta s}(u,x_i)<\delta,\label{uxidel}
\ee
for any $ i\in\{0,1,2,...,k\} $, where $ \delta>0 $ is to be determined later. Fix $ y_0\in B_s\backslash B_{2\beta s}(L) $. To complete the proof, we choose appropriate $ \delta,\delta'>0 $ such that $ y_0\notin S_{\va,\delta's}^k(u)\cap B_s $. There exists $ 0<\ga<\f{1}{100} $ such that $ B_{\ga s}(y_0)\subset \cap_{i=0}^kB_{4s}(x_i) $. Using Corollary \ref{coruse}, we have
\begin{align*}
s^{\al_p-n}\int_{B_{4s}(x_i)}\left|\f{(y-x_i)\cdot\na u}{s}+\f{2u}{(p-1)s}\right|^2\ud y\leq C(n,p)\(\vt_s(u,x_i)-\vt_{\f{s}{2}}(u,x_i)\),
\end{align*}
for any $ i\in\{0,1,2,...,k\} $. Consequently, we can use \eqref{uxidel}, Proposition \ref{MonFor}, and $ 0<\beta<\f{1}{2} $ to obtain that
\be
\begin{aligned}
\int_{B_{\ga s}(y_0)}\left|(y-x_i)\cdot\na u+\f{2u}{p-1}\right|^2\ud y&\leq\int_{B_{4s}(x_i)}\left|(y-x_i)\cdot\na u+\f{2u}{p-1}\right|^2\ud y\\
&\leq C(\vt_s(u,x_i)-\vt_{\beta s}(u,x_i))s^{n+2-\al_p}\\
&\leq C(n,p)\delta s^{n+2-\al_p},
\end{aligned}\label{Bgas}
\ee
for any $ i\in\{0,1,2,...,k\} $, and then
\be
\begin{aligned}
\int_{B_{\ga s}(y_0)}|(x_i-x_0)\cdot\na u|^2&\leq 2\int_{B_{\ga s}(y_0)}\left|(y-x_i)\cdot\na u+\f{2u}{p-1}\right|^2\ud y\\
&\quad\quad+2\int_{B_{\ga s}(y_0)}\left|(y-x_0)\cdot\na u+\f{2u}{p-1}\right|^2\ud y\\
&\leq C(n,p)\delta s^{n+2-\al_p},
\end{aligned}\label{xix0}
\ee
for any $ i\in\{1,2,...,k\} $. As a result,
\be
\int_{B_{\ga s}(y_0)}|V\cdot\na u|^2\leq C(\beta,n,p)\delta s^{n-\al_p},\label{vhat}
\ee
where $ V=\op{span}\{x_i-x_0\}_{i=1}^k $. For any $ y\in B_{\ga s}(y_0) $, let 
\be
\pi_L(y)=x_0+\sum_{i=1}^k\al_i(y)(x_i-x_0)\label{piLy}
\ee
be the point in $ L $ such that $
|\pi_L(y)-y|=\dist(y,L)\geq\beta s $. Using the first property of Lemma \ref{rhoind}, it follows that for any $ y\in B_{\ga s}(y_0) $, there holds
\be
|\al_i(y)|\leq\f{C(n)|\pi_L(y)-x_0|}{\beta s}\leq C(\beta,n),\text{ for any }i\in\{1,2,...,k\}.\label{aiy}
\ee
Combining \eqref{Bgas}, \eqref{xix0}, \eqref{piLy}, and \eqref{aiy}, we have
\be
\begin{aligned}
&\int_{B_{\ga s}(y_0)}\left|(y-\pi_L(y))\cdot\na u+\f{2u}{p-1}\right|^2\ud y\\
&\quad\quad\leq C\int_{B_{\ga s}(y_0)}\left|(y-x_0)\cdot\na u+\f{2u}{p-1}\right|^2\ud y\\
&\quad\quad\quad\quad+C\(\sum_{i=1}^k\int_{B_{\ga s}(y_0)}|(x_i-x_0)\cdot\na u|^2\)\\
&\quad\quad\leq C(\beta,n,p)\delta s^{n+2-\al_p}.
\end{aligned}\label{Bgasy0}
\ee
Since $ \pi_L $ is indeed the orthogonal projection onto $ L $, we have $ |\pi_L(y)-\pi_L(y_0)|\leq|y-y_0| $ for any $ y\in\R^n $. Then
\begin{align*}
&\int_{B_{\ga s}(y_0)}|(y_0-\pi_L(y_0))\cdot\na u|^2\\
&\quad\quad\leq C\int_{B_{\ga s}(y_0)}\left|(y-\pi_L(y))\cdot\na u+\f{2u}{p-1}\right|^2\\
&\quad\quad\quad\quad+C\int_{B_{\ga s}(y_0)}\left|((y-y_0)-(\pi_L(y)-\pi_L(y_0))\cdot\na u+\f{2u}{p-1}\right|^2\ud y\\
&\quad\quad\leq C\(\delta s^{n+2-\al_p}+\ga^2s^2\int_{B_{\ga s}(y_0)}|\na u|^2+\int_{B_{\ga s}(y_0)}|u|^2\)\\
&\quad\quad\leq C\[\delta s^{n+2-\al_p}+(\ga s)^{n+2-\al_p}\(\theta_{\ga s}(u,y_0)+(\theta_{\ga s}(u,y_0))^{\f{2}{p+1}}\)\]\\
&\quad\quad\leq C(\beta,\Lda,n,p)(\delta s^{n+2-\al_p}+(\ga s)^{n+2-\al_p}),
\end{align*}
where for the second inequality, we have used \eqref{Bgasy0}, for the third inequality, we have used H\"{o}lder's inequality, for the last inequality, we have used \eqref{thevt}, \eqref{NonDegThe}, and Proposition \ref{MonFor}. This, together with \eqref{vhat}, implies that
$$
(\ga s)^{\al_p-n}\int_{B_{\ga s}(y_0)}|V'\cdot\na u|^2\leq C(\beta,\Lda,n,p)(\ga^2+\delta\ga^{\al_p-n}),
$$
where
$$
V'=V\oplus\op{span}\left\{\f{y_0-\pi_L(y_0)}{|y_0-\pi_L(y_0)|}\right\}.
$$
Choosing $ \ga=\ga(\beta,\va,\Lda,n,p)>0 $ and $ \delta=\delta(\beta,\va,\Lda,n,p)>0 $ sufficiently small and applying Lemma \ref{kplus1} to the ball $ B_{\ga s}(y_0) $, we deduce that $ y_0\notin S_{\va,\delta's}^k(u)\cap B_s $, where $ \delta'=\delta'(\beta,\va,\Lda,n,p)>0 $. Now we can complete the proof.
\end{proof}

\begin{lem}\label{lempinch}
Let $ 0<\beta<\f{1}{10} $, $ k\in\{0,1,2,...,n-1\} $, $ 0<s\leq 1 $, $ \xi>0 $, and $ x\in\R^n $. Assume that $ u\in(H^1\cap L^{p+1})(B_{20s}(x)) $ is a stationary solution of \eqref{superequation}, satisfying $ \theta_{20s}(u,x)\leq\Lda $ and
$$
\sup_{y\in B_{2s}(x)}\vt_s(u,y)\leq E.
$$
There exists $ 0<\delta<1 $, depending only on $ \beta,\xi,\Lda,n $, and $ p $ such that if the set
$$
F=\{y\in B_{2s}(x):\vt_{\beta s}(u,y)>E-\delta\}
$$
$ \beta s $-effectively spans $ L\in\bA(n,k) $, then 
$$
\vt_{\beta s}(u,y)>E-\xi 
$$
for any $ y\in L\cap B_{2s}(x) $.
\end{lem}
\begin{proof}
Up to a scaling and a translation, we can assume that $ s=1 $ and $ x=0 $. If the statement is false, then there are $ \xi_0>0 $, $ 0<\beta_0<\f{1}{10} $, and a sequence of stationary solutions of \eqref{superequation}, denoted by $ \{u_j\}\subset(H^1\cap L^{p+1})(B_{20}) $ such that for any $ j\in\Z_+ $, the following properties hold.
\begin{itemize}
\item $ \theta_{20}(u_j,0)\leq\Lda $ and $ \sup_{y\in B_2}\vt_1(u_j,y)\leq E_j $.
\item $ F_j:=\{y\in B_2:\vt_{\beta_0}(u_j,y)>E_j-j^{-1}\} $ contains a subset $ \{x_{i,j}\}_{i=0}^k $, which $ \beta_0 $-effectively spans $ L_j\in\bA(n,k) $.
\item There exists $ y_j\in L_j\cap B_2 $ such that 
\be
\vt_{\beta_0}(u_j,y_j)\leq E_j-\xi_0.\label{Eeta}
\ee
\end{itemize}
By the assumptions as above, we have $ \sup_{j\in\Z_+}E_j\leq C(\Lda,n,p) $ since if not if $ E_j $ is chosen sufficiently large, then $ F_j=\emptyset $, a contradiction. If for some $ j\in\Z_+ $, $ E_j\leq\f{\xi_0}{2} $, then it contradicts to \eqref{Eeta}. As a result, we can assume that for any $ E_j\in(\f{\xi_0}{2},C(\Lda,n,p)] $ for any $ j\in\Z_+ $. By the definition of $ F_j $, we have
\be
\begin{gathered}
E_j-j^{-1}<\vt_{\beta_0}(u_j,x_{i,j})\leq E_j,\\
\vt_1(u_j,x_{i,j})-\vt_{\beta_0}(u_j,x_{i,j})<j^{-1}
\end{gathered}\label{ujxijcon}
\ee
for any $ j\in\Z_+ $ and $ i\in\{0,1,2,...,k\} $. Up to a subsequence, we can assume that $ x_{i,j}\to x_i $ and $ y_j\to y_0 $. This implies that $ L_j\to L\in\bA(n,k) $ with $ y_0\in L $. Moreover, we assume that $ E_j\to E\geq\f{\xi_0}{2} $, $
u_j\wc v $ weakly in $ (H^1\cap L^{p+1})(B_{20}) $, and $ m_{u_j}\wc^*\mu $ in $ \M(B_{20}) $. Combining \eqref{ujxijcon}, we apply Lemma \ref{lemusefre} to obtain that $ (v,\mu) $ is $ 0 $-symmetric at $ x_i $ in $ B_9(x_i) $, for any $ i\in\{0,1,2,...,k\} $. Using Lemma \ref{ConSpl}, \ref{ConSplm}, Corollary \ref{CorSpl}, \ref{CorSplm}, and their proofs, we deduce that $ (v,\mu) $ is invariant with respect to $ L $. Precisely, if we denote $ L=x_0+\op{span}\{x_i-x_0\}_{i=1}^k $, then $ (v,\mu) $ is $ k $-symmetric at $ x_0 $, with respect to $ V:=\op{span}\{x_i-x_0\}_{i=1}^k $. Consequently, it follows from \eqref{ujxijcon} that
$$
\vt_{\beta_0}(\mu,y_0)=\vt_{\beta_0}(\mu,x_0)=...=\vt_{\beta_0}(\mu,x_k)=E.
$$
On the other hand, we can take $ j\to+\ift $ in \eqref{Eeta} to obtain that $ 0\leq\vt_{\beta_0}(\mu,y_0)\leq E-\xi_0 $, a contradiction.
\end{proof}

\section{Reifenberg and \texorpdfstring{$ L^2 $}{}-best approximation theorems}\label{RandL2}

\subsection{Introduction of Reifenberg-type theorems} Before presenting the theorems, we give the following definition.

\begin{defn}\label{displacementk}
Let $ k\in\{0,1,2,...,n\} $, $ \mu\in\M(B_2) $, $ 0<r\leq 1 $, and $ x\in B_1 $. We define the $ k $-dimensional displacement
$$
D_{\mu}^k(x,r)=\min_{L\in\bA(n,k)}\(r^{-k-2}\int_{B_r(x)}\dist^2(y,L)\ud\mu(y)\).
$$
\end{defn}

Intuitively, the Reifenberg-type theorems imply that if $ D_{\mu}^k(x,r) $ is small in some sense, we can obtain estimates related to Ahlfors $ k $-regularity for the measure $ \mu $.

\begin{thm}[\cite{NV17}, Theorem 3.4]\label{Rei1}
Let $ k\in\{0,1,2,...,n\} $, $ 0<r\leq 1 $, and $ x_0\in\R^n $. Assume that $ \{B_{r_y}(y)\}_{y\in\cD}\subset B_{2r}(x_0) $ is a collection of pairwise disjoint balls with $ \cD\subset B_r(x_0) $ and 
$$
\mu:=\sum_{y\in\cD}\w_kr_y^k\delta_y 
$$
is the associated measure. There exist $ \delta_{Rei,1}>0 $ and $ C_{Rei,1}>0 $, depending only on $ n $ such that if for any $ x\in B_r(x_0) $ and $ 0<t<\f{r}{10} $, there holds
$$
\int_{B_t(x)}\(\int_0^tD_{\mu}^k(y,s)\f{\ud s}{s}\)\ud\mu(y)<\delta_{Rei,1}t^k,
$$
then
$$
\mu(B_r(x_0))=\sum_{y\in\cD}w_kr_y^k\leq C_{Rei,1}r^k.
$$
\end{thm}

\begin{thm}[\cite{NV17}, Theorem 3.3]\label{Rei2}
Let $ k\in\{0,1,2,...,n\} $, $ 0<r\leq 1 $, and $ x_0\in\R^n $. Assume that $ S\subset B_{2r}(x_0)\subset\R^n $ is an $ \HH^k $-measurable subset. There exist $ \delta_{Rei,2}>0 $ and $ C_{Rei,2}>0 $, depending only on $ n $ such that if for any $ x\in B_r(x_0) $ and $ 0<t\leq r $, we have
$$
\int_{B_t(x)}\(\int_0^tD_{\HH^k\llcorner S}^k(y,s)\frac{\ud s}{s}\)\ud(\HH^k\llcorner S)(y)<\delta_{Rei,2}t^k,
$$
then $ S\cap B_r(x_0) $ is $ k $-rectifiable such that for any $ x\in S $ and $ t>0 $, there holds 
$$
\HH^k(S\cap B_t(x))\leq C_{Rei,2}t^k.
$$
\end{thm}

\subsection{\texorpdfstring{$ L^2 $}{}-best approximation theorems}

Let $ 0<r\leq 1 $ and $ x\in B_2  $. Assume that $ u\in(H^1\cap L^{p+1})(B_{40}) $ is a stationary solution of \eqref{superequation}. We define
\be
W_r(u,x)=\vt_{2r}(u,x)-\vt_{r}(u,x)\label{Wrfunc}
\ee
be the difference of the density of scales $ r $ and $ 2r $. The main theorem in this subsection is the following $ L^2 $-best approximation theorem. In this theorem, we obtain a connection between the $ k $-dimensional displacement $ D_{\mu}^k $ and \eqref{Wrfunc}. Such a result allows us to apply Theorems \ref{Rei1} and \ref{Rei2} in the analysis of stationary solutions of \eqref{superequation}.

\begin{thm}\label{beta2}
Let $ \va>0 $, $ k\in\{0,1,2,...,n\} $, $ 0<r\leq 1 $, and $ x\in B_2 $. Assume that $ u\in(H^1\cap L^{p+1})(B_{40}) $ is a stationary solution of \eqref{superequation} such that $ \theta_{40}(u,0)\leq\Lda $. There exist $ C>0 $ depending only on $ \va,\Lda,n $, and $ p $ such that if for some $ \ga>0 $,
\be
\inf_{V\in\bG(n,k+1)}\(r^{\al_p-n}\int_{B_{4r}(x)}|V\cdot\na u|^2\)>\ga,\label{Vgeqassu}
\ee
then for any $ \mu\in\M(B_r(x)) $ with $ \mu(B_r(x))<+\ift $, there holds
\be
D_{\mu}^k(x,r)\leq C\ga^{-1}r^{-k}\int_{B_r(x)}W_r(u,y)\ud\mu(y).\label{L2bestesti}
\ee
\end{thm}

We first give the following direct corollary of this theorem.

\begin{cor}\label{beta22}
Let $ \va>0 $, $ k\in\{0,1,2,...,n\} $, $ 0<r\leq 1 $, and $ x\in B_2 $. Assume that $ u\in(H^1\cap L^{p+1})(B_{40}) $ is a stationary solution of \eqref{superequation} such that $ \theta_{40}(u,0)\leq\Lda $. There exist $ \delta>0 $ and $ C>0 $ depending only on $ \va,\Lda,n $, and $ p $ such that if $ u $ is $ (0,\delta) $-symmetric but not $ (k+1,\va) $-symmetric in $ B_{4r}(x) $, then for any $ \mu\in\M(B_r(x)) $ with $ \mu(B_r(x))<+\ift $, we have
\be
D_{\mu}^k(x,r)\leq Cr^{-k}\int_{B_r(x)}W_r(u,y)\ud\mu(y).\label{L2bestesti2}
\ee
\end{cor}

The proof of this corollary is a direct consequence of the lemma as follows.

\begin{lem}\label{geqlem}
Let $ \va>0 $. Assume that $ u\in(H^1\cap L^{p+1})(B_4) $ is a stationary solution of \eqref{superequation} satisfying $ \theta_4(u,0)\leq\Lda $. There exists $ \delta>0 $, depending only on $ \va,\Lda,n $, and $ p $ such that if $ u $ is $ (0,\delta) $-symmetric in $ B_4 $ but not $ (k+1,\va) $-symmetric, then
$$
\inf_{V\in\bG(n,k+1)}\(\int_{B_4}|V\cdot\na u|^2\)>\delta,
$$
\end{lem}

\begin{proof}[Proof of Corollary \ref{beta22}]
Up to a scaling and a translation, we assume that $ r=1 $ and $ x=0 $. If follows from Lemma \ref{geqlem} that the condition \eqref{Vgeqassu} is verified. As a result, we obtain \eqref{L2bestesti2}.
\end{proof}

\begin{proof}
Assume that the result is not true. There exist $ \va_0>0 $, a sequence of stationary solutions of \eqref{superequation}, denoted by $ \{u_i\}\subset(H^1\cap L^{p+1})(B_4) $, and $ \{V_i\}\subset\bG(n,k+1) $ such that for any $ i\in\Z_+ $, the following properties hold.
\begin{itemize}
\item $ \theta_4(u_i,0)\leq\Lda $, and $ u_i $ is $ (0,i^{-1}) $-symmetric but not $ (k+1,\va_0) $-symmetric in $ B_4 $. In particular, there exists a $ 0 $-symmetric pair $ (v_i,\eta_i)\in L_{\loc}^2(\R^n)\times\M(\R^n) $ such that 
\be
D_{0,4}((u_i,m_{u_i}),(v_i,\eta_i))<i^{-1}.\label{uimiimiusp}
\ee
\item $ u_i $ satisfies the inequality
\be
\int_{B_4}|V_i\cdot\na 
u_i|^2<i^{-1}.\label{B4B3delta}
\ee
\end{itemize}
In view of the assumption above and Lemma \ref{comsympro}, up to a subsequence, we can assume that $ V_i\to V\in\bG(n,k+1) $, and 
\be
\begin{aligned}
u_i&\wc w\text{ weakly in }(H^1\cap L^{p+1})(B_4),\\ m_{u_i}&\wc^*\mu\text{ in }\M(\R^n),\\
v_i&\wc v\text{ weakly in }L_{\loc}^2(\R^n),\\
\eta_i&\wc^*\eta\text{ in }\M(\R^n).
\end{aligned}\label{uietaimui}
\ee
Taking $ i\to+\ift $ in \eqref{uimiimiusp} and combining \eqref{uietaimui}, we have $ w=v $ a.e. in $ B_4 $ and $ \mu=\eta $ in $ \M(B_4) $. It follows from Lemma \ref{ConSym} and \ref{convergesymmetry} that $ (v,\eta) $ is $ (k+1) $-symmetric with respect to $ V $, and then $ (w,\mu) $ is $ (k+1) $-symmetric in $ B_4 $ with respect to $ V $. Moreover, given \eqref{B4B3delta} and Lemma \ref{lemusefre}, we deduce that the pair $ (w,\mu) $ is invariant with respect to $ V $. Combining the properties all above, we have that $ (w,\mu) $ is a $ (k+1) $-symmetric pair in $ B_4 $. For sufficiently large $ i\in\Z_+ $, $ D_{0,4}((u_i,m_{u_i}),(w,\mu))<\va_0 $, a contradiction.
\end{proof}

We divide the proof of Theorem \ref{beta2} into two steps. The first is to give the explicit representation for $ D_{\mu}^k(x,r) $ (Lemma \ref{NVlem7.4}). The second step is to establish some nondegeneracy for the right hand side of the formula \eqref{L2bestesti}, (Lemma \ref{L2lem1}).

To conduct the first step, for a probability measure $ \mu\in\M(B_1) $, we let
\be
x_{\op{cm}}:=x_{\op{cm}}(\mu)=\int_{B_1}y\ud\mu(y)\label{centermass}
\ee
be the center of mass for $ \mu $. 

\begin{defn}\label{defvj}
We inductively define $ \{(\lda_i,v_i)\}_{i=1}^n\subset\R_{\geq 0}\times\R^n $. Let 
$$
\lda_1:=\lda_1(\mu):=\max_{|v|^2=1}\int_{B_1}|(y-x_{\op{cm}})\cdot v|^2\ud\mu(y).
$$
Define $ v_1:=v_1(\mu) $ with $ |v_1|=1 $ as the unit vector obtaining this maximum. Given $ \{(\lda_j,v_j)\}_{j=1}^i $, we define $ (\lda_{i+1},v_{i+1}) $ by
$$
\lda_{i+1}:=\lda_{i+1}(\mu):=\max_{\substack{|v|^2=1,\,\,v\cdot v_i=0\\i=1,2,...,i}}\int_{B_1}|(y-x_{\op{cm}})\cdot v|^2\ud\mu(y),
$$
where $ v_{i+1}:=v_{i+1}(\mu) $ is a unit vector obtaining this maximum.
\end{defn}

By standard results of linear algebra, $ \{v_i\}_{i=1}^n $ is an orthonormal basis of $ \R^n $, and 
\be
\lda_1\geq\lda_2\geq...\geq\lda_n\geq 0.\label{ldaorder}
\ee
For $ j\in\{1,2,...,n\} $, we define affine subspaces
\be
L_j:=L_j(\mu):=x_{\op{cm}}+\op{span}\{v_i\}_{i=1}^j.\label{Vkmu}
\ee
Through the definitions of $ \{(\lda_i,v_i)\}_{i=1}^n $ as above, we can represent $ D_{\mu}^k(0,1) $ by the following lemma.

\begin{lem}[\cite{NV17}, Lemma 7.4]\label{NVlem7.4} 
Let $ \{(\lda_i,v_i)\}_{i=1}^n $ be given in Definition \ref{defvj}. If $ \mu\in\M(B_1) $ is a probability measure, then for any $ k\in\{1,2,...,n\} $,
$$
\min_{L\in\bA(n,k)}\int_{B_1}\dist^2(y,L)\ud\mu(y),
$$
attains its minimum at $ L_k $ defined by \eqref{Vkmu}. Precisely, there holds
$$
\min_{L\in\bA(n,k)}\int_{B_1}\dist^2(y,L)\ud\mu(y)=\int_{B_1}\dist^2(x,L_k)\ud\mu(y)=\sum_{i=k+1}^n\lda_i,
$$
where we have the convention that the right hand side of above is $ 0 $ if $ k=n $.
\end{lem}

Now, we turn to the second step for the proof of Theorem \ref{beta2}. By the minimum property of $ \{(\lda_i,v_i)\}_{i=1}^n $, we consider the Euler-Lagrange equations and obtain further properties.

\begin{lem}[\cite{NV17}, Lemma 7.5]\label{NVlem7.5}
Let $ \{(\lda_i,v_i)\}_{i=1}^n $ be given in Definition \ref{defvj}. If $ \mu\in\M(B_1) $ is a probability measure, then
\be
\int_{B_1}((y-x_{\op{cm}})\cdot v_i)(y-x_{\op{cm}})\ud\mu(y)=\lda_iv_i,\label{El}
\ee
for any $ i\in\{1,2,...,n\} $, where
\be
\lda_i=\int_{B_1}|(y-x_{\op{cm}})\cdot v_i|^2\ud\mu(y).\label{ldadef2}
\ee
\end{lem}

Applying this lemma and Proposition \ref{MonFor}, we have the following result.

\begin{lem}\label{L2lem1}
Let $ \mu\in\M(B_1) $ be a probability measure. Assume that $ u\in(H^1\cap L^{p+1})(B_{40}) $ is a stationary solution of \eqref{superequation}. Let $ \{(\lda_i,v_i)\}_{i=1}^n $ be given in Definition \ref{defvj}. There exists a constant $ C>0 $, depending only on $ n $ and $ p $ such that
$$
\lda_i\int_{B_4}|v_i\cdot\na u|^2\leq C\int_{B_1}W_1(u,y)\ud\mu(y),
$$
for any $ i\in\{1,2,...,n\} $.
\end{lem}
\begin{proof}
Fix $ i\in\{1,2,...,n\} $. If $ \lda_i=0 $, there is nothing to prove, so we assume that $ \lda_i>0 $. By Lemma \ref{NVlem7.5}, we take the inner product on both sides of \eqref{El} by $ \na u(z) $ with $ z\in B_4 $ to obtain that 
\be
\lda_i(v_i\cdot\na u(z))=\int_{B_1}((y-x_{\op{cm}})\cdot v_i)((y-x_{\op{cm}})\cdot\na u(z))\ud\mu(y),\label{mucm}
\ee
for any $ k\in\{1,2,...,n\} $. Using \eqref{centermass}, we have 
\be
\int_{B_1}((y-x_{\op{cm}})\cdot v_i)\ud\mu(y)=0,\label{cenmass}
\ee
which implies that
$$
\int_{B_1}((y-x_{\op{cm}})\cdot v_i)\((z-x_{\op{cm}})\cdot\na u(z)+\f{2u(z)}{p-1}\)\ud\mu(y)=0,
$$
for any $ z\in B_4 $. This, together with \eqref{mucm} and \eqref{cenmass}, implies that
\be
\lda_i(v_i\cdot\na u(z))=\int_{B_1}((y-x_{\op{cm}})\cdot v_i)\((y-z)\cdot\na u(z)-\f{2u(z)}{p-1}\)\ud\mu(y).\label{ldaknau}
\ee
By Cauchy's inequality, for any $ z\in B_4 $, we deduce that
$$
\lda_i^2|(v_i\cdot\na u(z))|^2\leq\lda_i\int_{B_1}\left|(z-y)\cdot\na u(z)+\f{2u(z)}{p-1}\right|^2\ud\mu(y).
$$
Integrating with respect to $ z\in B_4 $ on both sides, we obtain
\be
\begin{aligned}
&\lda_i\int_{B_4}|v_i\cdot\na u|^2\leq\int_{B_1}\(\int_{B_4}\left|(z-y)\cdot\na u+\f{2u}{p-1}\right|^2\ud z\)\ud\mu(y).
\end{aligned}\label{zmuz}
\ee
It follows from Corollary \ref{coruse} that
\begin{align*}
\int_{B_4}\left|(z-y)\cdot\na u+\f{2u}{p-1}\right|^2\ud z&\leq C\int_{B_5(y)}\left|(z-y)\cdot\na u+\f{2u}{(p-1)}\right|^2\ud z\\
&\leq C(n,p)(\vt_2(u,y)-\vt_1(u,y))\\
&=C(n,p)W_1(u,y).
\end{align*}
Combining \eqref{zmuz}, we can complete the proof.
\end{proof}

Combining all the necessary ingredients above, we can now prove the main theorem of this subsection.

\begin{proof}[Proof of Theorem \ref{beta2}]
By a translation, a scaling, and a normalization, we can assume that $ r=1 $, $ x=0 $, and $ \mu $ is a probability measure on $ B_1 $. Since $ D_{\mu}^n(0,1)=0 $, we can also assume that $ k\leq n-1 $. Letting $ \{(\lda_i,v_i)\}_{i=1}^n $ be as in the Definition \ref{defvj}, with $ \{L_j\}_{j=1}^n $ given by \eqref{Vkmu}, we define $ \{V_j\}_{j=1}^n\subset\R^n $ as subspaces of $ \R^n $ such that $
V_j=L_j-x_{\op{cm}}=\op{span}\{v_i\}_{i=1}^j $ with $ j\in\{1,2,...,n\} $. Using \eqref{ldaorder} and Lemma \ref{NVlem7.4}, we have
\be
\min_{L\subset\bA(n,k)}\int_{B_1}\dist^2(y,L)\ud\mu(y)=\sum_{i=k+1}^n\lda_i\leq(n-k)\lda_{k+1}.\label{dismufor}
\ee
We only need to estimate $ \lda_{k+1} $. It follows from Lemma \ref{L2lem1} that 
$$
\lda_i\int_{B_4}|v_i\cdot\na u|^2\leq C\int_{B_1} W_1(u,y)\ud\mu(y),
$$
for any $ i\in\{1,2,...,n\} $. Summing both sides of the above formula with $ i\in\{1,2,...,k+1\} $, we obtain 
$$
\sum_{i=1}^{k+1}\(\lda_i\int_{B_4}|v_i\cdot\na u|^2\)\leq C\int_{B_1}W_1(u,y)\ud\mu(y).
$$
Again, by \eqref{ldaorder}, we have
\be
\lda_{k+1}\int_{B_4}|V_{k+1}\cdot\na u|^2\leq\sum_{i=1}^{k+1}\(\lda_i\int_{B_4}|v_i\cdot\na u|^2\)\leq C\int_{B_1}W_1(u,y)\ud\mu(y).\label{7.31}
\ee
In view of \eqref{Vgeqassu}, we have
$$
\int_{B_4}|V_{k+1}\cdot\na u|^2>\ga.
$$
This, together with \eqref{7.31}, implies that
$$
\ga\lda_{k+1}\leq\lda_{k+1}\int_{B_4}|V_{k+1}\cdot\na u|^2\leq C\int_{B_1}W_1(u,y)\ud\mu(y),
$$
and then
$$
\lda_{k+1}\leq C\ga^{-1}\int_{B_1}W_1(u,y)\ud\mu(y).
$$
Combining \eqref{dismufor}, we can complete the proof.
\end{proof}

\section{Covering lemmas}\label{coveringlemma}

In this section, we establish several crucial covering lemmas using the Reifenberg-type theorems presented in the previous section. These covering lemmas play vital roles in the proofs of main theorems presented in this paper. Throughout this section, we will follow the original arguments in \cite{NV18} while making necessary modifications based on different forms of energy density in our paper.

\begin{lem}[Main covering lemma]\label{maincover}
Let $ k\in\{1,2,...,n-1\} $, $ \va>0 $, $ 0<r<R\leq 1 $, and $ x_0\in B_2 $. Assume that $ u\in(H^1\cap L^{p+1})(B_{40}) $ is a stationary solution of \eqref{superequation}, satisfying $ \theta_{40}(u,0)\leq\Lda $. There exist $ \delta>0 $, depending only on $ \va,\Lda,n,p $, and a finite covering of $ S_{\va,\delta r}^k(u)\cap B_R(x_0) $, denoted by $ \{B_r(x)\}_{x\in\cC} $ such that
$$
S_{\va,\delta r}^k(u)\cap B_R(x_0)\subset\bigcup_{x\in\cC}B_r(x),
$$
with $ (\#\cC)r^k\leq CR^k $, where $ C>0 $ depends only on $ \va,\Lda,n $, and $ p $.
\end{lem}

In the next section, we will see that the estimates of the $ r $-neighborhood of $ S_{\va,r}^k(u)\cap B_1 $ can be directly derived from this covering lemma. To establish this, we require the following auxiliary lemma.

\begin{lem}\label{Cover2}
Let $ \va>0 $, $ k\in\{1,2,...,n-1\} $, $ 0<r<R\leq 1 $, and $ x_0\in B_2 $. Assume that $ u\in(H^1\cap L^{p+1})(B_{40}) $ is a stationary solution of \eqref{superequation}, satisfying $ \theta_{40}(u,0)\leq\Lda $. There exists $ \delta>0 $, depending only on $ \va,\Lda,n,p $ such that $ S_{\va,\delta r}^k(u)\cap B_R(x_0) $ has a finite covering, denoted by $ \{B_{r_x}(x)\}_{x\in\cC} $ with
$$
S_{\va,\delta r}^k(u)\cap B_R(x_0)\subset\bigcup_{x\in\cC}B_{r_x}(x),\quad\#\cC<+\ift.
$$
Moreover, the following properties hold.
\begin{enumerate}[label=$(\theenumi)$]
\item $ B_{2r_x}(x)\subset B_{2R}(x_0) $, $ r_x\geq r $ for any $ x\in\cC $, and 
$$
\sum_{x\in\cC}r_x^k\leq C_{\op{I}}R^k,
$$
where $ C_{\op{I}}>0 $ depends only on $ n $.
\item For any $ x\in\cC $, either $ r_x=r $, or
$$
\sup_{y\in B_{2r_x}(x)}\vt_{r_x}(u,y)\leq E-\delta,
$$
where 
\be
E:=E(x_0,R):=\sup_{y\in B_{2R}(x_0)}\vt_{R}(u,y).\label{Evtx}
\ee
\end{enumerate}
\end{lem}

Intuitively, this lemma is a weaker form of Lemma \ref{maincover}. To modify this weaker covering, we do not need to change those balls with radius $ r $ and will apply the lemma inductively to those balls $ B_{r_x}(x) $ with $ r_x>r $.

\begin{proof}[Proof of Lemma \ref{maincover} given Lemma \ref{Cover2}]
Let $ E $ be given by \eqref{Evtx}. By \eqref{thevt} and Proposition \ref{MonFor}, we obtain 
\be
0\leq E\leq C(\Lda,n,p).\label{ECldanp}
\ee
For sufficiently small $ \delta=\delta(\va,\Lda,n,p,\rho)>0 $, we will inductively construct a covering of $ S_{\va,\delta r}^k(u)\cap B_R(x_0) $, denoted by $ \{B_{r_x}(x)\}_{x\in\cC_i} $ with $ i\in\Z_+ $ such that
$$
S_{\va,\delta r}^k(u)\cap B_R(x_0)\subset\bigcup_{x\in\cC_i}B_{r_x}(x)=\bigcup_{x\in\cC_i^{(1)}}B_{r_x}(x)\cup\bigcup_{x\in\cC_i^{(2)}}B_{r_x}(x),
$$
and the following properties hold.
\begin{enumerate}[label=$(\theenumi)$]
\item If $ x\in\cC_i^{(1)} $, then $ r_x=r $.
\item If $ x\in\cC_i^{(2)} $, then $ r_x>r $, and
\be
\sup_{y\in B_{2r_x}(x)}\vt_{r_x}(u,y)\leq E-i\delta.\label{Ci2cover}
\ee
\item We have the estimate
\be
\sum_{x\in\cC_i}r_x^k\leq (1+C_{\op{I}}(n))^iR^k.\label{Ciestim}
\ee
\end{enumerate}
First, assume that we have already constructed such coverings. For $ i=\lfloor\delta^{-1}E\rfloor+1 $, we deduce from $ \eqref{Ci2cover} $ that $ \cC_i^{(2)}=\emptyset $, which implies that $ \cC_i=\cC_i^{(1)} $ and
$$
S_{\va,\delta r}^k(u)\cap B_R(x_0)\subset\bigcup_{x\in\cC_i^{(1)}}B_r(x).
$$
It follows from \eqref{ECldanp} and \eqref{Ciestim} that
$$
(\#\cC_i)r^k\leq (1+C_{\op{I}}(n))^{\lfloor\delta^{-1}E\rfloor+1}R^k\leq C(\va,\Lda,n,p)R^k.
$$
Now we see that the covering $ \{B_r(x)\}_{x\in\cC_i} $ is what we need.

For $ i=1 $, we can apply Lemma \ref{Cover2} to $ B_R(x_0) $ and the results follow directly. Assume that such covering exists for some $ i\in\Z_+ $. We will construct the covering for $ i+1 $. For any $ x\in\cC_i^{(2)} $, we apply Lemma \ref{Cover2} to obtain a covering of $ S_{\va,\delta r}^k(u)\cap B_{r_x}(x) $ by $ \{B_{r_y}(y)\}_{y\in\cC_{x,i}} $ such that 
$$
S_{\va,\delta r}^k(u)\cap B_{r_x}(x)\subset\bigcup_{y\in C_{x,i}}B_{r_y}(y)=\bigcup_{y\in C_{x,i}^{(1)}}B_{r_y}(y)\cup\bigcup_{y\in C_{x,i}^{(2)}}B_{r_y}(y),
$$
and the following properties hold.
\begin{itemize}
\item For any $ y\in\cC_{x,i} $, $ B_{2r_y}(y)\subset B_{2r_x}(x) $.
\item If $ y\in\cC_{x,i}^{(1)} $, then $ r_y=r $.
\item If $ y\in\cC_{x,i}^{(2)} $, then 
$$
\sup_{\zeta\in B_{2r_y}(y)}\vt_{r_y}(u,\zeta)\leq\sup_{\zeta\in B_{2r_x}(x)}\vt_{r_x}(u,\zeta)-\delta\leq E-(i+1)\delta,
$$
where we have used \eqref{Ci2cover} for the second inequality.
\item We have the estimate
\be
\sum_{y\in\cC_{x,i}}r_y^k\leq C_{\op{I}}(n)r_x^k.\label{CxiCx2}
\ee
\end{itemize}
Define $ \{\cC_{i+1}^{(j)}\}_{j=1,2} $ as
$$
\cC_{i+1}^{(1)}=\cC_i^{(1)}\cup\bigcup_{x\in \cC_i^{(2)}}\cC_{x,i}^{(1)},\quad\cC_{i+1}^{(2)}=\bigcup_{x\in \cC_i^{(2)}}\cC_{x,i}^{(2)},\quad\cC_{i+1}=\cC_{i+1}^{(1)}\cup\cC_{i+1}^{(2)}.
$$
It follows from this definition that
\begin{align*}
\sum_{x\in\cC_{i+1}}r_x^k&\leq\sum_{x\in\cC_i^{(1)}}r_x^k+\sum_{x\in\cC_i^{(2)}}\sum_{y\in\cC_{x,i}}r_y^k\\
&\leq\sum_{x\in\cC_i^{(1)}}r_x^k+\sum_{x\in\cC_i^{(2)}}C_{\op{I}}(n)r_x^k\\
&\leq (1+C_{\op{I}}(n))\(\sum_{x\in\cC_i}r_x^k\)\\
&\leq (1+C_{\op{I}}(n))^{i+1}R^k,
\end{align*}
where we have used \eqref{CxiCx2} for the second inequality and used \eqref{Ciestim} for the third inequality. Now, we can complete the proof.
\end{proof}

Now, let us prove Lemma \ref{Cover2}. The most essential part of the proof is the following lemma, and we have to apply Theorem \ref{Rei1} to establish it.

\begin{lem}\label{cover1}
Let $ \va>0 $, $ k\in\{1,2,...,n-1\} $, $ 0<\rho<\f{1}{100} $, $ 0<r<R\leq 1 $, and $ x_0\in B_2 $. Assume that $ u\in (H^1\cap L^{p+1})(B_{40}) $ is a stationary solution of \eqref{superequation}, satisfying $ \theta_{40}(u,0)\leq\Lda $. There exist $ \delta>0 $ depending only on $ \va,\Lda,n,p,\rho $, and $ C_{\op{II}}>0 $, depending only on $ n $, and a finite covering $ \{B_{r_x}(x)\}_{x\in\cD} $ of $ S_{\va,\delta r}^k(u)\cap B_R(x_0) $ with
$$
S_{\va,\delta r}^k(u)\cap B_R(x_0)\subset\bigcup_{x\in\cD}B_{r_x}(x),\quad\#\cD<+\ift,
$$
such that $ r_x\geq r $ for any $ x\in\cD $, and
\be
\sum_{x\in\cD}r_x^k\leq C_{\op{II}}R^k.\label{estiCov}
\ee
Moreover, for any $ x\in\cD $, $ B_{2r_x}(x)\subset B_{2R}(x_0) $, and either of the following conditions holds.
\begin{enumerate}[label=$(\theenumi)$]
\item $ r_x=r $.
\item There exists $ L(x,r_x)\in\bA(n,k-1) $ such that
$$ 
\{y\in B_{2r_x}(x):\vt_{\f{\rho r_x}{20}}(u,y)>E-\delta\}\subset B_{\f{\rho r_x}{10}}(L(x,r_x))\cap B_{2r_x}(x),
$$
where $ E=E(x_0,R) $ is defined by \eqref{Evtx}.
\end{enumerate}
\end{lem}

Before we start the proof, for $ B_{2s}(x)\subset B_{2R}(x_0) $, $ 0<\rho<\f{1}{100} $, and $ \delta>0 $, we define
\begin{align*}
F_{\delta}(x,s)&:=\{y\in B_{2s}(x):\vt_{\f{\rho s}{20}}(u,y)>E-\delta\},\\
F_{\delta}'(x,s)&:=\{y\in B_{2s}(x):\vt_{s}(u,y)-\vt_{\f{\rho s}{20}}(u,y)<\delta\}.
\end{align*}
By the definition of $ E $, we have $ \vt_s(u,y)\leq E $ for any $ y\in B_{2s}(x) $. Additionally, $ F_{\delta}(x,s)\subset F_{\delta}'(x,s) $.

\subsection{Proof of Lemma \ref{cover1}} 

\subsubsection*{Step 0. Preparations of the proof}

Up to a translation, we can assume that $ x_0=0 $. By simply covering arguments, we also assume that $ 0<r<\f{R}{100} $ and $ 0<R<\f{1}{100} $. For fixed $ \va>0,0<\rho<\f{1}{100} $, there exists $ \ell\in\Z_+ $ such that 
\be
\rho^{\ell}R\leq r<\rho^{\ell-1}R.\label{choiceofrRell}
\ee
Consider the ball $ B_R $ and the set $ F_{\delta}(0,R) $. If there exists $ L(0,R)\in\bA(n,k-1) $ such that 
$$
F_{\delta}(0,R)\subset B_{\f{\rho R}{10}}(L(0,R))\cap B_{2R},
$$
we define the covering by $ \{B_R\} $. It already satisfies all the properties in Lemma \ref{cover1}. For this reason, without loss of generality, we assume that there exists $ L'(0,R)\in\bA(n,k) $, which is $ \f{\rho R}{20} $-effectively spanned by $ F_{\delta}(0,R) $. 

We construct a covering with $ i\in\{1,2,...,\ell\} $ for $ S_{\va,\delta r}^k(u)\cap B_R $ in the form
\be
S_{\va,\delta r}^k(u)\cap B_R\subset\bigcup_{x\in\cB_i}B_{r_x}(x)\cup\bigcup_{x\in\cG_i}B_{r_x}(x),\label{CoverS}
\ee
where we call that $ \cB_i $ is the collection of centers of \emph{bad ball}s and $ \cG_i $ is the collection of centers of \emph{good ball}s. We let $ \cD_i=\cB_i\cup\cG_i $. Moreover, for any $ i\in\{1,2,...,\ell\} $, the covering given by \eqref{CoverS} satisfies the properties as follows.
\begin{enumerate}[label=(\Alph*)]
\item \label{p1}If $ x\in\cB_i $, then $ r_x\geq\rho^iR $, and there exists $ L(x,r_x)\in\bA(n,k-1) $ such that 
$$
F_{\delta}(x,r_x)\subset B_{\f{\rho r_x}{10}}(L(x,r_x))\cap B_{2r_x}(x),
$$
where $ \delta=\delta(\va,\Lda,n,p,\rho)>0 $ is to be determined later.
\item \label{p2}If $ x\in\cG_i $ and $ i\in\{1,2,...,\ell-1\} $, then $ r_x=\rho^iR $, and $ F_{\delta}(x,r_x) $ $ \f{\rho r_x}{20} $-effectively spans some $ L'(x,r_x)\in\bA(n,k) $. If $ x\in\cG_{\ell} $, then $ r_x=r $.
\item \label{p3}The balls in the collection $ \{B_{\f{r_x}{10}}(x)\}_{x\in\cD_i} $ are pairwise disjoint. For any $ x\in\cD_i $, we have $ B_{2r_x}(x)\subset B_{2R}(x_0) $ .
\item \label{p4}For any $ x\in\cD_i $, we have
$$
\vt_{\f{r_x}{20}}(u,x)>E-\xi,
$$
where $ \xi=\xi(\va,\Lda,n,p,\rho)>0 $ is to be determined later.
\item \label{p5}For any $ x\in\cD_i $ and $ t\in[r_x,R] $
\be
\inf_{V\in\bG(n,k+1)}\(t^{\al_p-n}\int_{B_t(x)}|V\cdot\na u|^2\)>\ga,\label{Vgeq}
\ee
where $ \ga=\ga(\va,\Lda,n,p,\rho)>0 $.
\item \label{p6}If $ i=\ell $, then
\be
\sum_{x\in\cD_{\ell}}r_x^k\leq C_{\op{II}}(n)R^k.\label{BlGl}
\ee
\end{enumerate}

Note that if we have already obtain the above covering, then $ \{B_{r_x}(x)\}_{x\in\cD_{\ell}} $ is exactly what we need in this lemma. To begin with, we will inductively construct such a covering satisfying the above properties \ref{p1}-\ref{p5} in Step 1 to Step 3, and prove the estimate \eqref{BlGl} after this construction by using theorems given in \S \ref{RandL2} in Step 4.

\subsubsection*{Step 1. The base of induction.} By the assumption on the ball $ B_R $, $ F_{\delta}(0,R) $ $ \f{\rho R}{20} $-effectively spans $ L'(0,R)\in\bA(n,k) $. Choosing $ \delta=\delta(\va,\Lda,n,p,\rho)>0 $ sufficiently small, since $ F_{\delta}(0,R)\subset F_{\delta}'(0,R) $, we can apply Proposition \ref{Fprop} with $ \beta=\f{\rho}{20} $ and $ s=R $ to deduce that
$$
S_{\va,\delta r}^k(u)\cap B_R\subset S_{\va,\delta R}^k(u)\cap B_R\subset B_{\f{\rho R}{10}}(L'(0,R))\cap B_R.
$$
We construct a finite covering of $ B_{\f{\rho R}{10}}(L'(0,R))\cap B_R $ by balls $ \{B_{\rho R}(x)\}_{x\in\cD_1} $, where $ \cD_1 $ is the set of centers of balls such that the following properties hold.
\begin{itemize}
\item $ \cD_1\subset L'(0,R)\cap B_{\f{3R}{2}} $, and for any $ x \in \mathcal{D}_1 $, $ S_{\va,\delta r}^k\cap B_R\cap B_{\f{\rho R}{2}}(x)\neq\emptyset $.
\item For any $ x,y\in\cD_1 $ with $ x\neq y $, we have $ B_{\f{\rho R}{10}}(x)\cap B_{\f{\rho R}{10}}(y)=\emptyset $.
\end{itemize}
Since $ F_{\delta}(0,R) $ $ \f{\rho R}{20} $-effectively spans $ L'(0,R) $, it follows from Lemma \ref{lempinch} with $ \beta=\f{\rho}{20} $ and $ s=R $ that if $
\delta=\delta(\xi,\va,\Lda,n,p,\rho)>0 $ is sufficiently small, then
\be
\vt_{\f{\rho R}{20}}(u,x)>E-\xi,\label{rhoRux}
\ee
for any $ x\in\cD_1 $, where $ \xi>0 $ is to be determined later. Considering the balls in $ \{B_{r_x}(x)\}_{x\in\cD_1}=\{B_{\rho R}(x)\}_{x\in\cD_1} $, we divide them into collections of bad and good balls. Indeed, we define $ \cD_1=\cB_1\cup\cG_1 $ such that the following properties hold.
\begin{itemize}
\item If $ x\in\cB_1 $, then there exists a $ L(x,\rho R)\in\bA(n,k-1) $ such that 
$$
F_{\delta}(x,\rho R)\subset B_{\f{\rho^2R}{10}}(L(x,\rho R))\cap B_{\rho R}(x).
$$
\item If $ x\in\cG_1 $, then $ F_{\delta}(x,\rho R) $ $ \f{\rho^2R}{20} $-effectively spans $ L'(x,\rho R)\in\bA(n,k) $.
\end{itemize}
By the construction of $ \cB_1 $ and $ \cG_1 $, and the assumption that $ 0<\rho<\f{1}{100} $, for $ i=1 $, properties \ref{p1}-\ref{p3} hold. Additionally, in this case, \ref{p4} follows from \eqref{rhoRux}. It remains to show \ref{p5}. Precisely, for $ y\in\cD_1 $, we fix $ t\in[\rho R,R] $, and will show the estimate \eqref{Vgeq}. By the choice of $ \cD_1 $, there exists $ y'\in S_{\va,\delta r}^k(u)\cap B_{\f{\rho R}{2}}(y) $. Using Corollary \ref{HSVStraCor}, there exists $ \xi'=\xi'(\va,\Lda,n,p)>0 $ such that 
\be
y'\in S_{\va,\delta r}^k(u)\cap B_{\f{\rho R}{2}}(y)\subset S_{[\op{HSV}];\xi',\delta r}^k(u)\cap B_{\f{\rho R}{2}}(y).\label{HSVxprime}
\ee
Given Proposition \ref{MonFor}, we have $ \vt_t(u,y')\leq C_0(\Lda,n,p) $. Let $ N_0=N_0(\va,\Lda,n,p)\in\Z_+ $ be such that $ C(n)N_0^{-1}C_0(\Lda,n,p)<\xi' $ for some $ C(n)>0 $ to be determined. Arguing as the proof of Lemma \ref{rxlogr}, there exists $ r'\in[\f{1}{2^N},\f{1}{2}] $ such that 
$$
\vt_{r't}(u,y')-\vt_{\f{r't}{2}}(u,y')<\xi'.
$$
Using \eqref{choiceofrRell}, we can choose $ \delta=\delta(\va,\Lda,n,p,\rho)>0 $ such that $ \delta r<\f{r't}{2} $. As a result, we deduce from \eqref{HSVxprime} that
$$
\inf_{V\in\bG(n,k+1)}\[\(\f{r't}{2}\)^{\al_p-n}\int_{B_{\f{r't}{2}}(y')}|V\cdot\na u|^2\]>\xi'.
$$
Again, by using \eqref{HSVxprime}, we deduce that $ B_{\f{r't}{2}}(y')\subset B_t(y) $, and then 
\be
\begin{aligned}
t^{\al_p-n}\int_{B_t(y)}|V\cdot\na u|^2&\geq \(\f{2}{r'}\)^{\al_p-n}\(\f{r't}{2}\)^{\al_p-n}\int_{B_{\f{r't}{2}}(y')}|V\cdot\na u|^2\\
&\geq C(\xi',\va,\Lda,n,p)\xi':=\ga,
\end{aligned}\label{geqga}
\ee
for any $ V\in\bG(n,k+1) $. As a result, when $ i=1 $, the property \ref{p5} follows.

\subsubsection*{Step 2. Proof of results from \texorpdfstring{$ i $}{} to \texorpdfstring{$ i+1 $}{} for \texorpdfstring{$ i\in\{1,2,...,\ell-2\} $}{}} 

Now, we assume that the properties \ref{p1}-\ref{p5} are valid for such $ i $, and we want to show that they are also true for $ i+1 $. Given the proof of the base of induction, we recover the balls with centers in $ \cG_i $ and do not change those bad balls. Fix $ x\in\cG_i $, we have that $ r_x=\rho^iR $ and $ B_{\rho^i R}(x) $ is a good ball. By the assumption of induction, $ F_{\delta}(x,r_x) $ $ \f{\rho^{i+1}R}{20} $-effectively spans $ L'(x,r_x)\in\bA(n,k) $. We can choose $ \delta=\delta(\va,\Lda,n,p,\rho)>0 $ sufficiently small and apply Proposition \ref{Fprop} with $ \beta=\f{\rho}{10} $ and $ s=\rho^iR $ to obtain
\be
S_{\va,\delta r}^k(u)\cap B_{\rho^iR}(x)\subset B_{\f{\rho^{i+1}R}{10}}(L'(x,\rho^iR))\cap B_{\rho^iR}(x).\label{SvaR5}
\ee
Define
\be
A_i:=\(\bigcup_{y\in\cG_i}\(L'(y,\rho^iR)\cap B_{\f{3\rho^iR}{2}}(y)\)\)\backslash\(\bigcup_{y\in\cB_i}B_{\f{4r_y}{5}}(y)\).\label{defAi}
\ee
In view of \eqref{CoverS} and \eqref{SvaR5}, we have
\be
(S_{\va,\delta r}^k(u)\cap B_R)\backslash\(\bigcup_{y\in\cB_i}B_{r_y}(y)\)\subset B_{\f{\rho^{i+1}R}{5}}(A_i).\label{SvadeltarBry}
\ee
Let $ \cD_{A_i}\subset A_i $ be a maximal subset of points such that 
$$ 
\dist(y,z)\geq\f{\rho^{i+1}R}{5},\quad\text{for any }y,z\in\cD_{A_i}.
$$
It follows from this property and the definition of $ A_i $ that balls in the collection 
$$
\{B_{\f{\rho^{i+1}R}{10}}(y)\}_{y\in\cD_{A_i}}\cup\{B_{\f{r_y}{10}}(y)\}_{y\in\cB_i} 
$$
are pairwise disjoint. By the maximality of $ D_{A_i} $, we deduce that
$$
(S_{\va,\delta r}^k(u)\cap B_R)\backslash\(\bigcup_{y\in\cB_i}B_{r_y}(y)\)\subset\bigcup_{y\in\cD_{A_i}}B_{\f{2\rho^{i+1}R}{5}}(y).
$$
We can eliminate the balls in $ \{B_{\rho^{i+1}R}(y)\}_{y\in\cD_{A_i}} $ such that 
\be
S_{\va,\delta r}^k(u)\cap B_R\cap B_{\f{\rho^{i+1}R}{2}}(y)\neq\emptyset,\text{ for any }y\in\cD_{A_i}.\label{DAiSva}
\ee
Now we classify the centers of balls in the collection $ \cD_{A_i} $ into $ \wt{\cB}_{i+1} $ and $ \wt{\cG}_{i+1} $ such that the following properties hold.
\begin{itemize}
\item If $ y\in\wt{\cB}_{i+1} $, then there exists $ L(y,\rho^{i+1}R)\in\bA(n,k-1) $ such that 
$$
F_{\delta}(y,\rho^{i+1}R)\subset B_{\f{\rho^{i+2}R}{10}}(L(y,\rho^{i+1}R))\cap B_{\rho^{i+1}R}(y).
$$
\item If $ y\in\wt{\cG}_{i+1} $, then $ F_{\delta}(y,\rho^{i+1}R) $ $ \f{\rho^{i+2}R}{20} $-effectively spans $ L'(y,\rho^{i+1}R)\in\bA(n,k) $.
\end{itemize}
For any $ y\in\wt{\cB}_{i+1}\cup\wt{\cG}_{i+1} $, we set $ r_y=\rho^{i+1}R $. Define
$$
\cB_{i+1}=\cB_i\cup\wt{\cB}_{i+1},\quad\cG_{i+1}=\wt{\cG}_{i+1}.
$$
Given the above setting, we also see that if $ y\in\cD_{A_i} $, then $ B_{2r_y}(y)\subset B_{2R}(x_0) $. Consequently, properties \ref{p1}-\ref{p3} are satisfied for $ i+1 $, and it remains to show \ref{p4} and \ref{p5}. For fixed $ y\in\wt{\cB}_{i+1}\cup\wt{\cG}_{i+1} $, there exists $ x\in\cG_i $, such that
\be
y\in L'(x,\rho^{i}R)\cap B_{\f{3\rho^iR}{2}}(x),\label{yLprimexesti}
\ee
where $ F_{\delta}(x,\rho^{i}R) $ $ \f{\rho^{i+1}R}{20} $-effectively spans $ L'(x,\rho^iR)\in\bA(n,k) $. Moreover, by \ref{p3} with $ i $, Proposition \ref{MonFor}, the definition of $ E=E(0,R) $, and the choice of $ \cD_{A_i} $, we have $ B_{2\rho^iR}(x)\subset B_{2R} $ and
\be
\sup_{z\in B_{2\rho^iR}(x)}\vt_{\rho^iR}(u,z)\leq \sup_{z\in B_{2R}}\vt_R(u,z)\leq E.\label{simiuzE}
\ee
Combining \eqref{yLprimexesti}, we apply Lemma \ref{lempinch} with $ \beta=\f{\rho}{20} $ and $ s=\rho^iR $ such that if $ \delta=\delta(\va,\xi,\Lda,n,p,\rho)>0 $ is sufficiently small, there holds
$$
\vt_{\f{r_y}{20}}(u,y)=\vt_{\f{\rho^{i+1}R}{20}}(u,y)>E-\xi,
$$
which implies \ref{p4} for $ i+1 $. Finally, we prove that the point $ y\in\cD_{A_i} $ satisfies \ref{p5}. Fix $ t\in[\rho^{i+1}R,R] $. By \eqref{DAiSva}, there exists 
$$
y'\in S_{\va,\delta r}^k(u)\cap B_R\cap B_{\f{\rho^{i+1}R}{2}}(y).
$$
Applying similar arguments in the proof of \eqref{geqga}, the result follows directly.

\subsubsection*{Step 3. Construction for \texorpdfstring{$ \cD_{\ell} $}{}} 

We assume that the covering \eqref{CoverS} has been constructed for $ i=\ell-1 $, satisfying \ref{p1}-\ref{p5}. For $ x\in\cG_{\ell-1} $, $ B_{r_x}(x) $ is a good ball, and $ r_x=\rho^{\ell-1}R $. It follows from the result for $ \cG_{\ell-1} $ that $ F_{\delta}(x,\rho^{\ell-1}R) $ $ \f{\rho^{\ell}R}{20} $-effectively spans $ L'(x,\rho^{\ell-1}R)\in\bA(n,k) $. Choosing $ \delta=\delta(\va,\Lda,n,p,\rho)>0 $ sufficiently small, it follows from Proposition \ref{Fprop} with $ \beta=\f{\rho}{10} $ and $ s=\rho^{\ell-1}R $ that
$$
S_{\va,\delta r}^k(u)\cap B_{\rho^{\ell-1}R}(x)\subset B_{\f{\rho^{\ell}R}{10}}(L'(x,\rho^{\ell-1}R))\cap B_{\rho^{\ell-1}R}(x).
$$
Given $ A_{\ell-1} $ in \eqref{defAi}, analogous to \eqref{SvadeltarBry}, we deduce from \eqref{choiceofrRell} that
$$
(S_{\va,\delta r}^k(u)\cap B_R)\backslash\(\bigcup_{y\in\cB_{\ell-1}}B_{r_y}(y)\)\subset B_{\f{2\rho^{\ell}R}{5}}(A_{\ell-1})\subset B_{\f{2r}{5}}(A_{\ell-1}).
$$
Choose $ \cD_{A_{\ell-1}} $ as the maximal subset of $ A_{\ell-1} $ such that for any $ y,z\in\cD_{A_{\ell-1}} $, there holds $ \dist(y,z)\geq\f{r}{5} $. As a result, the balls in the collection
$$
\{B_{\f{r}{10}}(y)\}_{y\in\cD_{A_{\ell-1}}}\cup\{B_{\f{r_y}{10}}(y)\}_{y\in\cB_{\ell-1}} 
$$
are pairwise disjoint, and
$$
(S_{\va,\delta r}^k(u)\cap B_R)\backslash\(\bigcup_{y\in\cB_{\ell-1}}B_{r_y}(y)\)\subset\bigcup_{y\in\cD_{A_{\ell-1}}} B_{\f{2r}{5}}(y).
$$
Additionally, similar to \eqref{DAiSva}, we assume that 
$$
S_{\va,\delta r}^k(u)\cap B_R\cap B_{\f{r}{2}}(y)\neq\emptyset,\text{ for any }y\in\cD_{A_{\ell-1}}.
$$
Define $ \cB_{\ell}=\cB_{\ell-1} $ and $ \cG_{\ell}=\cD_{A_{\ell-1}} $ such that for any $ y\in\cG_{\ell} $, $ r_y=r $. By the construction above, we have \ref{p1}-\ref{p3} hold. Analogous to \eqref{simiuzE}, we have
$$
\sup_{z\in B_{2\rho^{\ell-1}R}(x)}\vt_{\rho^{\ell-1}R}(u,z)\leq\sup_{z\in B_{2R}}\vt_R(u,z)\leq E.
$$
Applying Lemma \ref{lempinch} with $ \beta=\f{\rho}{20} $ and $ s=\rho^{\ell-1}R $, we deduce that if $ \delta=\delta(\va,\xi,\Lda,n,p,\rho)>0 $ is sufficiently small, there holds
$$
\vt_{\f{r}{20}}(u,y)\geq\vt_{\f{\rho^{\ell}R}{20}}(u,y)>E-\xi,
$$
for any $ y\in\cG_{\ell} $. Moreover, we can prove \eqref{Vgeq} for $ y\in\cD_{\ell} $ by almost the same methods used in the proof of \eqref{geqga}.

\subsubsection*{Step 4. Proof of \eqref{estiCov}} 

In this step, we denote $ \cD_{\ell} $ by $ \cD $. For convenience, we define
$$
\mu_{\cD}:=\sum_{x\in\cD}\w_kr_x^k\delta_x,
$$
and
$$
\wt{\cD}_t:=\cD\cap\{r_x\leq t\},\quad\mu_t:=\mu_{\cD}\llcorner\wt{\cD}_t\ll\mu_{\cD}.
$$
Here $ \mu_t=\mu_{\cD}\llcorner\wt{\cD}_t\ll\mu_{\cD} $ means that for any $ A\subset\R^n $, if $ \mu_{\cD}(A)=0 $, then $ (\mu\llcorner\wt{\cD}_t)(A)=0 $. Let $ N\in\Z_+ $ be such that $ 2^{N-1}r<\f{R}{70}\leq 2^Nr $. We will inductively show that there exists $ C_{\op{II}}'=C_{\op{II}}'(n)>0 $ such that for $ j\in\{0,1,2,...,N\} $, there holds
\be
\mu_{2^jr}(B_{2^jr}(x))=\sum_{y\in\wt{\cD}_{2^jr}\cap B_{2^jr}(x)}r_y^k\leq C_{\op{II}}'(n)(2^jr)^k,\label{jjplus1}
\ee
for any $ x\in B_R $. It directly implies \eqref{estiCov} by covering $ B_R $ with smaller balls. As a result, the constant $ C_{\op{II}}>0 $ is chosen as $
C_{\op{II}}(n)=C(n)C_{\op{II}}'(n)>0 $ for some $ C(n)>0 $.

We first note that by \ref{p3}, the collection of balls $ \{B_{\f{r_y}{10}}(y)\}_{y\in\cD} $ are pairwise disjoint. As a result, we can choose appropriate $ C_{\op{II}}'(n)>0 $ such that 
$$
\mu_r(B_r(x))\leq C_{\op{II}}'(n)r^k,\text{ for any }x\in B_R,
$$
which implies \eqref{jjplus1} for $ j=0 $.

Now we assume that \eqref{jjplus1} is true for any $ i\in\{1,2,...,j\} $ with $ j\in\{1,2,...,N-1\} $ and we intend to show the property for $ j+1 $. We start with a rough bound. Precisely, for any $ x\in B_R $,
\be
\mu_{2^{j+1}r}(B_{2^{j+1}r}(x))\leq C(n)C_{\op{II}}'(n)(2^{j+1}r)^k.\label{mu2j}
\ee
Covering $ B_{2^{j+1}r}(x) $ with balls $ B_{2^jr}(y) $ such that $ y\in\cD $, we obtain 
\be
\mu_{2^jr}(B_{2^{j+1}r}(x))\leq C(n)C_{\op{II}}(n)(2^jr)^k.\label{mu2j2j1plus}
\ee
Moreover, we have
$$
\mu_{2^{j+1}r}=\mu_{2^jr}+\sum_{x\in\cD,r_x\in(2^jr,2^{j+1}r]}\w_kr_x^k\delta_x.
$$
It follows from the disjointedness of $ \{B_{\f{r_y}{10}}(y)\}_{y\in\cD} $ that
$$
\(\sum_{y\in\cD,r_y\in(2^jr,2^{j+1}r]}\w_kr_y^k\delta_y\)(B_{2^{j+1}r}(x))\leq C(n)(2^{j+1}r)^k.
$$
Combining this and \eqref{mu2j2j1plus}, assuming that $ C_{\op{II}}'(n)>1 $, we directly obtain the rough bound estimate \eqref{mu2j}. Next, we will use this rough bound to give the refined one by using the Reifenberg-type theorem. Fix $ x\in B_R $, we define
$$
\mu:=\mu_{2^{j+1}r}\llcorner B_{2^{j+1}r}(x).
$$
We will show that
\be
\mu(B_{2^{j+1}r}(x))\leq C_{\op{II}}'(n)(2^{j+1}r)^k.\label{wanttoshow}
\ee
Define the truncation of \eqref{Wrfunc} by
$$
\wt{W}_s(u,y):=\left\{\begin{aligned}
&\vt_{2s}(u,y)-\vt_s(u,y)&\text{ for }&\f{r_y}{10}\leq s<R,\\
&0&\text{ for }&0<s<\f{r_y}{10},
\end{aligned}\right.
$$
for $ y\in\cD $. We claim that
\be
D_{\mu}^k(y,s)\leq C(\va,\Lda,n,p,\rho)s^{-k}\int_{B_{\f{5s}{2}}(y)}\wt{W}_{\f{5s}{2}}(u,z)\ud\mu(z),\label{beta2leq}
\ee
for any $ y\in\cD $ and $ 0<s<\f{R}{10} $. If $ 0<s<\f{r_y}{10} $, since the balls in $ \{B_{\f{r_z}{10}}(z)\}_{z\in\cD} $ are pairwise disjoint, we can deduce that the left hand side of \eqref{beta2leq} is $ 0 $ and the result is trivially true. For otherwise, given the property \ref{p5} in the construction of this covering, we apply Theorem \ref{beta2} to the ball $ B_{10s}(y) $, and \eqref{beta2leq} follows. For any $ s\in(0,2^{j+1}r) $ and $ z\in B_R $, we claim that 
\be
\mu_s(B_s(z))\leq C(n)C_{\op{II}}'(n)s^k.\label{musBs}
\ee
To show this claim, by the base of induction, namely \eqref{mu2j} with $ j=0 $, we can assume that $ s\geq r $. For fixed $ s\in(r,2^{j+1}r) $, we choose $ N'\in\{1,2,...,j\} $ such that $ 2^{N'}r\leq s<2^{N'+1}r $. Now we apply the assumption of induction and the rough bound \eqref{mu2j} to obtain 
$$
\mu_s(B_s(z))\leq \mu_{2^{N'+1}r}(B_{2^{N'+1}r}(z))\leq C(n)C_{\op{II}}'(n)s^k,
$$
which implies \eqref{musBs}. We also claim that for any $ r\leq s\leq\f{2^{j+1}r}{10} $ and $ z\in B_R $,
\be
\mu_{2^{j+1}r}(B_{\f{5s}{2}}(z))\leq C(n)C_{\op{II}}'(n)s^k.\label{murj1plus}
\ee
Up to a further covering of the ball $ B_{\f{5s}{2}}(z) $ and a translation if necessary, we only need to show that 
\be
\mu_{2^{j+1}r}(B_s(z))\leq C(n)C_{\op{II}}'(n)s^k,\label{onlyneed}
\ee
for any $ r\leq s\leq\f{2^{j+1}r}{10} $ and $ z\in\cD $. If $ y\in B_s(z)\cap\supp(\mu) $, we can deduce from the property that balls in the collection $ \{B_{\f{r_{\zeta}}{10}}(\zeta)\}_{\zeta\in\cD} $ are pairwise disjoint that $
\f{r_y}{10}\leq|y-z|\leq s $ and $ y\in\wt{\cD}_{10s} $. This implies that $ B_s(z)\cap\supp(\mu)\subset\wt{\cD}_{10s} $, and then
\be
\mu_{2^{j+1}r}(B_s(z))\leq\mu_{10s}(B_{s}(z))\leq\mu_{10s}(B_{10s}(z))\leq C(n)C_{\op{II}}'(n)s^k,\label{mu2jplus1}
\ee
which implies \eqref{onlyneed}. Here in \eqref{mu2j2j1plus}, for the last inequality, we have used \eqref{musBs}. Letting $ s<t<\f{2^{j+1}r}{10} $ and $ y\in B_{2^{j+1}r}(x) $, we have $ B_{\f{5s}{2}}(z)\subset B_{\f{7t}{2}}(y) $ for any $ z\in B_t(y) $. Integrating \eqref{beta2leq} for both sides on $ B_t(y) $, we obtain 
\begin{align*}
&\int_{B_t(y)}D_{\mu}^k(z,s)\ud\mu(z)\leq \f{C}{s^k}\int_{B_t(y)}\(\int_{B_{\f{5s}{2}}(z)}\wt{W}_{\f{5s}{2}}(u,\zeta)\ud\mu_{2^{j+1}r}(\zeta)\)\ud\mu_{2^{j+1}}(z)\\
&\quad\quad\leq\f{C}{s^k}\int_{B_t(y)}\int_{B_{\f{7t}{2}}(y)}\chi_{B_{\f{5s}{2}}(z)}(\zeta)\wt{W}_{\f{5s}{2}}(u,\zeta)\ud\mu_{2^{j+1}r}(\zeta)\ud\mu_{2^{j+1}}(z)\\
&\quad\quad\leq\f{C}{s^k}\int_{B_{\f{7t}{2}}(y)}\(\int_{B_t(y)}\chi_{B_{\f{5s}{2}}(\zeta)}(z)\ud\mu_{2^{j+1}r}(z)\)\wt{W}_{\f{5s}{2}}(u,\zeta)\ud\mu_{2^{j+1}r}(\zeta)\\
&\quad\quad\leq\f{C}{s^k}\int_{B_{\f{7t}{2}}(y)}\mu_{2^{j+1}r}(B_{\f{5s}{2}}(\zeta))\wt{W}_{\f{5s}{2}}(u,\zeta)\ud\mu_{2^{j+1}r}(\zeta)\\
&\quad\quad\leq C(\va,\Lda,n,p,\rho)C_{\op{II}}'(n)\int_{B_{\f{7t}{2}}(y)}\wt{W}_{\f{5s}{2}}(u,\zeta)\ud\mu_{2^{j+1}r}(\zeta),
\end{align*}
where for the last inequality, we have used \eqref{murj1plus}. Moreover, we deduce that
\be
\begin{aligned}
&\int_{B_t(y)}\(\int_0^tD_{\mu}^k(z,s)\f{\ud s}{s}\)\ud\mu(z)\\
&\quad\quad\leq C(\va,\Lda,n,p,\rho)C_{\op{II}}'(n)\int_{B_{\f{7t}{2}}(y)}\(\int_0^t\wt{W}_{\f{5s}{2}}(u,z)\f{\ud s}{s}\)\ud\mu_{2^{j+1}r}(z).
\end{aligned}\label{beta2k1}
\ee
On the other hand, using Proposition \ref{MonFor} and \ref{p4}, we obtain
\begin{align*}
&\int_0^t\wt{W}_{\f{5s}{2}}(u,z)\f{\ud s}{s}=\int_{\f{r_z}{10}}^t\wt{W}_{\f{5s}{2}}(u,z)\f{\ud s}{s}\\
&\quad\quad=\int_{\f{r_z}{10}}^t\(\vt_{5s}(u,z)-\vt_{\f{5s}{2}}(u,z)\)\f{\ud s}{s}\\
&\quad\quad=\int_{\f{t}{2}}^t\vt_{5s}(u,z)\f{\ud s}{s}+\int_{\f{r_z}{10}}^{\f{t}{2}}\vt_{5s}(u,z)\f{\ud s}{s}-\int_{\f{r_z}{5}}^t\vt_{\f{5s}{2}}(u,z)\f{\ud s}{s}-\int_{\f{r_z}{10}}^{\f{r_z}{5}}\vt_{\f{5s}{2}}(u,z)\f{\ud s}{s}\\
&\quad\quad=\int_{\f{t}{2}}^t\(\vt_{5s}(u,z)-\vt_{\f{r_zs}{2t}}(u,z)\)\f{\ud s}{s}\\
&\quad\quad\leq C\(\vt_{5t}(u,z)-\vt_{\f{r_z}{4}}(u,z)\)\\
&\quad\quad\leq C(\va,\Lda,n,p,\rho)\xi,
\end{align*}
for any $ z\in\cD $ and $ 0<t<\f{2^{j+1}r}{10} $. This, together with \eqref{mu2j} and \eqref{beta2k1}, implies that
$$
\int_{B_t(y)}\(\int_0^tD_{\mu}^k(z,s)\f{\ud s}{s}\)\ud\mu(z)\leq C'(\va,\Lda,n,p,\rho)C_{\op{II}}'(n)\xi t^k,
$$
for any $ y\in B_{2^{j+1}r}(x) $ and $ 0<t<\f{2^{j+1}r}{10} $. If $ \xi=\xi(\va,\Lda,n,p,\rho)>0 $ sufficiently small such that $ C'(\va,\Lda,n,p,\rho)C_{\op{II}}'(n)\xi<\delta_{Rei,1} $, then we can apply Theorem \ref{Rei1} to obtain
$$
\mu(B_{2^{j+1}r}(x))\leq C_{Rei,1}(n)(2^{j+1}r)^k.
$$
Choosing $ C_{\op{II}}'(n)>C_{Rei,1}(n) $, we deduce \eqref{wanttoshow}, which completes the proof.

\subsubsection*{Some remarks on the proof} In the proof of Lemma \ref{cover1}, we frequently apply Proposition \ref{Fprop} and Lemma \ref{lempinch} by choosing $ \delta>0 $ sufficiently large. One does not need to worry that the choice of such a $ \delta $ is not uniform since those applied in the proof are all about some fixed scales.

\subsection{Proof of Lemma \ref{Cover2}} Up to a translation, we can assume that $ x_0=0 $. Letting $ 0<\rho<\f{1}{100} $ be determined later, we can choose $ \ell\in\Z_+ $ such that 
\be
\(\f{\rho}{20}\)^{\ell}R<r\leq\(\f{\rho}{20}\)^{\ell-1}R.\label{lchose}
\ee
We will construct inductively that for any $ i\in\{1,2,...,\ell\} $, there exists a covering of $ S_{\va,\delta r}^k(u)\cap B_R $, denoted by $ \{B_{r_x}(x)\}_{x\in\cR_i\cup\cF_i\cup\cB_i}$ such that
$$
S_{\va,\delta r}^k(u)\cap B_R\subset\bigcup_{x\in\cR_i}B_r(x)\cup\bigcup_{x\in\cF_i}B_{r_x}(x)\cup\bigcup_{x\in\cB_i}B_{r_x}(x),
$$
and the following properties hold.
\begin{enumerate}[label=$(\theenumi)$]
\item For any $ x\in\cR_i\cup\cF_i\cup\cB_i $, $ r_x\geq r $ and $ B_{2r_x}(x)\subset B_{2R} $.
\item If $ x\in\cR_i $, then $ r_x=r $.
\item If $ x\in\cF_i $, then
$$
\sup_{y\in B_{2r_x}(x)}\vt_{r_y}(u,y)\leq E-\delta,
$$
where $ E=E(0,R) $.
\item If $ x\in\cB_i $, neither of the above properties is true, and 
\be
r<r_x\leq\(\f{\rho}{20}\)^iR.\label{rBx}
\ee
\item We have estimates
\be
\sum_{x\in\cR_i\cup\cF_i}r_x^k\leq C_{\op{I}}(n)\(\sum_{j=0}^i\f{1}{10^j}\)R^k,\quad\sum_{x\in\cB_i}r_x^k\leq\f{R^k}{10^i},\label{C3es}
\ee
for any $ i\in\{1,2,...,\ell\} $.
\end{enumerate}

For $ i=\ell $, by \eqref{lchose} and \eqref{rBx}, we have $ \cB_{\ell}=\emptyset $, and then  any ball in this covering will either satisfy \eqref{energydrop} or have radius $ r $. Estimates \eqref{C3es} for $ i=\ell $ give the desired bound on the final covering.

\subsubsection*{Step 1. Preliminaries of the construction} To begin with, we consider a fixed ball $ B_{2s}(x)\subset B_{2R} $. We will obtain a covering of $ S_{\va,\delta r}^k(u)\cap B_s(x) $, denoted by $ \{B_{r_y}(y)\}_{y\in\cR_x\cup\cF_x\cup\cB_x} $ such that
$$
S_{\va,\delta r}^k(u)\cap B_s(x)\subset\bigcup_{y\in\cR_x}B_{r_y}(y)\cup\bigcup_{y\in\cF_x}B_{r_y}(y)\cup\bigcup_{y\in\cB_x}B_{r_y}(y),
$$
and the following properties hold.
\begin{itemize}
\item For any $ y\in\cR_x\cup\cF_x\cup\cB_x $, $ r_y\geq r $ and $ B_{2r_y}(y)\subset B_{2s}(x) $.
\item If $ y\in\cR_x $, then $ r_y=r $.
\item If $ y\in\cF_x $, then
\be
\sup_{z\in B_{2r_y}(y)}\vt_{r_z}(u,z)\leq E-\delta.\label{energydrop}
\ee
\item If $ y\in\cB_x $, then $
r<r_y\leq\f{\rho s}{20} $.
\item We have the estimates
$$
\sum_{y\in\cR_x\cup\cF_x}r_y^k\leq C_{\op{I}}(n)s^k,\quad\sum_{y\in\cB_x}r_y^k\leq\f{s^k}{10}.
$$
\end{itemize}

To obtain such covering, we first apply Lemma \ref{cover1} to the ball $ B_s(x) $ and choose a sufficiently small $ \delta=\delta(\va,\Lda,n,p,\rho)>0 $ to obtain a covering 
$$
S_{\va,\delta r}^k(u)\cap B_s(x)\subset\bigcup_{y\in\cD}B_{r_y}(y)
$$
with $ r_y\geq r $ for any $ y\in\cD $, and
$$
\sum_{x\in\cD}r_y^k\leq C_{\op{II}}(n)s^k.
$$
For any $ y\in\cD $, $ B_{2r_y}(y)\subset B_{2s}(x) $, and either $ r_y=r $ or there exists $ L(y,r_y)\in\bA(n,k-1) $ such that
\be
\{z\in B_{2r_y}(y):\vt_{\f{\rho r_y}{20}}(u,z)>E(x,s)-\delta\}\subset B_{\f{\rho r_y}{10}}(L(y,r_y))\cap B_{2r_y}(y),\label{BsFcontain}
\ee
where $ E(x,s)=\sup_{z\in B_{2s}(x)}\vt_s(u,z) $.

Now we divide the centers of balls $ \cD $ into two collections $ \cD^{(r)} $ and $ \cD^{(+)} $. Indeed, we have the following properties.
\begin{itemize}
\item If $ y\in\cD^{(r)} $, then $ \f{\rho r_y}{20}\leq r $.
\item If $ y\in\cD^{(+)} $, then $ \f{\rho r_y}{20}>r $.
\end{itemize}
Next, we will refine the balls with centers in $ \cD^{(r)} $ and $ \cD^{(+)} $ through recovering. Precisely speaking, we construct the recovering as follows.
\begin{itemize}
\item For $ y\in\cD^{(r)} $, we consider a minimal covering of $ B_{r_y}(y) $ with balls of radius $ r $. Indeed, we let $ \{B_r(z)\}_{z\in\cR_{x}^{(y)}} $ satisfy 
$$
B_{r_y}(y)\subset\bigcup_{z\in\cR_x^{(y)}}B_r(z),
$$
and have the minimum number $ \#\cR_x^{(y)} $ among all these coverings with radius $ r $. We deduce that 
\be
\#\cR_x^{(y)}\leq C(n)\rho^{-n}.\label{RxCnbx}
\ee
Moreover, we require that for any $ z\in\cR_x^{(y)} $, there holds $ B_{2r}(z)\subset B_{2s}(x) $. We define the collection of all these centers of balls in the recovering by
$$
\cR_x=\bigcup_{y\in\cD^{(r)}}\cR_x^{(y)}.
$$
\item For $ y\in\cD^{(+)} $, since $ r_y>\f{20r}{\rho}>r $, the property \eqref{BsFcontain} must hold. Consider a covering of $ B_{r_y}(y) $ with balls of radius $ \f{\rho r_y}{20}>r $ centered inside this ball such that the balls with half of the radius in this collection are pairwise disjoint. We have
$$
B_{r_y}(y)\subset\bigcup_{z\in\cB_x^{(y)}}B_{\f{\rho r_y}{20}}(z)\cup \bigcup_{z\in\cF_x^{(y)}}B_{\f{\rho r_y}{20}}(z),
$$
where
\be
\{z\in B_{2r_y}(y):\vt_{\f{\rho r_y}{20}}(u,z)>E(x,s)-\delta\}\cap\bigcup_{z\in\cF_x^{(y)}}B_{\f{\rho r_y}{10}}(z)=\emptyset.\label{Fzdef}
\ee
It follows from \eqref{BsFcontain} that 
\be
\cB_{x}^{(y)}\subset B_{\rho r_y}(L(y,r_y))\cap B_{r_y}(y).\label{cBxyBrho}
\ee
For $ z\in\cF_x^{(y)} $, $ r_z=\f{\rho r_y}{20} $. It follows from \eqref{Fzdef} and $ B_{\f{\rho r_y}{10}}(z)\subset B_{2r_y}(y) $ that
$$
\sup_{\zeta\in B_{2r_z}(z)}\vt_{r_z}(u,\zeta)=\sup_{\zeta\in B_{\f{\rho r_y}{10}}(z)}\vt_{\f{\rho r_y}{20}}(u,\zeta)\leq E(x,s)-\delta\leq E-\delta.
$$
Since $ \{B_{\f{\rho r_y}{40}}(z)\}_{z\in\cB_{x}^{(y)}} $ are pairwise disjoint, we can use \eqref{cBxyBrho} and $ L(y,r_y)\in\bA(n,k-1) $ to obtain
\be
\#\cF_{x}^{(y)}\leq C(n)\rho^{-n},\quad\#\cB_{x}^{(y)}\leq C(n)\rho^{1-k}.\label{BxFxcount}
\ee
\end{itemize}

Define
$$
\cB_x:=\bigcup_{y\in\cD^{(+)}}\cB_x^{(y)},\quad\cF_x:=\bigcup_{y\in\cD^{(+)}}\cF_x^{(y)}.
$$
It follows from \eqref{RxCnbx} and \eqref{BxFxcount} that
\begin{align*}
\sum_{y\in\cR_x\cup\cF_x}r_y^k&\leq C\rho^{-n+k}\(\sum_{y\in\cD}r_y^k\)\leq C(n)\rho^{-n+k}C_{\op{II}}(n)s^k,\\
\sum_{y\in\cB_x}r_y^k&\leq C(n)\rho\(\sum_{y\in\cD}r_y^k\)\leq C(n)\rho C_{\op{II}}(n)s^k.
\end{align*}
Choosing $ 0<\rho<\f{1}{100} $ sufficiently small, we can obtain some $ C_{\op{I}}(n)>0 $ such that
\be
\sum_{y\in\cR_x\cup\cF_x}r_y^k\leq C_{\op{I}}(n)s^k,\quad\sum_{y\in\cB_x}r_y^k\leq\f{s^k}{10}.\label{induces}
\ee

\subsubsection*{Step 2. Inductive constructions}

Given the preliminary result in the previous step, we can now conduct inductive constructions for our covering.

For $ i=1 $, we can apply the results in Step 1 to the ball $ B_R $, and the properties for such a case follow directly. 

We now assume that for $ i\in\{1,2,...,\ell-1\} $, there is a covering of $ S_{\va,\delta r}^k(u)\cap B_R $, given by
$$
S_{\va,\delta r}^k(u)\cap B_R\subset\bigcup_{x\in\cR_i}B_r(x)\cup\bigcup_{x\in\cF_i}B_{r_x}(x)\cup\bigcup_{x\in\cB_i}B_{r_x}(x),
$$
which satisfies the properties mentioned above. Apply the results in Step 1 to the balls in $ \{B_{r_x}(x)\}_{x\in\cB_i} $ and define
$$
\cR_{i+1}=\cR_i\cup\bigcup_{x\in\cB_i}\cR_x,\quad\cF_{i+1}=\cF_i\cup\bigcup_{x\in\cB_i}\cF_x,\quad\cB_{i+1}=\bigcup_{x\in\cB_i}\cB_x.
$$
As a result, this covering satisfies all the properties required at first except \eqref{C3es}. To complete the proof, by using \eqref{induces}, we obtain 
\begin{align*}
\sum_{x\in\cR_{i+1}\cup\cF_{i+1}}r_x^k&\leq\sum_{x\in\cR_i\cup\cF_i}r_x^k+\sum_{x\in\cB_i}\sum_{y\in\cR_x\cup\cF_x}r_y^k\\
&\leq C_{\op{I}}(n)\(\sum_{j=0}^i\f{1}{10^j}\)R^k+C_{\op{I}}(n)\sum_{x\in\cB_i}r_x^k\\
&\leq C_{\op{I}}(n)\(\sum_{j=0}^{i+1}\f{1}{10^j}\)R^k,
\end{align*}
and
$$
\sum_{x\in\cB_{i+1}}r_x^k\leq \sum_{x\in\cB_i}\sum_{y\in\cB_x}r_y^k\leq\f{R^k}{10^{i+1}}.
$$

\section{Proof of main theorems}\label{Mainproof}

\subsection{Proof of Theorem \ref{volthm}} Applying Lemma \ref{maincover} to $ 0<r'<R\leq 1 $ and $ x_0\in B_1 $, we can choose $ \delta=\delta(\va,\Lda,n,p)\in(0,1) $ and have a covering of $ S_{\va,\delta r'}^k(u)\cap B_R(x_0) $, denoted by $ \{B_{r'}(x)\}_{x\in\cC} $ such that
$$
S_{\va,\delta r'}^k(u)\cap B_R(x_0)\subset\bigcup_{x\in\cC}B_{r'}(x),
$$
and $ (\#\cC)r^k\leq C(\va,\Lda,n,p)R^k $. As a result, we have
\begin{align*}
\cL^n(B_{r'}(S_{\va,\delta r'}^k(u)\cap B_R(x_0)))&\leq \cL^n\(B_{r'}\(\bigcup_{x\in\cC}B_{r'}(x)\)\)\leq C(\va,\Lda,n,p)(r')^{n-k}R^k.
\end{align*}
Letting $ r=\delta r' $, we see that for any $ 0<r<\delta R $, there holds
\be
\begin{aligned}
\cL^n(B_{r}(S_{\va,r}^k(u)\cap B_R(x_0)))&\leq \cL^n(B_{r'}(S_{\va,\delta r'}^k(u)\cap B_R(x_0)))\\
&\leq C(\va,\Lda,n,p)r^{n-k}R^k.
\end{aligned}\label{Lnleq}
\ee
On the other hand, if $ \delta R\leq r\leq R $, then
$$
\cL^n(B_{r}(S_{\va,r}^k(u)\cap B_R(x_0)))\leq\cL^n(B_{2R}(x_0))\leq C(\va,\Lda,n,p)r^{n-k}R^k.
$$
This, together with \eqref{Lnleq}, implies that  
$$
\cL^n(B_{r}(S_{\va,r}^k(u)\cap B_R(x_0)))\leq C(\va,\Lda,n,p)r^{n-k}R^k,
$$
and then
\be
\cL^n(B_{r}(S_{\va}^k(u)\cap B_R(x_0)))\leq C(\va,\Lda,n,p)r^{n-k}R^k,\label{Svaes2}
\ee
for any $ 0<r\leq R $, and $ x_0\in B_1 $. Taking $ R=1 $ and $ x_0=0 $, the estimates \eqref{main1es} and \eqref{main2es} follow directly. For any $ 0<r<\f{1}{2} $ and a covering of $ S_{\va}^k(u)\cap B_R(x_0) $, denoted by $ \{B_{r_x}(x)\}_{x\in\cC} $ with $ r_x\leq r $ such that $ S_{\va}^k(u)\cap B_R(x_0)\cap B_{r_x}(x)\neq\emptyset $ for any $ x\in\cC $, we can deduce from \eqref{Svaes2} that
\begin{align*}
\sum_{x\in\cC}r_x^k&\leq C(n)r^{k-n}\cL^n(B_{2r}(S_{\va}^k(u)\cap B_R(x_0)))\leq C(\va,\Lda,n,p)R^k,
\end{align*}
which implies that 
\be
\HH_r^k(S_{\va}^k(u)\cap B_R(x_0))\leq C(\va,\Lda,n,p)R^k.\label{Ahlforses}
\ee
Letting $ r\to 0^+ $, we can complete the proof of Ahlfors $ k $-regularity of $ S_{\va}^k(u) $.

\subsection{Proof of Theorem \ref{rectthm}} Without loss of generality, we prove that $ S^k(u)\cap B_1 $ and $ S_{\va}^k(u)\cap B_1 $ are rectifiable for any $ k\in\{1,2,...,n-\lceil\al_p\rceil\} $ and $ 0<\va<1 $. It follows from Lemma \ref{SkSkva} that
\be
S^k(u)\cap B_{10}=\bigcup_{i\in\Z_+}S_{i^{-1}}^k(u).\label{Skiminus1}
\ee
By this, we only need to show the rectifiability of $ S_{\va}^k(u)\cap B_1 $, for any $ \va>0 $. Let $ S\subset S_{\va}^k(u)\cap B_1 $ be such that $ \HH^k(S)>0 $. For any $ x\in B_1 $ and $ 0<r\leq 1 $, we define
$$
g_r(u,x):=\vt_r(u,x)-\vt(u,x).
$$
By \eqref{thevt} and Proposition \ref{MonFor}, we have that $ \lim_{r\to 0^+}g_r(u,x)=0 $ for any $ x\in B_1 $ and $ g_r(u,\cdot) $ is bounded. As a result, we can apply the dominated convergence theorem to obtain that for any $ \delta>0 $, there exists $ r_0>0 $ such that 
$$
\f{1}{\HH^k(S)}\int_Sg_{10r_0}(u,x)\ud\HH^k(x)\leq\delta.
$$
By average arguments, we can choose an $ \HH^k $-measurable set $ E\subset S $ such that $ \HH^k(E)\leq\delta\HH^k(S) $ and $ g_{10r_0}(u,x)\leq\delta $ for any $ x\in F:=S\backslash E $. We can cover $ F $ by a finite number of balls $ \{B_{r_0}(x_i)\}_{i=1}^N $ such that $ \{x_i\}_{i=1}^N\subset F $. Now we claim that if $ \delta=\delta(\va,\Lda,n,p)>0 $ is sufficiently small, then for any $ i\in\{1,2,...,N\} $, $ F\cap B_{r_0}(x_i) $ is $ k $-rectifiable. If such a claim is true, then $ F $ is rectifiable. If $ \HH^k(F)>0 $, we can repeat this procedure to $ F $ for countably times and finally obtain that $ S $ is $ k $-rectifiable. By the arbitrariness for the choice of $ S $, $ S_{\va}^k(u)\cap B_1 $ is $ k $-rectifiable. Let us show this claim. Without loss of generality, we only consider the ball $ B_{r_0}(x_1) $ and assume that $ 0<r_0<\f{1}{100} $. By the assumption on $ E $, we have 
\be
g_{10r_0}(u,z)=\vt_{10r_0}(u,z)-\vt(u,z)\leq\delta,\label{G10r0small}
\ee 
for any $ z\in F $. Choosing $ \delta=\delta(\delta',\Lda,n,p)>0 $ sufficiently small, we can apply Lemma \ref{SmaHom} to obtain that $ u $ is $ (0,\delta') $-symmetric in $ B_{4s}(z) $ for any $ 0<s\leq r_0 $, where $ \delta'>0 $ is to be determined later. For $ z\in F\subset S_{\va}^k(u) $, $ u $ is not $ (k+1,\va) $-symmetric in $ B_{4s}(z) $. Choosing $ \delta'=\delta'(\va,\Lda,n,p)>0 $ sufficient small and using Corollary \ref{beta22}, we deduce that
$$
D_{\HH^k\llcorner F}^k(z,s)\leq C(\va,\Lda,n,p)s^{-k}\int_{B_s(z)}W_s(u,\zeta)\ud(\HH^k\llcorner F)(\zeta),
$$
for any $ z\in F $ and $ 0<s\leq r_0 $. Integrating with respect to $ z $ for both sides of the above on $ B_r(x) $ with $ x\in B_{r_0}(x_1) $ and $ 0<r\leq r_0 $, we have
\begin{align*}
&\int_{B_r(x)}D_{\HH^k\llcorner F}^k(z,s)\ud(\HH^k\llcorner F)(z)\\
&\quad\quad\leq Cs^{-k}\int_{B_r(x)}\(\int_{B_s(z)}W_s(u,\zeta)\ud(\HH^k\llcorner F)(\zeta)\)\ud(\HH^k\llcorner F)(z)\\
&\quad\quad\leq Cs^{-k}\int_{B_r(x)}\(\int_{B_{r+s}(x)}\chi_{B_s(z)}(\zeta)W_s(u,\zeta)\ud(\HH^k\llcorner F)(\zeta)\)\ud(\HH^k\llcorner F)(z)\\
&\quad\quad\leq Cs^{-k}\int_{B_{r+s}(x)}\HH^k(F\cap B_s(\zeta))W_s(u,\zeta)\ud(\HH^k\llcorner F)(\zeta)\\
&\quad\quad\leq C(\va,\Lda,n,p)\int_{B_{r+s}(x)}W_s(u,z)\ud(\HH^k\llcorner F)(z).
\end{align*}
For the last inequality above, we have used \eqref{Ahlforses}, the Ahlfors $ k $-regularity of $ S_{\va}^k(u)\cap B_1 $. It follows that
\begin{align*}
&\int_{B_r(x)}\(\int_0^rD_{\HH^k\llcorner F}^k(z,s)\f{\ud s}{s}\)\ud(\HH^k\llcorner F)(z)\\
&\quad\quad\leq C\int_{B_{2r}(x)}\(\int_0^r(\vt_{2s}(u,z)-\vt_s(u,z))\f{\ud s}{s}\)\ud(\HH^k\llcorner F)(z)\\
&\quad\quad=C\int_{B_{2r}(x)}\(\sum_{i=0}^{+\ift}\int_{\f{r}{2^{i+1}}}^{\f{r}{2^i}}(\vt_{2s}(u,z)-\vt_s(u,z))\f{\ud s}{s}\)\ud(\HH^k\llcorner F)(z)\\
&\quad\quad\leq C\int_{B_{2r}(x)}\sum_{i=0}^{+\ift}\(\vt_{\f{r}{2^{i-1}}}(u,z)-\vt_{\f{r}{2^{i+1}}}(u,z)\)\ud(\HH^k\llcorner F)(z)\\
&\quad\quad\leq C\int_{B_{2r}(x)}g_{10r_0}(u,z)\ud(\HH^k\llcorner F)(z)\\
&\quad\quad\leq C'(\va,\Lda,n,p)\delta r^k,
\end{align*}
for any $ x\in B_{r_0}(x_1) $ and $ 0<r\leq r_0 $. Here for the last inequality, we have used \eqref{G10r0small} and \eqref{Ahlforses}. Taking sufficiently small $ \delta=\delta(\va,\Lda,n,p)>0 $ such that $ C'(\va,\Lda,n,p)\delta<\delta_{Rei,2} $, we can apply Theorem \ref{Rei2} to obtain that $ F\cap B_{r_0}(x_1) $ is $ k $-rectifiable.

To show further results, we focus on $ S^k(u)\cap B_{10} $ and for $ S^k(u) $, results follow from some covering arguments. We claim that for any $ 0<\va<1 $ and $ \HH^k $-a.e. $ x\in S_{\va}^k(u) $, there exists $ V\in\bG(n,k) $ such that any tangent pair of $ u $ at $ x $ is $ k $-symmetric with respect to $ V $. We first assume that the claim has already been proved. As a result, we choose $ \wt{S}_{i^{-1}}^k(u)\subset S_{i^{-1}}^k(u) $ such that
$$
\HH^k(S_{i^{-1}}^k(u)\backslash \wt{S}_{i^{-1}}^k(u))=0,
$$
and for any $ x\in\wt{S}_{i^{-1}}^k(u) $, there exists $ V\in\bG(n,k) $ with any tangent pair of $ u $ at $ x $ being $ k $-symmetric with respect to $ V $. We define 
\be
\wt{S}^k(u):=\bigcup_{i\in\Z_+}\wt{S}_{i^{-1}}^k(u).\label{Swtunion}
\ee
Given the definition of $ \wt{S}_{i^{-1}}^k(u) $ and \eqref{Skiminus1}, we obtain 
$$
\HH^k((S^k(u)\cap B_{10})\backslash\wt{S}^k(u))=0.
$$
This, together with \eqref{Swtunion}, implies that for $ \HH^k $-a.e. $ x\in S^k(u)\cap B_{10} $, there exists some $ i\in\Z_+ $ such that $ x\in\wt{S}_{i^{-1}}^k(u) $. By the claim given at the beginning of this paragraph, the final results follow directly.

Let us fix $ 0<\va<1 $ and prove the claim. For $ x\in S_{\va}^k(u) $, recall the definition of $ \vt(u,x) $, given by \eqref{uxthetalim}. For any $ \xi,\va>0 $, and $ i\in\Z_+ $, we define
$$
W_{\va,\xi}^{k,i}(u):=\{x\in S_{\va}^k(u):\vt(u,x)\in[i\xi,(i+1)\xi)\}.
$$
It follows from the assumption $ \theta_{40}(u,0)\leq\Lda $ that
$$
S_{\va}^k(u)=\bigcup_{i=0}^{\lceil\xi^{-1}C''(\Lda,n,p)\rceil}W_{\va,\xi}^{k,i}(u).
$$
Here, $ \xi>0 $ is to be determined. Since $ S_{\va}^k(u) $ is $ k $-rectifiable, we see that for any $ \xi>0 $ and $ i\in\{0,1,2,...,\lceil\xi^{-1}C''(\Lda,n,p)\rceil\} $, $ W_{\va,\xi}^{k,i}(u) $ is also $ k $-rectifiable. By Theorem \ref{Simexis}, for fixed $ \xi,\va,i $, there exists a subset $ \wt{W}_{\va,\xi}^{k,i}(u)\subset W_{\va,\xi}^{k,i}(u) $ such that for any $ x\in\wt{W}_{\va,\xi}^{k,i}(u) $, the approximate tangent space of $ W_{\va,\xi}^{k,i}(u) $ exists, denoted by $ V_x\subset\R^n $. Fix $ x\in\wt{W}_{\va,\xi}^{k,i}(u) $, and let $ V_x $ be the approximate tangent space of $ W_{\va,\xi}^{k,i}(u) $ at $ x $. Choosing $ r>0 $ sufficiently small, we have 
$$
\vt_r(u,x)-\vt(u,x)<\xi.
$$
In view of Proposition \ref{MonFor} and the sixth property of Lemma \ref{5prop}, we infer that for $ r>0 $ sufficiently small and $ y\in W_{\va,\xi}^{k,i}(u)\cap B_r(x) $, there holds that 
\be
\vt_r(u,y)-\vt(u,y)<2\xi.\label{vtrxi2}
\ee
Since the approximate tangent space at $ x $ is $ V_x $, for any $ r>0 $ sufficiently small, we can obtain $ \f{r}{10} $-independent points $ \{x_i\}_{i=0}^k\subset W_{\va,\xi}^{k,i}(u)\cap B_r(x) $, where $ x_0=x $. In view of \eqref{vtrxi2}, we can apply Proposition \ref{QuaConSpl} to obtain that if $ \xi=\xi(\delta,\Lda,n,p)>0 $ is sufficiently small, then $ u $ is $ (k,\delta) $-symmetric in $ B_r(x) $, with respect to $ V_x $, where $ \delta=\delta(\xi)>0 $, and $ \delta\to 0^+ $ as $ \xi\to 0^+ $. Consequently, any tangent pair at $ x $ is $ (k,\delta(\xi)) $-symmetric with respect to $ V_x $. Define
$$
\wt{W}_{\va,\xi}^k(u):=\bigcup_{i\in\Z_+} \wt{W}_{\va,\xi}^{k,i}(u).
$$
As a result, we see that $
\wt{W}_{\va,\xi}^k(u)\subset S_{\va}^k(u) $ is a subset with 
\be
\HH^k(S_{\va}^k(u)\backslash\wt{W}_{\va,\xi}^k(u))=0.\label{Svakuback}
\ee 
We further define
$$
\wh{S}_{\va}^k(u):=\bigcap_{j\in\Z_+}\wt{W}_{\va,j^{-1}}^k(u).
$$
Using \eqref{Svakuback}, we have $
\HH^k(S_{\va}^k(u)\backslash\wh{S}_{\va}^k(u))=0 $. Moreover, it follows from the definition of $ \wt{W}_{\va,\xi}^k(u) $ that for any $ x\in\wh{S}_{\va}^k(u) $, any tangent pair must be $ (k,\delta) $-symmetric with respect to some $ V_x $ for any $ \delta>0 $. In particular, every tangent pair of $ u $ at $ x $ must be $ k $-symmetric with respect to some $ V_x $. 

\subsection{Proof of Theorem \ref{enhance}}

Using Corollary \ref{rucor}, we obtain 
$$
\sing(u)\cap B_1\subset\{x\in B_1:r_u^j(x)<r\}\subset S_{\delta,2r}^{k_{n,p}-1}(u)\cap B_1
$$
for some $ \delta=\delta(j,\Lda,n,p)>0 $. The result now follows directly from estimates in Theorem \ref{volthm}.

\section*{Acknowledgement}

The authors thank Professor Daniele Valtorta for his essential enlightening guide on the covering arguments in this paper. The authors are partially supported by the National Key R\&D Program of China under Grant 2023YFA1008801 and NSF of China under Grant 12288101.


\begin{thebibliography}{00000000}

\bibitem[All72]{All72}
\newblock W.K. Allard
\newblock On the first variation of a varifold,
\newblock \emph{Annals of mathematics}, \textbf{95} (1972), 417-491.

\bibitem[Alp18]{Alp18}
\newblock O. Alper,
\newblock Rectifiability of line defects in liquid crystals with variable degree of orientation
\newblock \emph{Archive for Rational Mechanics and Analysis}, \textbf{228} (2018), 309-339.

\bibitem[Alp20]{Alp20}
\newblock O. Alper,
\newblock On the singular set of free interface in an optimal partition problem,
\newblock \emph{Communications on Pure and Applied Mathematics}, \textbf{73} (2020), 855-915.

\bibitem[Aro57]{Aro57}
\newblock N. Aronszajn, 
\newblock A unique continuation theorem for elliptic differential equations or inequalities of the second order, 
\newblock \emph{Journal de Math\'{e}matiques Pures et Appliqu\'{e}es}, \textbf{36} (1957), 235-239.

\bibitem[Bet93]{Bet93}
\newblock F. Bethuel,
\newblock On the singular set of stationary harmonic maps,
\newblock \emph{Manuscripta Mathematica}, \textbf{78} (1993), 417-443.

\bibitem[Cab17]{Cab17}
\newblock X. Cabr\'{e}, 
\newblock Boundedness of stable solutions to semilinear elliptic equations: a survey,
\newblock \emph{Advanced Nonlinear Studies}, \textbf{17} (2017), 355-368.

\bibitem[CFRS20]{CFRS20}
\newblock X. Cabr\'{e}, A. Figalli, X. Ros-Oton, and J. Serra,
\newblock Stable solutions to semilinear elliptic equations are smooth up to dimension $ 9 $,
\newblock \emph{Acta Mathematica}, \textbf{224} (2020), 187-252.

\bibitem[CHN13]{CHN13}
\newblock J. Cheeger, R. Haslhofer, and Aaron Naber,
\newblock Quantitative stratification and the regularity of mean curvature flow,
\newblock \emph{Geometric and Functional Analysis}, \textbf{23} (2013), 828-847.

\bibitem[CHN15]{CHN15}
\newblock J. Cheeger, R. Haslhofer, and A. Naber,
\newblock Quantitative stratification and the regularity of harmonic map flow,
\newblock \emph{Calculus of Variations and Partial Differential Equations}, \textbf{53} (2015), 365-381.

\bibitem[CN13a]{CN13a}
\newblock J. Cheeger and A. Naber,
\newblock Lower bounds on Ricci curvature and quantitative behavior of singular sets,
\newblock \emph{Inventiones mathematicae}, \textbf{191} (2013), 321-339.

\bibitem[CN13b]{CN13b}
\newblock J. Cheeger and A. Naber,
\newblock Quantitative stratification and the regularity of harmonic maps and minimal currents,
\newblock \emph{Communications on Pure and Applied Mathematics}, \textbf{66} (2013), 965-990.

\bibitem[Del08]{Del08}
\newblock C. De Lellis,
\newblock Rectifiable sets, densities and tangent measures, 
\newblock European Mathematical Society, 2008.

\bibitem[DMSV18]{DMSV18}
\newblock C. De Lellis, A. Marchese, E. Spadaro, and D. Valtorta,
\newblock Rectifiability and upper Minkowski bounds for singularities of harmonic $ Q $-valued maps, \newblock \emph{Commentarii Mathematici Helvetici}, \textbf{93} (2018), 737-779.

\bibitem[EE19]{EE19}
\newblock N. Edelen, and M. Engelstein,
\newblock Quantitative stratification for some free-boundary problems,
\newblock \emph{Transactions of the American Mathematical Society}, \textbf{371} (2019), 2043-2072.

\bibitem[Dup11]{Dup11}
\newblock L. Dupaigne,
\newblock Stable solutions of elliptic partial differential equations
\newblock CRC press, 2011.

\bibitem[Eva90]{Eva90}
\newblock L. C. Evans,
\newblock Partial regularity for stationary harmonic maps into spheres,
\newblock \emph{Archive for Rational Mechanics and Analysis}, \textbf{116}, (1990) 203-218.

\bibitem[Eva18]{Eva18}
\newblock L. Evans,
\newblock Measure theory and fine properties of functions,
\newblock Routledge, 2018.

\bibitem[Far07]{Far07}
\newblock A. Farina,
\newblock On the classification of solutions of the Lane-Emden equation on unbounded domains of $ \R^n $,
\newblock \emph{Journal de math\'{e}matiques pures et appliqu\'{e}es}, \textbf{87} (2007), 537-561.

\bibitem[GM05]{GM05}
\newblock M. Giaquinta and L. Martinazzi,
\newblock An introduction to the regularity theory for elliptic systems, harmonic maps and minimal graphs,
\newblock Edizioni della Normale, Pisa, 2005.

\bibitem[Gra08]{Gra08}
\newblock L. Grafakos,
\newblock Classical fourier analysis
\newblock New York: Springer, 2008.

\bibitem[GJXZ24]{GJXZ24}
\newblock C. Guo, G. Jiang, C. Xiang, and G. Zheng,
\newblock Optimal higher regularity for biharmonic maps via quantitative stratification,
\newblock \emph{arXiv preprint arXiv:2401.11177}, (2024).

\bibitem[HLP92]{HLP92}
\newblock R. Hardt, F. Lin, and C. Poon,
\newblock Axially symmetric harmonic maps minimizing a relaxed energy,
\newblock \emph{Communications on pure and applied mathematics}, \textbf{45} (1992), 417-459.

\bibitem[HSV19]{HSV19}
\newblock J. Hirsch, S. Stuvard, and D. Valtorta,
\newblock Rectifiability of the singular set of multiple-valued energy minimizing harmonic maps,
\newblock \emph{Transactions of the American Mathematical Society}, \textbf{371} (2019), 4303-4352.

\bibitem[Lin99]{Lin99}
\newblock F. Lin,
\newblock Gradient estimates and blow-up analysis for stationary harmonic maps,
\newblock \emph{Annals of Mathematics. Second Series}, \textbf{149} (1999), 785-829.

\bibitem[LW99]{LW99}
\newblock F. Lin and C. Wang, 
\newblock Harmonic and quasi-harmonic spheres,
\newblock \emph{Communications in Analysis and Geometry}, \textbf{7} (1999), 397-429.

\bibitem[LW02a]{LW02a}
\newblock F. Lin and and C. Wang, \newblock Harmonic and quasi-harmonic spheres, Part II,
\newblock \emph{Communications in Analysis and Geometry}, \textbf{10} (2002), 341-375.

\bibitem[LW02b]{LW02b}
\newblock F. Lin and C. Wang,
\newblock Harmonic and quasi-harmonic spheres, Part III, rectifiability of the parabolic defect measure and generalized varifold flows,
\newblock \emph{Annales de l'Institut Henri Poincar\`{e} C}, \textbf{19} (2002), 209-259.

\bibitem[LW08]{LW08}
\newblock F. Lin and C. Wang,
\newblock The analysis of harmonic maps and their heat flows,
\newblock World Scientific Publishing, Hackensack, 2008.

\bibitem[Mat23]{Mat23}
\newblock P. Mattila,
\newblock Rectifiability: A Survey
\newblock Cambridge University Press, 2023. 

\bibitem[Mos03]{Mos03}
\newblock R. Moser,
\newblock Stationary measures and rectifiability,
\newblock \emph{Calculus of Variations and Partial Differential Equations}, \textbf{17} (2003), 357-368.

\bibitem[NV17]{NV17}
\newblock A. Naber and D. Valtorta,
\newblock Rectifiable-Reifenberg and the regularity of stationary and minimizing harmonic maps,
\newblock \emph{Annals of Mathematics}, \textbf{85} (2017), 131-227.

\bibitem[NV18]{NV18}
\newblock A. Naber and D. Valtorta,
\newblock Stratification for the singular set of approximate harmonic maps,
\newblock \emph{Mathematische Zeitschrift}, \textbf{290} (2018), 1415-1455.

\bibitem[NV24]{NV24}
\newblock A. Naber and D. Valtorta,
\newblock Energy identity for stationary harmonic maps,
\newblock \emph{arXiv preprint arXiv:2401.02242}, (2024).

\bibitem[Pac93]{Pac93}
\newblock F. Pacard, 
\newblock Partial regularity for weak solutions of a nonlinear elliptic equation,
\newblock \emph{manuscripta mathematica}, \textbf{79} (1993), 161-172. 

\bibitem[Pac94]{Pac94}
\newblock F. Pacard,
\newblock Convergence and partial regularity for weak solutions of some nonlinear elliptic equation: the supercritical case
\newblock \emph{Annales de l'Institut Henri Poincar\`{e} C}, \textbf{11} (1994), 537-551.

\bibitem[Poo91]{Poo91}
\newblock C. Poon,
\newblock Some new harmonic maps from $ \mathbb{B}^3 $ to $ \mathbb{S}^2 $,
\newblock \emph{Journal of Differential Geometry}, \textbf{34} (1991), 165-168.

\bibitem[Pre87]{Pre87}
\newblock D. Preiss,
\newblock Geometry of measures in $ R^n $: distribution, rectifiability, and densities,
\newblock \emph{Annals of Mathematics. Second Series}, \textbf{125} (1987), 537-643.

\bibitem[Pri83]{Pri83}
\newblock P. Price
\newblock A monotonicity formula for Yang-Mills fields,
\newblock \emph{manuscripta mathematica} \textbf{43} (1983), 131-166.

\bibitem[SU82]{SU82}
\newblock R. Schoen and K. Uhlenbeck,
\newblock A regularity theory for harmonic maps,
\newblock \emph{Journal of Differential Geometry}, \textbf{17} (1982), 307-335.

\bibitem[Sim83]{Sim83}
\newblock L. Simon, 
\newblock Lectures on geometric measure theory,
\newblock  Proceedings of the Centre for
Mathematical Analysis, Australian National University, 3. Australian National University, Centre for Mathematical Analysis, Canberra, 1983.

\bibitem[Sin18]{Sin18}
\newblock Z. Sinaei,
\newblock Convex functionals and the stratification of the singular set of their stationary points,
\newblock \emph{Advances in Mathematics}, \textbf{338} (2018), 502-548.

\bibitem[Str00]{Str00}
\newblock M. Struwe, 
\newblock Variational methods,
\newblock Berlin: Springer-verlag, 2000.

\bibitem[Tia00]{Tia00}
\newblock G. Tian,
\newblock Gauge theory and calibrated geometry, I
\newblock \emph{Annals of mathematics}, \textbf{151} (2000), 193-268.

\bibitem[TT04]{TT04}
\newblock T. Tao and G. Tian,
\newblock A singularity removal theorem for Yang-Mills fields in higher dimensions
\newblock \emph{Journal of the American Mathematical Society}, \textbf{17} (2004), 557-593.


\bibitem[Ved21]{Ved21}
\newblock M. Vedovato,
\newblock Quantitative regularity for p-minimizing maps through a Reifenberg theorem
\newblock \emph{The Journal of Geometric Analysis}, \textbf{31} (2021): 8271-8317.

\bibitem[Wan12]{Wan12}
\newblock K. Wang,
\newblock Partial regularity of stable solutions to the supercritical equations and its applications, \newblock \emph{Nonlinear Analysis: Theory, Methods \& Applications}, \textbf{75} (2012), 5238-5260.

\bibitem[Wan21]{Wan21}
\newblock Y. Wang,
\newblock Quantitative stratification of stationary connections,
\newblock \emph{Journal f\"{u}r die reine und angewandte Mathematik (Crelles Journal)}, \textbf{775} (2021), 39-69.

\bibitem[WW15]{WW15}
\newblock K. Wang and J. Wei,
\newblock Analysis of blow-up locus and existence of weak solutions for nonlinear supercritical problems,
\newblock \emph{International Mathematics Research Notices}, \textbf{2015} (2015), 2634-2670.

\bibitem[WW21]{WW21}
\newblock K. Wang and J. Wei,
\newblock Refined blowup analysis and nonexistence of type II blowups for an energy critical nonlinear heat equation,
\emph{arXiv preprint arXiv:2101.07186}, (2021).

\end{thebibliography}
\end{document}